\documentclass{elsarticle}

\usepackage[T1]{fontenc}	
\usepackage{graphicx}
\usepackage{microtype}
\usepackage{amsmath, amsthm, amssymb,centernot}
\usepackage{mathrsfs}
\usepackage{pdfsync}
\usepackage[a4paper,top=2.7cm,bottom=2.7cm,left=2.7cm,right=2.7cm, bindingoffset=0mm]{geometry}

\usepackage[usenames,dvipsnames]{xcolor}
\usepackage[inline,shortlabels]{enumitem}
\usepackage[hyphens]{url}									
\usepackage{dsfont}									

\usepackage[pdftex,bookmarks,
						colorlinks,				
						breaklinks,unicode,
						citecolor=OliveGreen,
						linkcolor=Maroon
						]{hyperref} 

\usepackage{xypic}
\usepackage{mathtools}
\usepackage{scrextend}						
\usepackage{xparse}
\usepackage[normalem]{ulem}

\ExplSyntaxOn
\NewDocumentCommand{\makeabbrev}{mmm}
 {
  \yoruk_makeabbrev:nnn { #1 } { #2 } { #3 }
 }

\cs_new_protected:Npn \yoruk_makeabbrev:nnn #1 #2 #3
 {
  \clist_map_inline:nn { #3 }
   {
    \cs_new_protected:cpn { #2 } { #1 { ##1 } }
   }
 }
 \ExplSyntaxOff

\makeabbrev{\textbf}{tbf#1}{a,b,c,d,e,f,g,h,i,j,k,l,m,n,o,p,q,r,s,t,u,v,w,x,y,z,A,B,C,D,E,F,G,H,I,J,K,L,M,N,O,P,Q,R,S,T,U,V,W,X,Y,Z}

\makeabbrev{\textbf}{bf#1}{a,b,c,d,e,f,g,h,i,j,k,l,m,n,o,p,q,r,s,t,u,v,w,x,y,z,A,B,C,D,E,F,G,H,I,J,K,L,M,N,O,P,Q,R,S,T,U,V,W,X,Y,Z}

\makeabbrev{\textsf}{tsf#1}{a,b,c,d,e,f,g,h,i,j,k,l,m,n,o,p,q,r,s,t,u,v,w,x,y,z,A,B,C,D,E,F,G,H,I,J,K,L,M,N,O,P,Q,R,S,T,U,V,W,X,Y,Z}

\makeabbrev{\mathsf}{mss#1}{a,b,c,d,e,f,g,h,i,j,k,l,m,n,o,p,q,r,s,t,u,v,w,x,y,z,A,B,C,D,E,F,G,H,I,J,K,L,M,N,O,P,Q,R,S,T,U,V,W,X,Y,Z}

\makeabbrev{\mathfrak}{mf#1}{a,b,c,d,e,f,g,h,i,j,k,l,m,n,o,p,q,r,s,t,u,v,w,x,y,z,A,B,C,D,E,F,G,H,I,J,K,L,M,N,O,P,Q,R,S,T,U,V,W,X,Y,Z}

\makeabbrev{\mathrm}{mrm#1}{a,b,c,d,e,f,g,h,i,j,k,l,m,n,o,p,q,r,s,t,u,v,w,x,y,z,A,B,C,D,E,F,G,H,I,J,K,L,M,N,O,P,Q,R,S,T,U,V,W,X,Y,Z}

\makeabbrev{\mathbf}{mbf#1}{a,b,c,d,e,f,g,h,i,j,k,l,m,n,o,p,q,r,s,t,u,v,w,x,y,z,A,B,C,D,E,F,G,H,I,J,K,L,M,N,O,P,Q,R,S,T,U,V,W,X,Y,Z}

\makeabbrev{\mathcal}{mc#1}{A,B,C,D,E,F,G,H,I,J,K,L,M,N,O,P,Q,R,S,T,U,V,W,X,Y,Z}

\makeabbrev{\mathbb}{mbb#1}{A,B,C,D,E,F,G,H,I,J,K,L,M,N,O,P,Q,R,S,T,U,V,W,X,Y,Z}

\makeabbrev{\mathscr}{ms#1}{A,B,C,D,E,F,G,H,I,J,K,L,M,N,O,P,Q,R,S,T,U,V,W,X,Y,Z}

\makeabbrev{\mathrm}{#1}{
Id,id,ran,rk,diag,stab,ann,conv,pr,ev,tr,End,Hom,sgn,im,op,can,fin,ext,red,tot,
rot,usc,lsc,Lip,LocLip,lip,bSymLip,osc,AC,loc,spec,coz,z,
supp,Opt,Adm,Cpl,Geo,GeoOpt,GeoAdm,GeoCpl,GeoSel,reg,
bd,co,Ric,Exp,dExp,dist,seg,Seg,cut,fcut,Cut,SDiff,Iso,Isom,diam,cl,Homeo,Diff,Der,vol,dvol,inj,relint, Graph, sub,
var,law,Var,Poi,Gam,pa,so,iso,fs,inv,pqi,mix,
TestF,
}

\makeabbrev{\mathsf}{#1}{DP,CD,BE,MCP,Ent,wMTW,MTW,RCD,EVI,Irr,IH,SC,wFe,UP}

\makeabbrev{\mathsc}{#1}{mmaf,cg}

\DeclareMathOperator{\bLip}{Lip_{\it b}}
\DeclareMathOperator{\Lipu}{Lip^1}

\DeclareMathOperator{\bLipu}{Lip_{\it b}^1}
\newcommand{\T}{\tau} 
\newcommand{\A}{\Sigma} 
\newcommand{\Bo}[1]{\msB_{#1}} 
\newcommand{\Ne}[1]{\msN_{#1}} 
\newcommand{\eps}{\varepsilon}

\renewcommand{\complement}{\mathrm{c}}
\newcommand{\mathsc}[1]{\text{\textsc{#1}}}
\newcommand{\emparg}{{\,\cdot\,}}
\newcommand{\slo}[2][]{\abs{\mathrm{D}#2}_{#1}}
\newcommand{\Ch}[1][]{\mathsf{Ch}_{#1}}

\newcommand{\as}[1]{\quad #1\text{-a.e.}}
\newcommand{\qe}[1]{\;{#1}\text{-q.e.}}
\newcommand{\Leb}{{\mathrm{Leb}}}
\newcommand{\dom}{\mcF}
\newcommand{\domb}{\mcF_b}
\newcommand{\domdom}{\mcF_{\mathrm{dom}}}
\newcommand{\domain}[1]{\msD({#1})}
\newcommand{\dotloc}[1]{{{#1}^\bullet_{\loc}}}

\newcommand{\dotlocb}[1]{{{#1}^\bullet_{\loc,b}}}

\newcommand{\domloc}{\mcF_\loc}
\newcommand{\domext}{\mcF_e}
\newcommand{\domextb}{\mcF_{eb}}
\newcommand{\intE}{\interior_\mcE}
\newcommand{\clE}{\cl_\mcE}
\newcommand{\DzLocB}[1]{\mbbL^{#1}_{\loc,b}}

\newcommand{\DzB}[1]{\mbbL^{#1}_b}
\newcommand{\DzLoc}[1]{\mbbL^{#1}_{\loc}}
\newcommand{\Dz}[1]{\mbbL^{#1}}
\DeclareMathOperator{\eqdef}{\coloneqq}

\let\epsilon\varepsilon
\newcommand{\longrar}{\longrightarrow}
\newcommand{\rar}{\rightarrow}
\newcommand{\nlim}{\lim_{n}}								
\newcommand{\mlim}{\lim_{m}}
\newcommand{\klim}{\lim_{k }}

\newcommand{\nliminf}{\liminf_{n }}
\newcommand{\mliminf}{\liminf_{m }}

\newcommand{\diff}{\mathop{}\!\mathrm{d}}						
\newcommand{\ttabs}[1]{\lvert#1\rvert}	
\newcommand{\tabs}[1]{\big\lvert#1\big\rvert}	
\newcommand{\abs}[1]{\left\lvert#1\right\rvert}						

\newcommand{\norm}[1]{\left\lVert#1\right\rVert}					
\newcommand{\set}[1]{\left\{#1\right\}}							
\newcommand{\tset}[1]{\big\{#1\big\}}							
\newcommand{\ttset}[1]{\{#1\}}									
\newcommand{\tonde}[1]{\left(#1\right)}							
\newcommand{\ttonde}[1]{\big({#1}\big)}

\newcommand{\Li}[2][]{\mathrm{L}_{#1}(#2)}						
\newcommand{\rep}[1]{\hat{#1}}								
\newcommand{\reptwo}[1]{\tilde{#1}}							
\newcommand{\scalar}[2]{\left\langle #1 \,\middle |\, #2\right\rangle}		
\newcommand{\sym}[1]{{\scriptscriptstyle{(#1)}}}
\newcommand{\tym}[1]{{\scriptscriptstyle{\times #1}}}
\newcommand{\otym}[1]{{\scriptscriptstyle{\otimes #1}}}
\DeclareSymbolFont{symbolsC}{U}{pxsyc}{m}{n}
\SetSymbolFont{symbolsC}{bold}{U}{pxsyc}{bx}{n}
\DeclareFontSubstitution{U}{pxsyc}{m}{n}
\DeclareMathSymbol{\medcirc}{\mathbin}{symbolsC}{7}
\DeclareSymbolFont{symbolsZ}{OMS}{pxsy}{m}{n}
\SetSymbolFont{symbolsZ}{bold}{OMS}{pxsy}{bx}{n}
\DeclareFontSubstitution{OMS}{pxsy}{m}{n}
\DeclareMathOperator{\interior}{int}								

\newcommand{\seq}[1]{\tonde{#1}}								
\newcommand{\tseq}[1]{{\big(#1\big)}}
\newcommand{\ttseq}[1]{(#1)}
\DeclareMathOperator{\Cont}{\mcC}					
\newcommand{\Cb}{\mcC_b}									
\newcommand{\Cc}{\mcC_c}									
\newcommand{\Cz}{\mcC_0}									
\newcommand{\Cbinfty}{{\mcC_{b}^{\infty}}}
\newcommand{\Czinfty}{{\mcC_0^{\infty}}}
\newcommand{\Mbp}{\mfM^+_b}
\newcommand{\Msp}{\mfM^+_\sigma}
\newcommand{\Ms}{\mfM^\pm_\sigma}
\newcommand{\Mb}{\mfM^\pm_b}
\newcommand{\MbpR}{\mfM^+_{b\mathsc{r}}}

\newcommand{\MbR}{\mfM^\pm_{b\mathsc{r}}}
\newcommand{\pfwd}{\sharp}
\DeclareMathOperator*{\esssup}{esssup}
\DeclareMathOperator*{\essinf}{essinf}
\DeclareMathOperator{\car}{\mathds 1}
\DeclareMathOperator{\emp}{\varnothing} 
\newcommand{\N}{{\mathbb N}}
\newcommand{\R}{{\mathbb R}}

\newcommand{\restr}{\big\lvert}
\newcommand{\mrestr}{\mathbin{\vrule height 1.6ex depth 0pt width 0.13ex\vrule height 0.13ex depth 0pt width 1.3ex}}
\newcommand{\iref}[1]{\ref{#1}}
\newcommand{\comm}{\,\,\mathrm{,}\;\,}
\newcommand{\semicolon}{\,\,\mathrm{;}\;\,}
\newcommand{\fstop}{\,\,\mathrm{.}}
\newcommand{\cdc}{\Gamma}
\DeclareMathOperator{\zero}{{\mathbf 0}}

\newcommand{\blue}[1]{#1}
\DeclareMathOperator{\inter}{int}
\let\temp\phi
\let\phi\varphi
\let\varphi\temp
\newcommand{\hr}[1]{{\bar\mssd}_{#1}} 						

\newcommand{\sqf}[1]{\boldsymbol\Gamma_{#1}}
\newcommand{\sq}[1]{\mu_{#1}}
\newcommand{\Rad}[2]{\mathsf{Rad}_{#1,#2}}

\newcommand{\dRad}[2]{{#1}\text{-}\mathsf{Rad}_{#2}}
\newcommand{\ScL}[3]{\mathsf{ScL}_{#1,#2,#3}}
\newcommand{\dSL}[2]{{#1}\text{-}\mathsf{SL}_{#2}}
\newcommand{\SL}[2]{\mathsf{SL}_{#1,#2}}
\newcommand{\cSL}[3]{\mathsf{cSL}_{#1,#2,#3}}
\newcommand{\Loc}[1]{\mathsf{Loc}_{#1}}

\allowdisplaybreaks

\numberwithin{equation}{section}

\theoremstyle{plain}
\newtheorem{theorem}{Theorem}[section]
\newtheorem*{theorem*}{Theorem}
\newtheorem*{mthm*}{Main Theorem}
\newtheorem{proposition}[theorem]{Proposition}
\newtheorem*{proposition*}{Proposition}
\newtheorem{lemma}[theorem]{Lemma}
\newtheorem{corollary}[theorem]{Corollary}

\theoremstyle{definition}
\newtheorem{definition}[theorem]{Definition}
\newtheorem*{defs*}{Definition}

\theoremstyle{remark}
\newtheorem{remark}[theorem]{Remark}
\newtheorem{example}[theorem]{Example}
\newtheorem*{question*}{Question}

\renewcommand{\paragraph}[1]{\medskip{\noindent\emph{#1}.\quad}}

\makeatletter   
\def\@fnsymbol#1{\ensuremath{\ifcase#1\or *\or \mathsection \or  \mathparagraph \or \dagger\or
    \ddagger \or \|\or **\or \dagger\dagger
   \or \ddagger\ddagger \else\@ctrerr\fi}}
 \makeatother 

\makeatletter
\newcommand\thankssymb[1]{\textsuperscript{\@fnsymbol{#1}}}
\makeatother

\makeatletter
\def\@settitle{\begin{center}%
  \baselineskip14\p@\relax
    \normalfont\LARGE
  \@title
  \end{center}%
}
\makeatother

\begin{document}
\title{Rademacher-type Theorems and Sobolev-to-Lipschitz Properties\\for Strongly Local Dirichlet Spaces\tnoteref{t1}}
\tnotetext[t1]{
The authors are grateful to Professor Kazuhiro Kuwae for kindly providing a copy of~\cite{Kuw96}. They are also grateful to Dr.~Bang-Xian Han for helpful discussions on the Sobolev-to-Lipschitz property on metric measure spaces.
They wish to express their deepest gratitude to an anonymous Reviewer, whose punctual remarks and comments greatly improved the accessibility and overall quality of the initial submission.
}

\author[1]{Lorenzo Dello Schiavo\fnref{fn1}}
\ead{lorenzo.delloschiavo@ist.ac.at}

\author[2]{Kohei Suzuki\corref{cor1}\fnref{fn2}}
\ead{ksuzuki@math.uni-bielefeld.de}

\cortext[cor1]{Corresponding author}
\fntext[fn1]{This work was completed while L.D.S.\ was a member of the Institut f\"ur Angewandte Mathematik of the University of Bonn. He acknowledges funding of his position at that time by the Collaborative Research Center 1060 at the University of Bonn.
He also acknowledges funding of his current position by the Austrian Science Fund (FWF) grant F65, and by the European Research Council (ERC, grant No.~716117, awarded to Prof.\ Dr.~Jan Maas).
}
\fntext[fn2]{
K.S.~gratefully acknowledges funding by: the JSPS Overseas Research Fellowships, Grant Nr. 290142; World Premier International Research Center Initiative (WPI), MEXT, Japan; and JSPS Grant-in-Aid for Scientific Research on Innovative Areas ``Discrete Geometric Analysis for Materials Design'', Grant Number 17H06465.
}

\affiliation[1]{organization={IST Austria}, addressline={Am Campus 1}, postcode={3400}, city={Klosterneuburg}, country={Austria}}

\affiliation[2]{organization={Fakult\"at f\"ur Mathematik, Universit\"at Bielefeld}, postcode={D-33501}, city={Bielefeld}, country={Germany}}

\begin{abstract}
We extensively discuss the Rademacher and Sobolev-to-Lipschitz properties for generalized intrinsic distances on strongly local Dirichlet spaces possibly without square field operator.
We present many non-smooth and infinite-dimensional examples.

As an application, we prove the integral Varadhan short-time asymptotic with respect to a given distance function for a large class of strongly local Dirichlet forms.
\end{abstract}


\maketitle

\tableofcontents

\section{Introduction and Main Results}
\subsection{Introduction}
The study of heat kernels has provided a number of fruitful results connecting the fields of probability, analysis, and geometry. Among others, one important result in the theory is Varadhan's short-time asymptotic formula:
\begin{align} \label{eq: VAR}
\lim_{t \to 0} \ttonde{-2t\log \mssp_t(x,y)}=\mssd(x,y)^2\comm
\end{align}
connecting the \emph{geometric} information -- the distance function -- on the right-hand side, to the \mbox{(\emph{stochastic-})} \emph{analytic} information -- the short-time asymptotic of the heat kernel and the Brownian motion -- on the left-hand side.

Formula \eqref{eq: VAR} and its variants have been investigated in various settings:
complete connected Riemannian manifolds~\cite{Var67}, Lipschitz manifolds~\cite{Nor97}, degenerate diffusions on Euclidean spaces~\cite{CarKusStr87}, sub-Riemannian manifolds~\cite{Lea87a,Lea87b}, 
and fractals~\cite{Kum97,Kaj11}.
All these spaces are \emph{finite-dimensional}, locally compact complete separable metric space equipped with some canonical regular Dirichlet forms as energy structures. 

In the \emph{infinite-dimensional} setting, all standard examples ---~Wiener spaces, path/loop spaces, configuration spaces, etc.~--- are typically not locally compact, possibly even non-metrizable, and the corresponding Dirichlet spaces are therefore only quasi-regular.
This fact introduces a number of technical difficulties for obtaining and even in phrasing the short-time asymptotics~\eqref{eq: VAR}.
For instance, the heat kernel $\mssp_t(x,\diff y)$ is rarely absolutely continuous with respect to the reference measure, say $\mssm$, which does not allow us to treat the density~$\mssp_t(x,y)$ in \eqref{eq: VAR}.
Furthermore, any canonical distance function~$\mssd$ is typically an extended distance ---~i.e.,~$\mssd(x,y)=+\infty$ on sets of positive $\mssm$-measure~---, $\mssd$ is not a continuous function, and metric $\mssd$-balls are $\mssm$-negligible sets, all of which cause various technical difficulties such as measurability issues and counter-intuitive phenomena compared with the finite-dimensional setting. 

In spite of these difficulties, Varadhan-type short-time formulas have been successfully established on the Wiener space and on path/loop groups by Fang~\cite{Fan94}, Fang--Zhang~\cite{FanZha99}, Aida--Kawabi~\cite{AidKaw01}, Aida--Zhang~\cite{AidZha02}, and Hino--Ram\'irez~\cite{HinRam03}.
In view of the difficulties listed above, in all of these papers Formula~\eqref{eq: VAR} is rather phrased in an integral way, as follows.
For open sets~$A$ and $B$, let~$\mssp_t(A,B)\eqdef \int_A P_t\car_B \diff\mssm$ with $\seq{P_t}_{t\geq 0}$ being the heat semigroup, and $\mssd(A,B)$ be the distance between sets induced by some canonical geometric distance $\mssd$ --- e.g, the Cameron-Martin distance in the case of the Wiener space.
Then
\begin{align} \label{eq: VAR-2}
\lim_{t \to 0} \ttonde{-2t\log P_t(A,B)}=\mssd(A,B)^2 \fstop
\end{align}

In fact, according to Hino--Ramirez \cite{HinRam03} and Ariyoshi--Hino \cite{AriHin05}, Formula \eqref{eq: VAR-2} holds for any strongly local symmetric Dirichlet space, provided one replaces the set-distance~$\mssd$ with a suitable \emph{maximal function} $\hr{}$.
However, this very general statement does not tell us whether the function~$\hr{}$ comes from some geometric distance function on the underlying space, nor even whether~$\hr{}$ is in any way related to the intrinsic distance associated with the Dirichlet space under consideration  (See Rmk.s~\ref{r:IntrinsicVsHR} and~\ref{r:dOpen}).
In other words, if we only rely on this general theorem, $\hr{}$ does not \emph{a priori} convey any geometric information about the underlying space. 
Indeed, in typical fractal spaces, heat kernels have a sub-Gaussian asymptotics rather than a Gaussian one, and, in this case, it turns out that $\hr{}\equiv 0$, i.e.\ the right-hand side in \eqref{eq: VAR-2} identically vanishes.
Therefore, in each of the above mentioned infinite-dimensional examples, various techniques are required based on stochastic analysis to identify~$\hr{}$ with a geometric distance function~$\mssd$.
 
 \medskip
 
In this paper, we present a general framework to show the Varadhan short-time asymptotics~\eqref{eq: VAR-2} for a given geometric distance~$\mssd$ on the underlying space, that is, we identify~$\hr{}$ with~$\mssd$.
Our strategy is based on (infinite-dimensional) geometric analysis rather than on stochastic analysis. 
We focus on two fundamental properties connecting Dirichlet-forms theory and metric geometry: the \emph{Rademacher property}~$(\mathsf{Rad})$, and the \emph{Sobolev-to-Lipschitz property}~$(\mathsf{SL})$.
We extensively investigate $(\mathsf{Rad})$ for forms with or without square-field operators (Theorems~\ref{thm: RAD} and~\ref{thm: LENG}).
After studying relations among several inequivalent definitions of $(\mathsf{Rad})$ and~$(\mathsf{SL})$, we finally show that, under $(\mathsf{Rad})$ and~$(\mathsf{SL})$, the maximal function~$\hr{}$ can be identified with a given geometric distance function (Theorem \ref{thm: VSA}). 

As a consequence of our approach, we are able to show~\eqref{eq: VAR-2} for a geometric distance function~$\mssd$ in a number of non-smooth and infinite-dimensional examples not covered by the existing literature (see the \emph{Examples} after Theorem~\ref{thm: VSA}). 
In addition, our theorem provides a new result even on the Wiener space (see the paragraph after Theorem \ref{thm: VSA}).

\medskip

In the process of developing the theory on quasi-regular Dirichlet spaces, we discuss many results previously available only in the regular case.
In particular, we extensively discuss localization arguments by nests,~\S\ref{ss:BroadLoc}. Furthermore, in order to deal with the intrinsic distance~$\mssd_\mu$ with respect to\ any $\mcE$-smooth measure~$\mu$, we introduce the notion of \emph{$\mcE$-moderance} of measures, \S\ref{sss:Moderance}, and discuss its relations to $\mcE$-smoothness, \S\ref{sss:Smoothness}, and to the \emph{$\mcE$-dominance} of measures, \S\ref{sss:Dominance}, introduced by Hino~\cite[Def. 2.1]{Hin09} for regular forms.

We remark that $(\mathsf{Rad})$ and~$(\mathsf{SL})$ are significant not only for the Varadhan short-time asymptotics~\eqref{eq: VAR-2}. 
Once~$(\mathsf{Rad})$ and~$(\mathsf{SL})$ are established, they have various applications, including: properties of the (possibly: \emph{extended}) metric space~$(X,\mssd_\mu)$, as the length property (see~\S\ref{ss:RadLength} and Theorem \ref{thm: LENG}), or completeness (see~\S\ref{ss:StoLCompleteness});
properties of function spaces on~$X$, as the density of some families of functions in various Sobolev-type spaces on~$X$, the quasi-regularity of~$(\mcE,\dom)$ (see~\S\ref{ss:RadQuasiReg}), or the existence of Sobolev cut-off functions (see~\S\ref{ss:Localization});
properties of operators acting on said function spaces, as the regularizing effects of the heat semigroup associated to~$(\mcE,\dom)$.

\subsection{Main Results}
We consider a quasi-regular strongly local Dirichlet form~$(\mcE,\dom)$ on~$L^2(X,\mssm)$ for a possibly non-metrizable topological Luzin space~$(X,\T)$ endowed with a $\sigma$-finite not necessarily Radon topological measure~$\mssm$.
The need for such generality in the choice of a topology is not recondite, and arises already when considering Dirichlet spaces over Banach spaces endowed with their weak topology.
Allowing for non-Radon measures further includes standard examples in the theory of quasi-regular Dirichlet forms, see e.g.~\cite{AlbMa91,RoeSch95}.

We define a class of `$\mcE$-moderate' Borel measures on~$X$, Dfn.~\ref{d:Moderance}, including (non-Radon) $\mcE$-smooth measures, Dfn.~\ref{d:Smooth}.
For any such measure~$\mu$, the broad local domain of functions with $\mu$-bounded energy is
\begin{align*}
\DzLoc{\mu}\eqdef \set{f\in\dotloc{\dom} : \sq{f}\leq \mu} \fstop
\end{align*}
Here, for~$f$ in the broad local domain~$\dotloc{\dom}$ (in the sense of~\cite{Kuw98}, see~\S\ref{ss:BroadLoc}),~$\sq{f}$ is the energy measure of~$f$.
Finally, let~$\mssd\colon X^\tym{2}\rar [0,\infty]$ be any extended pseudo-metric, possibly unrelated to~$\T$, and set 
\begin{align*}
\Lipu(\mssd)\eqdef& \set{f\colon X\rar \R : \Li[\mssd]{f}\leq 1} \comm
\end{align*}
and
\begin{align*}
\Li[\mssd]{f}\eqdef& \inf\set{L>0: \abs{f(x)-f(y)}\leq L\, \mssd(x,y)\comm x,y\in X}  \fstop
\end{align*}

In order to analyze geometric structures generated by Dirichlet forms, we study relations between the first-oder differential structure on~$X$ arising from the $\mssd$-Lipschitz algebra, and the one arising from the square-field operator (or energy measure) of~$\mcE$.
Precisely, we investigate the Rademacher and the Sobolev-to-Lipschitz properties, as follows. 
Neglecting ---~for the purpose of this introduction~--- measurability details, say that $(\mcE,\mssd,\mu)$ has:
\begin{itemize}
\item[$(\mathsf{Rad})$] the \emph{Rademacher property}, if $\Lipu(\mssd)\subset \DzLoc{\mu}$;
\item[$(\mathsf{SL})$] the \emph{Sobolev-to-Lipschitz property}, if $\DzLoc{\mu}\subset \Lipu(\mssd)$.
\end{itemize}
Spaces satisfying one or both of these properties include: Riemannian manifolds, doubling metric measure spaces carrying a Poincar\'e inequality~\cite{Che99}, $\RCD(K,\infty)$ spaces~\cite{AmbGigSav15}, configuration spaces over smooth manifolds~\cite{AlbKonRoe98,RoeSch99}, Hilbert and Wiener spaces~\cite{BogMay96,EncStr93}, spaces of probability measures~\cite{LzDS19b,vReStu09}, and many others.
In addition, the combination of~$(\mathsf{Rad})$ and~$(\mathsf{SL})$ is one main tool in the identification of the right-hand side of~\eqref{eq: VAR-2} for diffusion processes constructed via Dirichlet forms and difficult to study by other means.
This applies in particular to diffusion processes on infinite-dimensional state spaces, such as Hilbert and Wiener spaces~\cite{AidKaw01,Fan94,FanZha99,Zha00}, path and loop groups~\cite{AidZha02,HinRam03}, spaces of probability measures~\cite{LzDS17+,KonvRe18,vReStu09}, etc..
Most of these examples arise from quasi-regular Dirichlet spaces. For this reason, they do not fall within the scope of any systematic treatment and are thus handled with \emph{ad hoc} techniques.

\medskip

We consider several inequivalent formulations of the properties above, Dfn.s~\ref{d:Rad} and~\ref{d:StoL}, suitable to address extended pseudo-distances. These properties are of particular interest in the case when~$\mu$ is $\mcE$-smooth, and~$\mssd=\mssd_\mu$ is the intrinsic distance
\begin{align}\label{eq:Intro:Intrinsic2}
\mssd_\mu(x,y)\eqdef \sup\set{f(x)-f(y) : f\in \DzLoc{\mu} \cap \Cb(\T)}\fstop
\end{align}
Whereas it is common in the literature to restrict the above definition to the case~$\mu=\mssm$, the intrinsic distance~$\mssd_\mu$ for $\mcE$-smooth~$\mu$ is naturally related to ---~however different from~--- the intrinsic distance of the $\mu$-perturbation~\eqref{eq:Perturbed} of~$(\mcE,\dom)$, see Example~\ref{ese:Perturbed}.
It provides a non-trivial intrinsic distance in the case when energy measures are singular with respect to\ the reference measure, e.g.\ on fractal spaces, or for the Dirichlet form~\cite{GarRhoVar14} of Liouville Brownian Motion.

\paragraph{On the Rademacher property} Firstly, we provide sufficient conditions for this property to hold.
\begin{theorem} \label{thm: RAD}
Suppose that~$\mu$ is an $\mcE$-moderate measure. Then, the following assertions hold:
\begin{enumerate}[$(a)$]
\item if $(X,\mssd_\mu)$ is a separable extended metric space, then~$(\mcE,\mssd_\mu,\mu)$ has the Rademacher property, see Theorem~\ref{t:KuwaeProposition};
\item if $\mssd(A,\emparg)\wedge r\in \DzLoc{\mu}$ for all measurable~$A\subset X$, and each~$r>0$, then~$(\mcE,\mssd,\mu)$ has the Rademacher property, see Thm.~\ref{t:Lenz}.
\end{enumerate}
\end{theorem}
\noindent In the strongly local case, the second assertion in the theorem extends Frank--Lenz--Wingert~\cite[Thm.~4.9]{FraLenWin14} to quasi-regular Dirichlet forms on spaces that are possibly not locally compact. Crucially, we also remove all continuity assumptions on~$\mssd$, thus sensibly generalizing both~\cite{FraLenWin14}, and Kuwae's~\cite[Thm.~1.1]{Kuw96}.
In the particular case~$\mssd=\mssd_\mssm$, the first assertion is an extension: of Sturm~\cite[Lem.~1]{Stu94}, removing strong regularity, see Rmk.~\ref{r:StrongRegularity};
and of some implications in Ambrosio--Gigli--Savar\'e~\cite[Thm~3.9]{AmbGigSav15}, removing the assumption on existence of continuous Sobolev cut-offs ---~\eqref{eq:Loc}, see Dfn.~\ref{d:Loc} below or~\cite[Dfn.~3.6(a)]{AmbGigSav15}~---, which we discuss extensively in~\S\ref{ss:Localization}.

\smallskip

Finally, among other generalizations from the regular to the quasi-regular case, we obtain the following characterization of the length property for intrinsic metrics. We say that~$(X,\mssd_\mssm)$ is locally complete if every point in~$X$ has a metrically complete neighborhood.

\begin{theorem}[See~Dfn.~\ref{d:StrictLoc} and Thm.~\ref{t:Stollmann}] \label{thm: LENG} Assume~$(\mcE,\dom)$ is strictly local. If~$(X,\mssd_\mssm)$ is a locally complete metric space, then it is a length space.
\end{theorem}
\noindent This result generalizes the analogous statements of Stollmann~\cite[Thm.~5.2]{Sto10}, for regular forms under local compactness and local completeness of~$(X,\mssd_\mssm)$; and of Ambrosio--Gigli--Savar\'e~\cite[Thm~3.10]{AmbGigSav15} under completeness of~$(X,\mssd_\mssm)$ and~\eqref{eq:Loc}.
For the importance of the locally complete non-complete case, see Remark~\ref{r:Stollmann} and Example~\ref{ese:Stollmann} below.
The Theorem also provides a partial answer to the question raised in~\cite[Rmk.~2.7]{HinRoeTep13}, by removing the local-compactness assumption on the topology.

\paragraph{Examples}
Concerning the metric properties of intrinsic distances, an interesting and simple class of examples for Theorems~\ref{thm: RAD} and~\ref{thm: LENG} is the perburbation by non-integrable Coulomb-type potentials of the standard Dirichlet forms on either Euclidean spaces (Example~\ref{ese:Coulomb1}), or Hilbert spaces endowed with log-concave measures (Example~\ref{ese:Coulomb2}).
We note that ---~strictly speaking: even on finite-dimensional Euclidean spaces~--- such perturbed Dirichlet spaces are not covered by the aforementioned literature, since the reference measure is \emph{neither} Radon \emph{nor} locally finite, and, therefore, the Dirichlet form is not regular (even in the Euclidean case).

\paragraph{On the Sobolev-to-Lipschitz property and the Varadhan short-time asymptotic}
Under various Sobolev-to-Lipschitz- and Rademacher-type assumptions, we compare the point-to-set distance~$\mssd(\emparg, A)$, with~$A\subset X$, to the function~$\hr{\mu, A}$, the $\mssm$-a.e.\ maximal function in~$\DzLoc{\mu}$ satisfying~$\hr{\mu, A}\equiv 0$ $\mssm$-a.e.\ on~$A$.
When~$\mu=\mssm$, maximal functions of this type were considered in Hino--Ram\'irez~\cite{HinRam03} and Ariyoshi--Hino~\cite{AriHin05}.
We expand these considerations to the case of general~$\mu$ in~\S\ref{ss:StoLMaxFunc}.
Maximal functions of this type identify the Varadhan short-time asymptotics for the heat kernel bi-measure~$P_t(A_1, A_2)$~\eqref{eq:PtAB} associated to~$(\mcE,\dom)$, viz.\
\begin{align}\label{eq:Intro:AriHino}
\lim_{t\rar 0} \ttonde{-2t \log P_t(A_1,A_2)}=\tonde{\mssm\text{-}\essinf_{y\in A_2} \hr{\mssm, A_1}(y)}^2 \comm \qquad \mssm A_1, \mssm A_2>0\fstop
\end{align}
Remarkably, their result shows that any strongly local Dirichlet space has Gaussian short-time behavior, thus suggesting ---~at first sight~--- that the Gaussian behavior is canonical.
However,~\eqref{eq:Intro:AriHino} holds without any topology, and therefore the maximal function $\hr{\mssm, A}$, defined only $\mssm$-a.e.\ and in a non-constructive way, is not a geometric object, in the sense that it bears no topological information on the underlying space.
In particular, this is the case for canonical diffusions on fractals, typically displaying \emph{sub}-Gaussian short-time behavior in terms of some geometric distances~$\mssd$, whereas having vanishing maximal functions~$\hr{\mssm,A}\equiv 0$.
We emphasize here that, in general, the maximal function $\hr{\mssm, A}$ does \emph{not} coincide with the point-to-set distance function~$\mssd_\mssm(\emparg, A)$ associated with the intrinsic distance~$\mssd_\mssm$ for each measurable~$A$ (See Rmk.s~\ref{r:IntrinsicVsHR} and~\ref{r:dOpen}).

In light of this observation, the identification of~$\hr{\mssm, A}$ with a given geometric distance~$\mssd(\emparg, A)$ can provide the necessary geometric understanding of the short-time asymptotics of heat kernels and Brownian motions.
The~$(\mathsf{Rad})$ and $(\mathsf{SL})$ properties discussed in the previous paragraphs play a significant role in linking these two objects together:

\begin{theorem}[Thm.s~\ref{t:Full},~\ref{t:UniversallyM} and Cor.s~\ref{c:Special},~\ref{c:AriyoshiHino}] \label{thm: VSA}
Assume that~$(\mcE,\mssd,\mu)$ possesses both~$(\mathsf{Rad})$ and $(\mathsf{SL})$. Then, for each measurable~$A\subset X$ there exists a measurable~$\tilde A\subset X$ with~$\mssm (A\triangle \tilde A)=0$, so that we have~$\hr{\mu, A}=\mssd(\emparg, \tilde A)$.
In particular, if~$\mu=\mssm$, then for arbitrary measurable non-negligible $A_i$ and suitable~$\tilde A_i$ (identified in the proofs) the following Varadhan short-time asymptotics holds:
\begin{align}\label{eq:Intro:Varadhan}
\lim_{t\rar 0} \ttonde{-2t \log P_t(A_1,A_2)}=\mssd(\tilde A_1,\tilde A_2)^2\fstop
\end{align}
\end{theorem}

We note that Theorem~\ref{thm: VSA} is new even in the case of the Wiener space (see Examples~\ref{ese:Wiener1} and~\ref{ese:Wiener2}), since the formulation in Corollary~\ref{c:AriyoshiHino} does not require the sets~$A_1$,~$A_2$ to be open, which is instead the case in Fang~\cite{Fan94}, Fang--Zhang~\cite{FanZha99} and Aida--Kawabi~\cite{AidKaw01}; also cf.\ Aida--Zhang~\cite{AidZha02} and Hino--Ram\'irez~\cite{HinRam03} for the case of path and loop groups.

\paragraph{Examples}
By Theorem \ref{thm: VSA} and its general form Corollary~\ref{c:AriyoshiHino}, several non-smooth and infinite-dimensional spaces turn out to satisfy Varadhan short-time asymptotics~\eqref{eq: VAR-2} for some geometric distance~$\mssd$, namely: $\RCD(K,\infty)$ spaces; $p$-thick geodesic infinitesimally doubling metric measure spaces (see Example~\ref{ese:RCD} and the References therein); configuration spaces over non-smooth spaces (as will be discussed in a forthcoming work by the authors).

Finally, besides the understanding of the short-time asymptotics, we investigate some pathological examples of quasi-regular Dirichlet spaces endowed with extended distances, and the dependence of~$(\mathsf{Rad})$ and~$(\mathsf{SL})$ on their $\sigma$-algebras (e.g.\ on configuration spaces, see Examples~\ref{ese:Config1},~\ref{ese:Config2},~\ref{ese:Config3},~\ref{ese:Config4}).

\paragraph{Notation}
Throughout the paper, let~$a\vee b\eqdef \max\set{a,b}$ and $a\wedge b\eqdef \min\set{a,b}$ for $a,b\in\R$. Analogously let~$\vee_{i=1}^k a_i\eqdef \max\set{a_i: i\leq k}$ and~$\wedge_{i=1}^k a_i\eqdef \min\set{a_i: i\leq k}$
for $\set{a_i}_{i\leq k}\subset \R$.

For a measure~$\mu$ on a measurable space~$(X,\A)$ we denote by~$\mu A$, resp.~$\mu f$, the $\mu$-measure of $A\in\A$, resp.\ the integral with respect to~$\mu$ of a $\A$-measurable function~$f$ (whenever the integral makes sense).

\section{Auxiliary results}\label{s:Preliminaries}
\subsection{Metric and topological spaces}
Let~$X$ be any non-empty set. A function~$\mssd\colon X^\tym{2}\rar [0,\infty]$ is an \emph{extended pseudo-distance} if it is symmetric and satisfying the triangle inequality. Any such~$\mssd$ is: a \emph{pseudo-distance} if it is everywhere finite, i.e.~$\mssd\colon X^{\tym{2}}\rar [0,\infty)$; an \emph{extended distance} if it does not vanish outside the diagonal in~$X^{\tym{2}}$, i.e.~$\mssd(x,y)=0$ iff~$x=y$; a \emph{distance} if it is both finite and non-vanishing outside the diagonal.

Let~$x_0\in X$ and~$r\in [0,\infty]$. We write~$B^\mssd_r(x_0)\eqdef \set{\mssd_{x_0}<r}$. Note that, if~$\mssd$ is an extended pseudo-metric, then both of the inclusions~$\set{x_0}\subset \cap_{r>0} B^\mssd_r(x_0)$ and~$B^\mssd_\infty(x_0)\subset X$ may be strict ones. 
We say that an extended metric space is, \emph{complete}, resp.\ \emph{length}, \emph{geodesic}, if~$B^\mssd_\infty(x)$ is complete, resp.\ length, geodesic for each~$x\in X$.
Finally set
\begin{align*}
\mssd(\emparg, A)\eqdef& \inf_{x\in A} \mssd(\emparg,x) \colon X\longrar [0,\infty] \comm \qquad A\subset X\fstop
\end{align*}

For an extended pseudo-distance~$\mssd$ on~$X$, let~$\T_\mssd$ denote the (possibly \emph{not} Hausdorff) topology on~$X$ induced by the pseudo-distance~$\mssd\wedge 1$.
The topology~$\T_\mssd$ is Hausdorff if and only if~$\mssd$ is an extended distance.
The topology~$\T_\mssd$ is separable if and only if there exists a countable family of points~$\seq{x_n}_n\subset X$ so that~$X=\cup_n B^\mssd_\infty(x_n)$ and~$(B^\mssd_\infty(x_n),\mssd)$ is a separable pseudo-metric space for every~$n\in \N$.

If~$\mssd$ is an extended pseudo-distance, then~$\mssd(\emparg, A)=\mssd(\emparg, \cl_{\mssd} A)$ for every~$A\subset X$, where~$\cl_{\mssd}=\cl_{\T_\mssd}$ always denotes the closure of~$A$ in the topology~$\T_\mssd$.

\paragraph{Lipschitz functions} A function~$\rep f\colon X\rar \R$ is $\mssd$-Lipschitz if there exists a constant~$L>0$ so that
\begin{align}\label{eq:Lipschitz}
\tabs{\rep f(x)-\rep f(y)}\leq L\, \mssd(x,y) \comm \qquad x,y\in X \fstop
\end{align}
The smallest constant~$L$ so that~\eqref{eq:Lipschitz} holds is the (global) \emph{Lipschitz constant of $\rep f$}, denoted by~$\Li[\mssd]{\rep f}$.
For any non-empty~$A\subset X$ we write~$\Lip(A,\mssd)$, resp.~$\bLip(A,\mssd)$ for the family of all finite, resp.\ bounded, $\mssd$-Lipschitz functions on~$A$. For simplicity of notation, further let~$\Lip(\mssd)\eqdef \Lip(X,\mssd)$, resp.\ $\bLip(\mssd)\eqdef \bLip(X,\mssd)$.

The next lemma is an easy adaptation of McShane~\cite{McS34} to extended metric spaces.

\begin{lemma}[Constrained McShane extensions]\label{l:McShane}
Let~$(X,\mssd)$ be an extended metric space. Fix~$A\subset X$,~$A\neq \emp$, and let~$\rep f\colon A\rar \R$ be a function in~$\bLip(A,\mssd)$. Further set
\begin{equation}\label{eq:McShane}
\begin{aligned}
\overline{f}\colon x&\longmapsto \sup_A \rep f\wedge \inf_{a\in A} \ttonde{\rep f(a)+\Li[\mssd]{\rep f}\,\mssd(x,a)} \comm
\\
\underline{f}\colon x&\longmapsto \inf_A \rep f \vee \sup_{a\in A} \ttonde{\rep f(a)-\Li[\mssd]{\rep f}\,\mssd(x,a)} \fstop
\end{aligned}
\end{equation}

Then,
\begin{enumerate}[$(i)$]
\item\label{i:l:McShane:1} $\underline{f}=\rep f=\overline{f}$ on~$A$ and~$\inf_A \rep f\leq \underline{f}\leq \overline{f}\leq \sup_A \rep f$ on~$X$;

\item\label{i:l:McShane:2} $\underline{f}$, $\overline{f}\in \bLip(\mssd)$ with~$\Li[\mssd]{\underline{f}}=\Li[\mssd]{\overline{f}}=\Li[\mssd]{\rep f}$;

\item\label{i:l:McShane:3} $\underline{f}$, resp.~$\overline{f}$, is the minimal, resp.\ maximal, function satisfying~\iref{i:l:McShane:1}-\iref{i:l:McShane:2}, that is, for every~$\rep g\in \bLip(\mssd)$ with~$\inf_A \rep f\leq \rep g\leq \sup_A \rep f$, $\rep g\restr_A=\rep f$ on~$A$ and~$\Li[\mssd]{\rep g}\leq \Li[\mssd]{\rep f}$, it holds that~$\underline{f}\leq \rep g \leq \overline{f}$.
\end{enumerate}
\end{lemma}
\begin{proof}
If~$\rep f$ is constant, then trivially~$\underline{f}=\rep f=\overline{f}$. Thus, we may assume that~$\rep f$ is non-constant, so that~$\osc(\rep f)\eqdef \sup_A \rep f-\inf_A \rep f>0$ and~$\Li[\mssd]{\rep f}>0$.
Up to range translation and rescaling, we may and shall assume with no loss of generality that~$\Li[\mssd]{\rep f}=1$ and~$\inf_A \rep f=0$.

Define a distance on~$X$ by~$\mssd'\eqdef \mssd \wedge \osc(\rep f)$. Since~$\Li[\mssd]{\rep f}=1$, then~$\Li[\mssd]{\rep f}=\Li[\mssd']{\rep f}$. Assertions~\iref{i:l:McShane:1}-\iref{i:l:McShane:3} with~$\mssd'$ in place of~$\mssd$ follow straightforwardly from the properties of the usual McShane extensions on metric spaces, see e.g.~\cite[Rmk.~4.14]{HeiKosShaTys15}.
Thus, it suffices to show that~$\underline{f}'$ and~$\overline{f}'$, defined as in~\eqref{eq:McShane} with~$\mssd'$ in place of~$\mssd$, satisfy in fact~$\underline{f}'=\underline{f}$ and~$\overline{f}'=\overline{f}$.
This last assertion follows since
\begin{align*}
\overline{f}(x)=&\sup_A \rep f \wedge \inf_{a\in A} \ttonde{\rep f(a)+\mssd(x,a)}=\sup_A \rep f \wedge \inf_{a\in A} \ttonde{\rep f(a)\wedge \sup_A \rep f+\mssd(x,a)\wedge \sup_A \rep f}
\\
=&\sup_A \rep f \wedge \inf_{a\in A}\ttonde{\rep f(a)+\mssd'(x,a)}=\overline{f}'(x)\comm
\end{align*}
and analogously for~$\underline{f}$.
\end{proof}

\paragraph{Topological spaces}
A Hausdorff topological space~$(X,\T)$ is:
\begin{enumerate}[$(a)$]
\item\label{i:Top:1} \emph{strongly Lindel\"of} if every open cover of any open set in~$X$ has a countable sub-cover;
\item\label{i:Top:2} a \emph{topological Luzin space} if it is a continuous injective image of a Polish space;
\item\label{i:Top:3} a \emph{metrizable Luzin space} if it is homeomorphic to a Borel subset of a compact metric space.
\end{enumerate}
As a consequence of~\cite[(6) and (5), p.~104]{Sch73}, we have that~\iref{i:Top:3}$\implies$\iref{i:Top:2}$\implies$\iref{i:Top:1}.
Furthermore, every strongly Lindel\"of space is \emph{hereditarily (strongly) Lindel\"of}, i.e., every subspace of~$(X,\T)$ is strongly Lindel\"of if~$(X,\T)$ is so, e.g.~\cite[(2), p.~104]{Sch73}.

\smallskip

Let~$(X,\T)$ be a Hausdorff topological space. A family of pseudo-distances~$\UP$ is a \emph{uniformity} (\emph{of pseudo-distances}) if:
\begin{enumerate*}[$(a)$]
\item it is directed, i.e., $\mssd_1\vee \mssd_2\in\UP$ for every~$\mssd_1,\mssd_2\in \UP$;
and
\item it is order-closed, i.e., $\mssd_2\in \UP$ and~$\mssd_1\leq \mssd_2$ implies~$\mssd_1\in \UP$ for every pseudo-distance~$\mssd_1$ on~$X$.
\end{enumerate*}
A uniformity is: \emph{bounded} if every~$\mssd\in \UP$ is bounded; \emph{Hausdorff} if it separates points.

The next definition is a reformulation of~\cite[Dfn.~4.1]{AmbErbSav16}.

\begin{definition}[Extended metric-topological space]\label{d:AES}
Let~$(X,\T)$ be a Hausdorff topological space. An extended pseudo-distance~$\mssd\colon X^{\tym{2}}\rar [0,\infty]$ is $\T$-\emph{admissible} if there exists a uniformity~$\UP$ of \emph{$\T^\tym{2}$-continuous} pseudo-distances~$\mssd'\colon X^\tym{2}\rar [0,\infty)$, so that
\begin{align}\label{eq:d=supUP}
\mssd=\sup\set{\mssd':\mssd'\in\UP}\fstop
\end{align}
The triple $(X,\T,\mssd)$ is an \emph{extended metric-topological space} if~$\mssd$ is $\T$-admissible, and there exists a uniformity~$\UP$ witnessing the $\T$-admissibility of~$\mssd$, and additionally Hausdorff and generating~$\T$.
\end{definition}

\begin{remark}\label{r:AES} Our definition is equivalent to~\cite[Dfn.~4.1]{AmbErbSav16}. Indeed, let~$\msQ\eqdef \set{\mssd_i}_{i\in I}$ be as in~\cite[Dfn.~4.1]{AmbErbSav16}.
As already noted in~\cite[\S4.1]{AmbErbSav16}, possibly up to enlarging~$I$, we may assume with no loss of generality: that~$\msQ$ is directed, up to taking~$\mssd_i,\mssd_j \leq \mssd_i\vee\mssd_j \in \msQ$; and that every~$\mssd_i\in\msQ$ is bounded, up to taking~$\mssd_i\wedge r\in \msQ$,~$r>0$.
Furthermore, we may assume that~$\msQ$ is order-closed, up to taking its order-closure, and therefore that it is a bounded uniformity, hence we write~$\msQ=\UP_b$.
Thus, an extended distance~$\mssd$ is $\T$-admissible if and only if it satisfies~\cite[Dfn.~4.1(a)]{AmbErbSav16}. If~$\UP$ is additionally generating~$\T$ (see e.g.~\cite[Thm.~1.2]{Pac13}), then it is Hausdorff, since~$\T$ is so, and~$(X,\T,\mssd)$ satisfies~\cite[Dfn.~4.1(b)]{AmbErbSav16} by~\cite[Lem.~4.2]{AmbErbSav16}.
Our definition makes apparent that:
\begin{enumerate*}[$(a)$]
\item there is a surjection between Hausdorff uniformities~$\UP$ on a set~$X$ and extended metric-topological spaces, given by letting~$\T$ be the topology generated by~$\UP$, and by defining~$\mssd$ as in~\eqref{eq:d=supUP};
and that
\item the topology~$\T$ of an extended metric-topological space is completely regular Hausdorff, see e.g.~\cite[Cor.~1.23]{Pac13}.
\end{enumerate*}
\end{remark}

Finally, let us collect here the following definition, which will be of use in \S\ref{ss:RadLength}.

\begin{definition}[Local completeness]\label{d:LocalCompleteness}
Let~$(X,\T)$ be a Hausdorff topological space, and~$\mssd\colon X^\tym{2}\to [0,\infty]$ be an extended pseudo-distance on~$X$.
We say that~$(X,\T,\mssd)$ is ($\T$-)\emph{locally complete} if for every~$x\in X$ there exists a neighborhood~$U_x$ of~$x$ such that~$(U_x, \mssd_{U_x})$ is a complete extended pseudo-metric space when endowed with the restriction~$\mssd_{U_x}$ of~$\mssd$ to~$U_x^\tym{2}$.
\end{definition}

\subsection{Measure spaces}
Let~$(X,\T)$ be a Hausdorff topological space. We denote by~$\Bo{\T}$ the Borel $\sigma$-algebra of~$(X,\T)$. Given a Borel measure~$\mu$ on~$(X,\Bo{\T})$, we denote by~$\Bo{\T}^\mu$ the (Carath\'edory) completion of~$\Bo{\T}$ with respect to~$\mu$.
Given $\sigma$-finite measures~$\mu_0$,~$\mu_1$ on~$(X,\Bo{\T})$, we write~$\mu_0\leq \mu_1$ to indicate that~$\mu_0 A\leq \mu_1 A$ for every~$A\in\Bo{\T}$. 
Every Borel measure on a strongly Lindel\"of space has support,~\cite[p.~148]{MaRoe92}.

Let~$\A$ be a $\sigma$-algebra over~$X$.
We denote by~$\mcL^0(\A)$, resp.~$\mcL^\infty(\A)$, the vector space of all everywhere-defined real-valued, resp.~uniformly bounded, ($\A$-)measurable functions on~$X$.
When~$\mu$ is a measure on~$(X,\A)$, we denote  by~$L^0(\mu)$ the corresponding vector space of $\mu$-classes.
We denote by~$\mcL^2(\mu)$ the vector space of all $\mu$-square-integrable functions in~$\mcL^0(\A)$, by~$L^2(\mu)$ the corresponding space of~$\mu$-classes.
Let the corresponding definition of~$\mcL^p(\mu)$, resp.~$L^p(\mu)$, be given for all~$p\in [1,\infty)$.

As a general rule, we denote measurable functions by either~$\rep f$ or~$\reptwo f$, and classes of measurable functions up to a.e.\ equality simply by~$f$.
When~$\mu$ has full support on~$X$, we drop this distinction for $\T$-continuous functions, simply writing~$f$ for both the class and its unique $\T$-continuous representative.

\paragraph{Measurability and continuity of Lipschitz functions} Let~$(X,\T)$ be a Hausdorff space, $\mssd\colon X^{\tym{2}}\rar [0,\infty]$ be an extended distance on~$X$.  Let~$\rep f\colon X \rar [-\infty,\infty]$ be $\mssd$-Lipschitz with~$\rep f\not\equiv \pm\infty$. In general, $\rep f$ is \emph{neither} everywhere finite, nor $\T$-continuous, nor $\Bo{\T}^\mssm$-measurable, see Example~\ref{ese:LipNonMeas} below. For a given $\sigma$-algebra~$\A$ on~$X$, this motivates to set
\begin{align*}
\Lip(\mssd,\A)\eqdef \Lip(\mssd)\cap \mcL^0(\A)\comm \qquad \bLip(\mssd,\A)\eqdef \bLip(\mssd)\cap \mcL^\infty(\A)\fstop
\end{align*}

\paragraph{Main assumptions} Everywhere in the following, $\mbbX$ is a quadruple~$(X,\T,\A,\mssm)$ so that~$\Bo{\T}\subset \A\subset \Bo{\T}^\mssm$, the reference measure~$\mssm$ is positive $\sigma$-finite on~$(X,\A)$, and one of the following holds:
\begin{enumerate}[$(\mathsc{sp}_1)$]
\item\label{ass:Hausdorff}
$(X,\T)$ is a Hausdorff space;
\item\label{ass:Luzin}
$(X,\T)$ is a topological Luzin space and~$\supp[\mssm]=X$;
\item\label{ass:Polish}
$(X,\T)$ is a second countable locally compact Hausdorff space,~$\mssm$ is Radon and $\supp[\mssm]=X$.
\end{enumerate}

Spaces satisfying~\ref{ass:Luzin} are generally non-metrizable. Natural examples of non-metrizable such spaces that are of interest in Dirichlet forms' theory include infinite-dimensional separable Banach spaces endowed with their weak topology.
If~$\mbbX$ satisfies~\ref{ass:Luzin}, then~$(X,\T)$ is Suslin~\cite[\S{II.1} Dfn.~3, p.~96]{Sch73} and therefore strongly Lindel\"of~\cite[\S{II.1} Prop.~3, p.~104]{Sch73}.
The support of~$f\in L^0(\mssm)$ is defined as~$\supp[f]\eqdef \supp[\ttabs{\rep f}\cdot \mssm]$; it is independent of the $\mssm$-representative~$\rep f$ of~$f$, e.g.~\cite[p.~148]{MaRoe92}.

We conclude this section with the next lemma. A proof is standard.
\begin{lemma}\label{l:AeEquality}
Let~$\mbbX$ be satisfying~\iref{ass:Hausdorff}, and further assume that~$\mssm$ has $\T$-support.
Further let~$U\subset X$ be $\T$-open, and~$f,g\colon X\rar \R$ be $\T$-continuous and agreeing $\mssm$-a.e.\ on~$U$.
Then,~$f=g$ everywhere on~$U$.
\end{lemma}

\subsection{Dirichlet spaces}
Given a bilinear form~$(Q,\domain{Q})$ on a Hilbert space~$H$, we write
\begin{align*}
Q(h)\eqdef Q(h,h)\comm \qquad Q_\alpha(h_0,h_1)\eqdef Q(h_0,h_1)+\alpha\scalar{h_0}{h_1}\comm \alpha>0\fstop
\end{align*}

Let~$\mbbX$ be satisfying Assumption~\ref{ass:Hausdorff}. A \emph{Dirichlet form on~$L^2(\mssm)$} is a non-negative definite densely defined closed symmetric bilinear form~$(\mcE,\dom)$ on~$L^2(\mssm)$ satisfying the Markov property
\begin{align*}
f_0\eqdef 0\vee f \wedge 1\in \dom \qquad \text{and} \qquad \mcE(f_0)\leq \mcE(f)\comm \qquad f\in\dom\fstop
\end{align*}

If not otherwise stated,~$\dom$ is always regarded as a Hilbert space with norm~$\norm{\emparg}_\dom\eqdef\mcE_1(\emparg)^{1/2}=\sqrt{\mcE(\emparg)+\norm{\emparg}_{L^2(\mssm)}^2}$.
A \emph{Dirichlet space} is a pair~$(\mbbX,\mcE)$, where~$\mbbX$ satisfies~\ref{ass:Hausdorff} and~$(\mcE,\dom)$ is a Dirichlet form on~$L^2(\mssm)$. 

\subsubsection{Quasi-notions} For any~$A\in\Bo{\T}$ set~$\dom_A\eqdef \set{u\in \dom: u= 0 \text{~$\mssm$-a.e.~on~} X\setminus A}$.
A sequence $\seq{A_n}_n\subset \Bo{\T}$ is a \emph{Borel $\mcE$-nest} if $\cup_n \dom_{A_n}$ is dense in~$\dom$.
For any~$A\in\Bo{\T}$, let~$(p)$ be a proposition defined with respect to~$A$. We say that~`$(p_A)$ holds' if~$A$ satisfies~$(p)$.
A \emph{$(p)$-$\mcE$-nest} is a Borel nest~$\seq{A_n}$ so that~$(p_{A_n})$ holds for every~$n$. In particular, a \emph{closed $\mcE$-nest}, henceforth simply referred to as an \emph{$\mcE$-nest}, is a Borel $\mcE$-nest consisting of closed sets.

A set~$N\subset X$ is \emph{$\mcE$-polar} if there exists an $\mcE$-nest~$\seq{F_n}_n$ so that~$N\subset X\setminus \cup_n F_n$.
A set~$G\subset X$ is \emph{$\mcE$-quasi-open} if there exists an $\mcE$-nest~$\seq{F_n}_n$ so that~$G\cap F_n$ is relatively open in~$F_n$ for every~$n\in \N$. A set~$F$ is \emph{$\mcE$-quasi-closed} if~$X\setminus F$ is $\mcE$-quasi-open.
Any countable union or finite intersection of $\mcE$-quasi-open sets is $\mcE$-quasi-open; analogously, any countable intersection or finite union of $\mcE$-quasi-closed sets is $\mcE$-quasi-closed;~\cite[Lem.~2.3]{Fug71}.

A property~$(p_x)$ depending on~$x\in X$ holds $\mcE$-\emph{quasi-everywhere} (in short:~$\mcE$-q.e.) if there exists an $\mcE$-polar set~$N$ so that~$(p_x)$ holds for every~$x\in X\setminus N$.
Given sets~$A_0,A_1\subset X$, we write~$A_0\subset A_1$ $\mcE$-q.e.\ if~$\car_{A_0}\leq \car_{A_1}$ $\mcE$-q.e. Let the analogous definition of~$A_0=A_1$ $\mcE$-q.e.\ be given.

A function~$\rep f\in \mcL^0(\A)$ is \emph{$\mcE$-quasi-continuous} if there exists an $\mcE$-nest~$\seq{F_n}_n$ so that~$\rep f\restr_{F_n}$ is continuous for every~$n\in \N$.
Equivalently,~$\reptwo f$ is $\mcE$-quasi-continuous if and only if it is $\mcE$-q.e.\ finite and $\reptwo f^{-1}(U)$ is $\mcE$-quasi-open for every open~$U\subset \R$, see e.g.~\cite[p.~70]{FukOshTak11}.
Whenever~$f\in L^0(\mssm)$ has an $\mcE$-quasi-continuous $\mssm$-version, we denote it by~$\reptwo f\in \mcL^0(\A)$.

\begin{lemma} Let~$(\mbbX,\mcE)$ be a Dirichlet space, and~$G\subset X$ be $\mcE$-quasi-open, resp.~$\mcE$-quasi-closed. Then, there exists~$G'\in\Bo{\T}$ $\mcE$-quasi-open, resp.\ $\mcE$-quasi-closed, and so that $G\triangle G'$ is $\mcE$-polar.
\end{lemma}
\begin{proof}
We show the assertion for $\mcE$-quasi-open~$G$. The case of $\mcE$-quasi-closed~$G$ follows by complementation. Let~$\seq{F_n}_n$ be an $\mcE$-nest witnessing that~$G$ is $\mcE$-quasi-open. That is, for every~$n$ there exists~$G_n\in \T$ so that~$F_n\cap G_n=G\cap F_n$ is relatively open in~$F_n$. Set~$G'\eqdef \cup_n (G_n\cap F_n)\in \Bo{\T}$ and note that~$G'$ is $\mcE$-quasi-open. Furthermore,~$G'\triangle G\subset X\setminus \cup_n F_n$ is $\mcE$-polar.
\end{proof}

Whenever needed, we shall always assume ---~without explicit mention~--- that $\mcE$-quasi-open, resp.\ $\mcE$-quasi-closed, sets are additionally Borel measurable.

\paragraph{Spaces of measures} We write~$\Mbp(\A)$, resp.~$\Msp(\A)$, $\Mb(\A)$, $\Ms(\A)$, for the space of finite, resp.\ $\sigma$-finite, finite signed, extended $\sigma$-finite signed, measures on~$(X,\A)$. A further subscript~`$\mathsc{r}$' indicates (sub-)spaces of Radon measures, e.g.~$\MbpR(\A)$.
We write~$\Ms(\A,\Ne{\mcE})$ for the space of extended $\sigma$-finite signed measures not charging sets in the family~$\Ne{\mcE}$ of $\mcE$-polar Borel subsets of~$X$.
We write~$\lim_\alpha \mu_\alpha=\mu$ to indicate that the net~$\seq{\mu_\alpha}_\alpha\subset \Ms(\A)$ is strongly converging to~$\mu\in\Ms(\A)$, i.e.~$\lim_\alpha \mu_\alpha A=\mu A$ for every~$A\in\A$.

\paragraph{General properties} A Dirichlet space~$(\mbbX,\mcE)$ is \emph{quasi-regular} if each of the following holds:
\begin{enumerate}[$({\mathsc{qr}}_1)$]
\item\label{i:QR:1} there exists an $\mcE$-nest~$\seq{F_n}_n$ consisting of $\T$-compact sets;

\item\label{i:QR:2} there exists a dense subset of~$\dom$ the elements of which all have $\mcE$-quasi-continuous $\mssm$-versions;

\item\label{i:QR:3} there exists an $\mcE$-polar set~$N$ and a countable family~$\seq{u_n}_n$ of functions~$u_n\in\dom$ having $\mcE$-quasi-continuous versions~$\reptwo u_n$ so that~$\seq{\reptwo u_n}_n$ separates points in~$X\setminus N$.
\end{enumerate}
We will occasionally consider Dirichlet spaces satisfying only some properties out of~\iref{i:QR:1}-\iref{i:QR:3}. For meaningful examples of such spaces, see~\cite{RoeSch95}.

Let~$(\mbbX,\mcE)$ be a quasi-regular Dirichlet space,~$\seq{F_n}_n$ be an $\mcE$-nest witnessing its quasi-regularity, and set~$X_0\eqdef \cup_n F_n$, endowed with the trace topology~$\T_0$,~$\sigma$-algebra~$\A_0$, and the restriction~$\mssm_0$ of~$\mssm$ to~$\A_0$.
Then,~$\mbbX_0$ satisfies~\iref{ass:Luzin}, and the space~$L^p(\mssm)$ may be canonically identified with the space~$L^p(\mssm_0)$,~$p\in[0,\infty]$, since~$X\setminus X_0$ is $\mcE$-polar, hence~$\mssm$-negligible.
By letting~$\mcE_0$ denote the image of~$\mcE$ under this identification,~$(\mbbX_0,\mcE_0)$ is a quasi-regular Dirichlet space, and~$\dom_0$ is canonically linearly isometrically isomorphic to~$\dom$. See~\cite[Rmk.~IV.3.2(iii)]{MaRoe92} for the details of this construction.

\begin{remark}
When considering a quasi-regular Dirichlet space~$(\mbbX,\mcE)$, we may and shall therefore assume, with no loss of generality, that~$\mbbX$ satisfies~\iref{ass:Luzin}.
In particular~$(X,\T)$ is separable.
Since~$X_0$ is $\sigma$-compact by definition of~$\seq{F_n}_n$, it is in principle not restrictive ---~in discussing quasi-regular Dirichlet spaces~--- to assume~$\mbbX$ to be additionally $\sigma$-compact.
 However, we refrain from so doing, since this assumption is indeed restrictive when discussing extended distances on such spaces and in particular their completeness. This is the case for infinite-dimensional Banach spaces with their strong topology, which are never $\sigma$-compact.
\end{remark}

A Dirichlet space~$(\mbbX,\mcE)$ with~$\mbbX$ satisfying~\iref{ass:Luzin} is
\begin{itemize}
\item \emph{local} if~$\mcE(f,g)=0$ for every~$f,g\in\dom$ with~$\supp[f]$,~$\supp[g]$ compact,~$\supp[f]\cap\supp[g]=\emp$;

\item \emph{strongly local} if~$\mcE(f,g)=0$ for every~$f,g\in\dom$ with~$\supp[f]$,~$\supp[g]$  compact and~$f$ constant on a neighborhood of~$\supp[g]$;

\item \emph{regular} if~$\mbbX$ satisfies~\ref{ass:Polish}, and~$\Cz(\T)\cap \dom$ is both dense in~$\dom$ and dense in~$\Cz(\T)$.
Here and henceforth,~$\Cz(\T)$ denotes the space of $\T$-continuous functions on~$X$ vanishing at infinity.
\end{itemize}

\paragraph{Quasi-homeomorphism}
Let~$(\mbbX,\mcE)$ be a Dirichlet space, and~$(X^\sharp,\T^\sharp)$ be a Hausdorff space.
Further let~$N\in \Ne{\mcE}$ and set~$X_0\eqdef X\setminus N$, endowed with the non-relabeled subspace topology and $\sigma$-algebra inherited from~$X$. Suppose~$j\colon X_0\rar X^\sharp$ is $\Bo{\T}/\Bo{\T^\sharp}$-measurable and define a measure~$\mssm^\sharp$ on~$(X^\sharp,\Bo{\T^\sharp})$ by~$\mssm^\sharp\eqdef j_\pfwd \mssm$. Here, and everywhere in the following, we denote by~$j_\sharp \mssm$ the push-forward measure of~$\mssm$ via the map~$j$.
Note that~$\mssm^\sharp$ is positive $\sigma$-finite, which justifies the notation~$\mbbX^\sharp\eqdef (X^\sharp, \T^\sharp, \Bo{\T^\sharp},\mssm^\sharp)$, satisfying~\ref{ass:Hausdorff}.
Since~$N\in\Ne{\mcE}$, it is in particular $\mssm$-negligible, and thus
\begin{align*}
j^*\colon L^2(\mssm^\sharp)\rar L^2(\mssm)\comm f^\sharp\mapsto f^\sharp\circ j
\end{align*}
is well-defined and and an isometry. The \emph{$j$-image of~$(\mcE,\dom)$} is the quadratic form on~$L^2(\mssm^\sharp)$ defined by
\begin{align}\label{eq:Quasihomeo}
\dom^j\eqdef \set{f^\sharp\in L^2(\mssm^\sharp): j^*f^\sharp\in \dom}\comm \qquad \mcE^j(f^\sharp, g^\sharp)\eqdef \mcE(j^*f^\sharp, j^*g^\sharp) \comm f^\sharp,g^\sharp \in \dom^j \fstop
\end{align}
It is densely defined, thus a Dirichlet form, if~$j^*$ is surjective.

Two Dirichlet spaces~$(\mbbX,\mcE)$ and~$(\mbbX^\sharp,\mcE^\sharp)$ are \emph{quasi-homeomorphic}~\cite[Dfn.~3.1]{CheMaRoe94} if there exists an $\mcE$-nest~$\seq{F_n}_n$ in~$X$, resp.~an $\mcE^\sharp$-nest $\tseq{F_n^\sharp}_n$ in~$X^\sharp$, and a map~$j\colon \cup_n F_n\rar \cup_n F^\sharp_n$ so that
\begin{enumerate}[$(\mathsc{qh}_1)$]
\item\label{i:QH:1} for each~$n$, the restriction~$j\restr_{F_n}$ is a topological homeomorphism of $F_n$ onto $F^\sharp_n$;
\item\label{i:QH:2} $j_\pfwd \mssm=\mssm^\sharp$;
\item\label{i:QH:3} $(\mcE^\sharp,\dom^\sharp)=(\mcE^j,\dom^j)$.
\end{enumerate}

Quasi-homeomorphisms between Dirichlet spaces induce an equivalence relation, see e.g.\ \cite[Rmk.~3.4(i)]{CheMaRoe94}.
A Dirichlet space~$(\mbbX,\mcE)$ is quasi-regular if and only if it is quasi-homeomorphic to a regular Dirichlet space~$(\mbbX^\sharp,\mcE^\sharp)$, \cite[Thm.~3.7]{CheMaRoe94}. Finally, (strong) locality is invariant under quasi-homeomorphism, (consequence of)~\cite[Thm.~5.2]{Kuw98}.

\paragraph{Quasi-interior, quasi-closure} Let~$(\mbbX,\mcE)$ be a quasi-regular Dirichlet space. Every $f\in\dom$ has an \emph{$\mcE$-q.e.-unique} $\mcE$-quasi-continuous $\mssm$-representative, denoted by~$\reptwo f$, \cite[Prop.~IV.3.3.(iii)]{MaRoe92}.
For~$A\subset X$ set
\begin{align*}
\msU\eqdef& \set{G : G \text{ is an~$\mcE$-quasi-open subset of } A} \comm
\\
\msF\eqdef& \set{F : F \text{ is an~$\mcE$-quasi-closed superset of } A }\fstop
\end{align*}
By~\cite[Thm.~2.7]{Fug71},~$\msU$ has an $\mcE$-q.e.-maximal element denoted by~$\intE A$, $\mcE$-quasi-open, and called the $\mcE$-\emph{quasi-interior} of~$A$. Analogously,~$\msF$ has an $\mcE$-q.e.-minimal element denoted by~$\clE A$, $\mcE$-quasi-closed, and called the $\mcE$-\emph{quasi-closure} of~$A$.

\begin{lemma}[{\cite[Lem.~3.3]{Kuw98}}]\label{l:Kuwae}
Let~$(\mbbX,\mcE)$ be a quasi-regular Dirichlet space. Then, a family~$\seq{F_n}_n$ of $\T$-closed sets is an $\mcE$-nest if and only if~$\cup_n \intE F_n= X$ $\mcE$-q.e..
\end{lemma}

\paragraph{Domains}
Let~$(\mbbX,\mcE)$ be a Dirichlet space. We write~$\domb\eqdef \dom\cap L^\infty(\mssm)$. The \emph{extended domain}~$\domext$ of~$(\mcE,\dom)$ is the space of all functions~$f\in L^0(\mssm)$ so that there exists an $\mcE^{1/2}$-fundamental sequence~$\seq{f_n}_n\subset \dom$ with $\mssm$-a.e.-$\nlim f_n=f$. We write~$\domextb\eqdef\domext\cap L^\infty(\mssm)$. The bilinear form~$\mcE$ on~$\mcF$ extends to a non-relabeled bilinear form on~$\domext$,~\cite[Prop.~3.1]{Kuw98}. Furthermore,
\begin{itemize}
\item $\domb$ is an algebra with respect to\ the pointwise multiplication, \cite[Prop.~I.2.3.2]{BouHir91};

\item if~$(\mbbX,\mcE)$ is quasi-regular, then $\domb$ is dense in~$\dom$, \cite[Cor.~2.1]{Kuw98}; 

\item if~$(\mbbX,\mcE)$ is quasi-regular, then $\dom$ is separable, \cite[Prop.~IV.3.3, p.~102]{MaRoe92};
\end{itemize}

\subsubsection{Energy measures}
We collect some properties quasi-regular strongly local Dirichlet spaces. Set
\begin{align*}
\sqf{f,g}(h)\eqdef \mcE(fh,g)+\mcE(gh,f)-\mcE(fg,h)\comm \qquad \sqf{f}(h)\eqdef \sqf{f,f}(h) \comm \qquad f,g,h\in\domb \fstop
\end{align*}

\begin{theorem}[Cf.~{\cite[Thm.~5.2, Lem.s~5.1, 5.2]{Kuw98}}]\label{t:Kuwae}
Let~$(\mbbX,\mcE)$ be a quasi-regular strongly local Dirichlet space. Then, a bilinear form~$\sq{\emparg,\emparg}$ is defined on~$\domb^{\tym2}$ with values in~$\MbR(\Bo{\T},\Ne{\mcE})$ by
\begin{align}\label{eq:EnergyMeas}
2\int \reptwo{h}\diff \sq{f,g}= \sqf{f,g}(h) \comm \qquad f,g,h\in\domb\fstop
\end{align}

The form~$\sq{\emparg, \emparg}$ satisfies:
\begin{enumerate}[$(i)$]
\item\label{i:p:Properties:00} the representation property
\begin{align}\label{eq:Representation}
\mcE(f,g)=\tfrac{1}{2}\sq{f,g} X \comm\qquad f,g\in\domb \semicolon
\end{align}
\item\label{i:p:Properties:0} the Cauchy--Schwarz inequality
\begin{align}\label{eq:CS}
\abs{\int \rep u \, \rep v \diff\sq{f,g}}\leq \sqrt{\int \rep u^2 \diff\sq{f}}\, \sqrt{\int \rep v^2 \diff\sq{g}} \comm \qquad f,g\in\domb \comm \qquad \rep u, \rep v\in \mcL^\infty(\Bo{\T}) \semicolon
\end{align}

\item\label{i:p:Properties:1} the truncation property
\begin{align}\label{eq:Truncation}
\sq{f\wedge g, h}=&\car_{\ttset{\reptwo f\leq \reptwo g}}\sq{f,h}+\car_{\ttset{\reptwo f> \reptwo g}} \sq{g,h} \comm && f,g,h\in\domb \semicolon
\end{align}

\item\label{i:p:Properties:2} the chain rule
\begin{align}\label{eq:ChainRule}
\sq{\phi\circ f}=(\phi'\circ \reptwo f)^2 \cdot \sq{f}\comm \qquad f\in\domb \comm \qquad \phi\in \mcC^1(\R)\comm \phi(0)=0 \semicolon
\end{align}

\item\label{i:p:Properties:3} the Leibniz rule
\begin{equation}\label{eq:Leibniz}
\begin{aligned}
\sq{f g, h}=&\reptwo f\, \sq{g,h}+\reptwo g\, \sq{f,h}\comm && f,g,h\in\domb \comm
\\
\sq{f g}=&\reptwo f^2\, \sq{g}+\reptwo g^2\, \sq{f} + 2\reptwo f \, \reptwo g\,\sq{f,g}\comm && f,g\in\domb \semicolon
\end{aligned}
\end{equation}

\item\label{i:p:Properties:4} the weak Banach algebra property for~$\domb$
\begin{align}\label{eq:Continuity}
\mcE(fg)\leq 2\norm{f}_{\infty}^2\mcE(g)+2\norm{g}_{\infty}^2\mcE(f) \comm \qquad f,g\in\domb \semicolon
\end{align}

\item\label{i:p:Properties:5} the strong locality property
\begin{align}\label{eq:SLoc:1}
\car_{G} \sq{f,g} =& 0\comm & G \text{~$\mcE$-quasi-open}\comm f,g\in\domb\comm f\equiv \text{constant~$\mssm$-a.e.\ on $G$} \fstop
\end{align}
\end{enumerate}
\end{theorem}
\begin{proof} We recall the construction of~$\sq{\emparg,\emparg}$ for reference in the proof of Proposition~\ref{p:PropertiesLoc} below.

Assume first that~$(\mbbX,\mcE)$ is a regular strongly local Dirichlet space.
Since~$(X,\T)$ is Polish, then $\Mbp(\Bo{\T})=\MbpR(\Bo{\T})$, e.g.~\cite[Thm.~II.7.1.7]{Bog07}.
The existence of~$\sq{\emparg,\emparg}\colon \dom^\tym{2}\rar \Mbp(\Bo{\T}, \Ne{\mcE})$ is the standard representation of the strongly local part in the Beurling--Deny decomposition of~$(\mcE,\dom)$, e.g.~\cite[Lem.s~3.2.3, 3.2.4, pp.~126-127]{FukOshTak11}.

Assume now that~$(\mbbX,\mcE)$ is quasi-regular strongly local.
By quasi-regularity, there exists a regular (strongly local) Dirichlet space~$(\mbbX^\sharp,\mcE^\sharp)$ quasi-homeomorphic to~$(\mbbX,\mcE)$. Let~$F_\bullet\eqdef \seq{F_n}_n$, $F^\sharp_\bullet\eqdef\tseq{F_n^\sharp}_n$, and~$j\colon X_0^\sharp\eqdef \cup_n F_n^\sharp\rar X$ be witnessing the definition of quasi-homeomorphism, and set
\begin{align*}
\sq{f,g} \eqdef j_\pfwd \ttonde{ \sq{j^*f,j^*g}^\sharp\mrestr{X_0^\sharp} } \comm \qquad f,g\in \dom\fstop
\end{align*}

Since~$(\mbbX^\sharp,\mcE^\sharp)$ is regular, it is in particular quasi-regular, thus it admits an $\mcE$-nest consisting of compact sets. Therefore, up to the refinement provided by~\cite[Lem.~3.2]{Kuw98}, we may and shall assume with no loss of generality that~$F^\sharp_\bullet$ is additionally increasing and consisting of compact sets. Up to redefining~$F_\bullet\eqdef j(F^\sharp_\bullet)$, we may assume that~$F_\bullet$ too consists of compact sets, by~\iref{i:QH:1}, and increasing. Since~$\sq{f,g}$ is finite and concentrated on~$X_0\eqdef \cup_n F_n$, one has~$\nlim \sq{f,g} F_n= \sq{f,g}X$. Thus,~$\sq{f,g}$ is tight, and therefore Radon, e.g.~\cite[Prop.~II.7.2.2(iv) and (iii)]{Bog07}.
Since~$\rep f\colon X_0\rar [-\infty,\infty]$ is $\mcE$-quasi-continuous if and only if~$j^*f\colon X_0^\sharp\rar [-\infty,\infty]$ is $\mcE^\sharp$-quasi-continuous,~\cite[Cor.~3.6(iii)]{CheMaRoe94}, Equation~\eqref{eq:EnergyMeas} for~$\sq{f,g}$ follows from~\ref{i:QH:3} and~\eqref{eq:EnergyMeas} for~$\sq{j^*f,j^*g}^\sharp$.

\smallskip

Assertion~\iref{i:p:Properties:00} is~\cite[Lem.~5.1]{Kuw98}.
\iref{i:p:Properties:0} is~\cite[Lem.~5.2($\Gamma1$)]{Kuw98}.
\iref{i:p:Properties:2} is~\cite[Lem.~5.2($\Gamma3$)]{Kuw98}.
\iref{i:p:Properties:3} is~\cite[Lem.~5.2($\Gamma4$)]{Kuw98}.
\iref{i:p:Properties:1} is noted (with~$\vee$ in place of~$\wedge$) in the proof of~\cite[Lem.~5.2($\Gamma5$)]{Kuw98} for~$f,g\in\dom$ and~$h=f\wedge g$. The extension to the non-symmetric case follows by polarization.
\iref{i:p:Properties:4} follows from a consecutive application of~\eqref{eq:Leibniz},~\eqref{eq:CS}, and~\eqref{eq:Representation}.
\iref{i:p:Properties:5} Equation~\eqref{eq:SLoc:1} for~$f=g$ is~\cite[Lem.~5.2($\Gamma6$)]{Kuw98}. Let~$f$ be constant~$\mssm$-a.e.\ on~$G$. By~\eqref{eq:CS} and~\eqref{eq:SLoc:1} for~$f=g$,
\begin{align*}
\abs{\int \car_G \rep v \diff\sq{f,g}}\leq \sqrt{\int \car_G \diff\sq{f}}\,\sqrt{\int \rep v^2\diff\sq{g}}=0 \comm \qquad \rep v\in \mcL^\infty(\Bo{\T})\fstop
\end{align*}
Since~$\sq{f,g}$ does not charge $\mcE$-polar sets, we can conclude~\eqref{eq:SLoc:1} by arbitrariness of~$\rep v$.
\end{proof}

Everywhere in the following, we denote nets by the subscripts~$\alpha,\beta$, etc..

\begin{lemma}[Continuity]\label{l:Continuity}
Let~$(\mbbX,\mcE)$ be a quasi-regular Dirichlet space,~$\seq{f_\alpha}_\alpha\subset \dom$ be a net 
so that~$\dom$-$\lim_\alpha f_\alpha= f$. Then,
\begin{align*}
\lim_\alpha \sq{f_\alpha} \reptwo h= \sq{f} \reptwo h\comm \qquad h\in\domb\fstop
\end{align*}
\end{lemma}
\begin{proof}
By definition of~$\sq{\emparg,\emparg}$,
\begin{align*}
\lim_\alpha 2\int \reptwo h \diff \sq{f_\alpha} = \lim_\alpha \sqf{f_\alpha}(h) \qquad \text{~and~} \qquad \sqf{f}(h)= 2\int  \reptwo h \diff\sq{f} \comm \qquad h\in\domb \fstop
\end{align*}
Therefore, it suffice to show that~$\lim_\alpha \sqf{f_\alpha}(h)=\sqf{f}(h)$, which is a consequence of~\cite[Lem.~2.4(ii)]{HinRam03}. Although the statement there is given in the case when~$\mssm$ is a probability measure, the proof applies \emph{verbatim} to the case when~$\mssm$ is merely $\sigma$-finite. Also cf.~\cite[Lem.~3.3(ii)]{AriHin05}.
\end{proof}

\subsection{Broad local spaces}\label{ss:BroadLoc}
Let~$(\mbbX,\mcE)$ be a quasi-regular Dirichlet space. We recall the main properties of `\emph{local spaces}' as introduced by K.~Kuwae in~\cite[\S4]{Kuw98}. Following~\cite[p.~272]{FukOshTak11}, we call these spaces \emph{broad local spaces}.
For $\mcE$-quasi-open~$E\subset X$ set
\begin{align}\label{eq:Xi0}
\msG(E)\eqdef\set{ G_\bullet\eqdef\seq{G_n}_n :  G_n \text{~$\mcE$-quasi-open,~} G_n\subset G_{n+1} \text{~$\mcE$-q.e.,~} \cup_n G_n=E \qe{\mcE}} \fstop
\end{align}
When~$E=X$ we simply write~$\msG$ in place of~$\msG(X)$.
For~$G_\bullet\in\msG(E)$ and~$\msA\subset L^0(\mssm)$, we say that~$f\in L^0(\mssm\mrestr{E})$ is in the \emph{broad local space}~$\dotloc{\msA}(E,G_\bullet)$ if for every~$n$ there exists~$f_n\in \msA$ so that~$f_n=f$ $\mssm$-a.e.\ on~$G_n$.
The \emph{broad local space}~$\dotloc{\msA}(E)$ of~$(\mbbX,\mcE)$ relative to~$E$ is the space~\cite[\S4, p.~696]{Kuw98},
\begin{align}\label{eq:Xi}
\dotloc{\msA}(E)\eqdef \bigcup_{G_\bullet\in \msG(E)} \dotloc{\msA}(E,G_\bullet) \fstop
\end{align}
The set~$\dotloc{\msA}(E,G_\bullet)$ depends on~$G_\bullet$. We shall comment extensively on this fact in Remark~\ref{r:AriyoshiHino} below. We omit the specification of~$E$ whenever~$E=X$.

\begin{lemma}\label{l:LocLoc} For each~$\msA\subset L^0(\mssm)$, we have~$\dotloc{\ttonde{\dotloc{\msA}}}=\dotloc{\msA}$.
\end{lemma}
\begin{proof}
Let~$G_\bullet\in\msG$ and~$f_\bullet\in\dotloc{\msA}$ be witnessing that~$f\in\dotloc{\ttonde{\dotloc{\msA}}}$.
By definition of~$\dotloc{\msA}$, for each~$n\in \N$ there exist~$G_{n,\bullet}\in\msG$ and~$f_{n,\bullet}\subset\msA$ witnessing that~$f_n\in\dotloc{\msA}$. 
Without loss of generality, up to substituting~$G_{n,m}$ with~$G_{n,m}\cap G_m$ for every~$m,n\in\N$, we may and will assume that~$G_{n,m}\subset G_m$ $\mcE$-q.e.\ for every~$m,n$.
Thus,~$f\equiv f_n\equiv f_{n,m}$ $\mssm$-a.e. on~$G_{n,m}$ for every~$m,n$.

Since~$F_{n,\bullet}\eqdef \seq{\cl_\T G_{n,m}}_m$ is an $\mcE$-nest for each~$n$, arguing similarly to~\cite[Lem.~3.2]{Kuw98} we may find indices~$\seq{n(\ell,m)}_{\ell\in\N}$, depending on~$m\in\N$, with~$n(\ell,m)\geq \ell$ for each~$m$, and so that the set $G'_\ell\eqdef \cap_m G_{n(\ell,m),m}$ satisfies~$G'_\bullet\in\msG$.
Thus, finally, since~$G'_\ell\subset G_{n(\ell,\ell),n(\ell,\ell)}$ $\mcE$-q.e., if we set $f'_\ell\eqdef f_{n(\ell,\ell),n(\ell,\ell)}\in\msA$, the pair~$(G'_\bullet,f'_\bullet)$ witnesses that~$f\in\dotloc{\msA}$.
\end{proof}

By (the proof of)~\cite[Thm.~4.1(i)]{Kuw98},
\begin{align}\label{eq:Fbloc}
\domext\subset \dotloc{\dom}=\dotloc{(\domb)}=\dotloc{(\domext)}=\dotloc{(\domextb)} \fstop
\end{align}
On a regular Dirichlet space~$(\mbbX,\mcE)$ one usually defines the \emph{local domain}~$\domloc$ of~$(\mcE,\dom)$ as the space of all functions~$f\in L^0(\mssm)$ so that for each relatively compact open~$G\subset X$ there exists~$f_G\in\dom$ with~$f_G=f$ $\mssm$-a.e.\ on~$G$, e.g.~\cite[p.~130]{FukOshTak11}.
In this case,~$\domloc\subset \dotloc{\dom}$, however the inclusion may be a strict one.

\begin{proposition}\label{p:PropertiesLoc}
Let~$(\mbbX,\mcE)$ be a quasi-regular strongly local Dirichlet space. Then, 
the quadratic form $\sq{\emparg}\colon \dom\rar \MbpR(\Bo{\T},\Ne{\mcE})$ associated to the bilinear form~$\sq{\emparg,\emparg}$ in~\eqref{eq:EnergyMeas} uniquely extends to a non-relabeled form on~$\dotloc{\dom}$ with values in~$\Msp(\Bo{\T},\Ne{\mcE})$, satisfying:
\begin{enumerate}[$(i)$]
\item\label{i:p:PropertiesLoc:00} the representation property
\begin{align}\label{eq:RepresentationLoc}
\mcE(f,g)=\tfrac{1}{2}\sq{f,g} X \comm\qquad f,g\in\domext \semicolon
\end{align}

\item\label{i:p:PropertiesLoc:1} the truncation property
\begin{equation}\label{eq:TruncationLoc}
\begin{aligned}
f\wedge g \in \dotloc{\dom} \quad\text{and}\quad \sq{f\wedge g}=\car_{\ttset{\reptwo f\leq \reptwo g}}\sq{f}+\car_{\ttset{\reptwo f> \reptwo g}} \sq{g}\comm \qquad f,g\in\dotloc{\dom} \semicolon
\\
f\vee g \in \dotloc{\dom} \quad\text{and}\quad \sq{f\vee g}=\car_{\ttset{\reptwo f\leq \reptwo g}}\sq{g}+\car_{\ttset{\reptwo f> \reptwo g}} \sq{f}\comm \qquad f,g\in\dotloc{\dom} \semicolon
\end{aligned}
\end{equation}

\item\label{i:p:PropertiesLoc:2} the chain rule
\begin{equation}\label{eq:ChainRuleLoc}
\phi\circ f \in \dotloc{\dom} \quad\text{and}\quad \sq{\phi\circ f}=(\phi'\circ \reptwo f)^2 \cdot \sq{f}\comm \qquad f\in\dotloc{\dom} \comm \quad \begin{aligned}&\phi\in \mcC^1(\R)\comm\\ &\phi(0)=0\end{aligned} \semicolon
\end{equation}

\item\label{i:p:PropertiesLoc:5} the strong locality property
\begin{align}\label{eq:SLoc:2}
\car_G \sq{f}=\car_G \sq{g}\comm  \qquad G \text{~$\mcE$-quasi-open}\comm f,g\in\dotloc{\dom}\comm f\equiv g \text{~$\mssm$-a.e.\ on $G$}\fstop
\end{align}
\end{enumerate}

Furthermore,~$\dotloc{\dom}$ is an algebra for the pointwise multiplication, and
\begin{align}\label{eq:E(1)=0}
\car\in\dotloc{\dom}\comm \qquad \sq{\car}\equiv 0 \fstop
\end{align}
\end{proposition}
\begin{proof}
Let~$f\in\dotloc{\dom}$, witnessed by~$\seq{G_n}_n$ and~$\seq{f_n}_n\subset\dom$. By strong locality~\eqref{eq:SLoc:1}, $\car_{G_m}\sq{f_n}=\car_{G_m}\sq{f_m}$ for every~$m\leq n$, since~$G_m\subset G_n$ $\mcE$-q.e.\ and~$\sq{f_n}$ does not charge $\mcE$-polar sets.
Thus, the set function~$\sq{f}\colon \Bo{\T}\rar [0,+\infty]$ given by~$\sq{f}A\eqdef \nlim \sq{f_n} (A\cap G_n)$, $A\in\Bo{\T}$, is well-defined, since the limit is monotone, and a $\sigma$-finite Borel measure, e.g.~\cite[\S{I.I.6(d)}, p.~44]{Sch73}.
The measure~$\sq{f}$ is independent of the approximating sequence~$\seq{G_n}_n$ in the obvious way, again by~\eqref{eq:SLoc:1}, cf.\ the proof of~\cite[Lem.~5.3(i)]{Kuw98}.
This proves the extension of the quadratic form~$\sq{\emparg}\colon \dom\rar \MbpR(\Bo{\T}, \Ne{\mcE})$ to~$\sq{\emparg}\colon \dotloc{\dom}\rar \Msp(\Bo{\T}, \Ne{\mcE})$.

\smallskip

Assertion~\iref{i:p:PropertiesLoc:00} is~\cite[Lem.~5.1]{Kuw98}.
\iref{i:p:PropertiesLoc:1} is an extension of the symmetric case in Theorem~\ref{t:Kuwae}\iref{i:p:Properties:1} to~$\dotloc{\dom}$, consequence of~\cite[Thm.~4.1(i)]{Kuw98}.
\iref{i:p:PropertiesLoc:2} is a special case of~\cite[Lem.~5.2($\Gamma3$), Lem.~5.3(iii)]{Kuw98}.
\iref{i:p:PropertiesLoc:5} Equation~\eqref{eq:SLoc:2} for~$g\equiv 0$ is~\cite[Lem.~5.3(iii)]{Kuw98}.
Let~$f,g\in\dotloc{\dom}$, witnessed by~$\tseq{G_n^f}_n$, $\seq{G_n^g}_n$ and~$\seq{f_n}_n$, $\seq{g_n}_n\subset \dom$. Up to the refinement provided by~\cite[Lem.~3.2]{Kuw98}, we may and shall assume with no loss of generality that~$G_n\eqdef G_n^f=G_n^g$ for every~$n\in \N$. Furthermore, by~\eqref{eq:Fbloc}, we may and shall assume without loss of generality that~$\seq{f_n}_n$,~$\seq{g_n}_n\subset \domb$.
Since~$G\cap G_n$ is $\mcE$-quasi-open, by~\eqref{eq:SLoc:1} with~$f_n-g_n$ in place of both~$f$ and~$g$, and with~$G\cap G_n$ in place of~$G$, and by~\eqref{eq:CS},
\begin{equation}
\begin{aligned}
0=\car_{G\cap G_n}\sq{f_n-g_n} A =&\ \car_{G\cap G_n}\sq{f_n} A+\car_{G\cap G_n}\sq{g_n} A - 2\cdot \car_{G\cap G_n}\sq{f_n,g_n} A
\\
\geq&\ \ttonde{(\car_{G\cap G_n} \sq{f_n}A)^{1/2}-(\car_{G\cap G_n} \sq{g_n} A)^{1/2}}^2\comm
\end{aligned} \qquad A\in\Bo{\T}\comm
\end{equation}
and therefore,~$\car_{G\cap G_n}\sq{f_n}=\car_{G\cap G_n}\sq{g_n}$ for every~$n\in \N$.
By the definition of~$\sq{f},\sq{g}$,
\begin{align*}
\car_G \car_{G_n} \sq{f} = \car_G \car_{G_n} \sq{f_n} = \car_G \car_{G_n} \sq{g_n} = \car_G \car_{G_n} \sq{g}\comm
\end{align*}
and the conclusion follows by arbitrariness of~$n$ since~$X\setminus \cup_n G_n$ is $\mcE$-polar and both~$\sq{f}$ and~$\sq{g}$ do no charge $\mcE$-polar sets.
In order to show that~$\dotloc{\dom}$ is an algebra, let~$f,g\in\dotloc{\dom}$, and~$\seq{G_n}_b$ be as above.
Then,~$f_ng_n=fg$ $\mssm$-a.e.\ on~$G_n$. Since~$f_n$,~$g_n\in\domb$, then~$f_ng_n\in\domb$ as well by~\eqref{eq:Continuity}, thus~$\seq{f_ng_n}_n$ and~$\seq{G_n}_n$ witness that~$fg\in\domloc$.
Equation~\eqref{eq:E(1)=0} is~\cite[Lem.~5.3(i)]{Kuw98}.
\end{proof}

\begin{remark}[\emph{Caveat}]\label{r:Caveat} Despite some claims to this fact in~\cite{Stu94,BirMos95,Kuw98} and others, \emph{no} bilinear form on~$\dotloc{\dom}^\tym{2}$ is induced by~$(\sq{\emparg},\dotloc{\dom})$ with values in~$\Ms(\Bo{\T})$, by polarization or otherwise.
In order to see this, let~$(\mcE,\dom)$ be the Dirichlet form associated with the standard Brownian motion on the real line.
It is readily seen that~$\mcC^\infty_b(\R)\subset \dotloc{\dom}$, and that, for every~$f,g\in\mcC^\infty_b(\R)$ and every open relatively compact set~$U\subset \R$, one has~$\car_U\sq{f,g}=\tfrac{1}{2}\car_U f'g' \Leb^1$, yet~$f'g' \Leb^1$ is generally not a well-defined extended signed measure on~$\R$; e.g.\ choosing~$f(x)\eqdef\sin(x)$,~$g(x)\eqdef\cos(x)$ one can verify that~$\sq{f,g}\R$ is not well-defined, since~$\sq{f,g}$ is not $\sigma$-additive.

However, if~$(\mbbX,\mcE)$ is regular, a bilinear form on~$\sq{\emparg,\emparg}$ is induced on~$\dotloc{\dom}^\tym{2}$ by polarization with values in the space of real-valued Radon measures on~$\mbbX$ \emph{\`a la Bourbaki}, i.e.\ in the sense of~\cite[\S{III}.1.3, D{\'e}finition~2, p.~47]{Bou07} or~\cite[p.~58]{Sch73}. In particular, the set function~$\sq{\emparg,\emparg}$ constructed in this way is \emph{not} defined on the whole of~$\Bo{\T}$, and is merely a $\sigma$-additive functional on the algebra of relatively compact Borel sets, cf.~\cite[p.~57]{Sch73}.
\end{remark}

The next lemma is an adaptation to our setting of~\cite[Lem.~3.7]{FraLenWin14}.
\begin{lemma}\label{l:FrankLenz} 
Let~$(\mbbX,\mcE)$ be a quasi-regular strongly local Dirichlet space. Further let~$f\in \dotloc{L^\infty(\mssm)}$ be so that~$f_r\eqdef (-r)\vee f \wedge r\in \dotloc{\dom}$ for every~$r>0$, and assume there exists a measure~$\mu$ on~$(X,\Bo{\T})$ so that~$\sq{f_r}\leq \mu$ for every~$r>0$. Then,~$f\in\dotloc{\dom}$ and~$\sq{f}=\displaystyle{\lim_{r\rar\infty}} \sq{f_r}\leq \mu$.
\end{lemma}
\begin{proof}
Let~$G_\bullet\in\msG$ and~$f_\bullet\subset L^\infty(\mssm)$ be witnessing that~$f\in\dotloc{L^\infty(\mssm)}$.
Then~$f_n=f_{r_n}$ $\mssm$-a.e.\ on~$G_n$ for some~$r_n>0$. Since~$f_{r_n}\in\dotloc{\dom}$ by assumption,~$G_\bullet$ and~$\seq{f_{r_n}}_n$ witness that~$f\in\dotloc{\ttonde{\dotloc{\dom}}}$, and thus~$f\in\dotloc{\dom}$ by Lemma~\ref{l:LocLoc}.
By~\eqref{eq:TruncationLoc}, $\sq{f_r}=\car_{\ttset{-r<\reptwo f\leq r}}\sq{f}$, whence the conclusion letting~$r$ to~$\infty$, again since~$\reptwo f$ is $\mcE$-q.e.\ finite.
\end{proof}

\begin{lemma}\label{l:LinftyLoc}
Let~$(\mbbX,\mcE)$ be a Dirichlet space satisfying~\iref{i:QR:1}. Then,~$\Cont(\T)\subset \dotloc{L^\infty(\mssm)}$.
\end{lemma}
\begin{proof}
Let~$\seq{F_n}_n$ be an $\mcE$-nest consisting of compact sets, and note that~$G_n\eqdef\intE F_n$ satisfies~$\seq{G_n}_n\in \msG$ by Lemma~\ref{l:Kuwae}. Further let~$f\in\Cont(\T)$ and note that~$f$ is bounded on~$F_n$, hence $\mssm$-a.e.\ on~$G_n$, for each~$n$ by $\T$-continuity and by $\T$-compactness of~$F_n$.
\end{proof}

\subsection{Dominance, moderance, and smoothness}
As shown in Example~\ref{ese:LQG} below, it can happen that $\mssd_\mssm=0$.
For this reason, one extends the usual definition of the intrinsic distance~$\mssd_\mssm$ to that~\eqref{eq:Intro:Intrinsic2} of~$\mssd_\mu$.
We introduce some classes of measures on~$(X,\Bo{\T})$, in particular that of $\mcE$-dominant $\mcE$-moderate measures, for which the definition of~$\mssd_\mu$ is meaningful and interesting.

\subsubsection{Dominance}\label{sss:Dominance}
We say that a Dirichlet space~$(\mbbX,\mcE)$ \emph{admits carr\'e du champ} if~$\sq{f}\ll \mssm$ for every~$f\in\dom$, in which case~$\mcE(f,g)=\tfrac{1}{2}\int \cdc(f,g)\diff\mssm$ where~$\cdc\colon \dom^\tym{2}\rar L^1(\mssm)$ is the carr\'e du champ operator~$\cdc\colon (f,g)\mapsto \frac{\diff\sq{f,g}}{\diff\mssm}$.
The potential lack of carr\'e du champ operator partly motivates the next definition, introduced by M.~Hino,~\cite[Dfn.~2.1]{Hin09} on regular Dirichlet spaces.

\begin{definition}[$\mcE$-dominance]\label{d:Dominance}
Let~$(\mbbX,\mcE)$ be a quasi-regular strongly local Dirichlet space. A $\sigma$-finite measure~$\mu$ on~$(X,\Bo{\T})$ is:
\begin{itemize}
\item \emph{$\mcE$-dominant}, if~$\sq{f}\ll \mu$ for every $f\in\dom$;
\item \emph{minimal $\mcE$-dominant}, if $\mu$ is $\mcE$-dominant and~$\mu\ll \nu$ for every $\mcE$-dominant~$\nu$.
\end{itemize}
We denote by~$\domdom$ the set of functions~$f\in\dom$ so that~$\sq{f}$ is minimal $\mcE$-dominant. 
\end{definition}

It is readily verified that, if~$(\mcE,\dom)$ admits carr\'e du champ operator, then~$\mssm$ is minimal $\mcE$-dominant. In general however,~$\mssm$ may be singular with respect to any $\mcE$-dominant measure~$\mu$, as shown by the next example.

\begin{example}[Sierpi\'nski gaskets]
The Dirichlet form~$(\mcE,\dom)$ associated with the Brownian motion on the standard Sierpi\'nski gasket was constructed by S.~Kusuoka, also cf.\ Goldstein~\cite{Gol87} for a different construction of the process. In the notation of~\cite{Kus89}, the measure~$\tilde\mu$ constructed there is (minimal) $\mcE$-dominant, and singular with respect to\ the reference measure~$\tilde\nu$.
\end{example}

Let us start by showing that~$\domdom$ is non-empty.

\begin{proposition}\label{p:EDominant}
Let~$(\mbbX,\mcE)$ be a quasi-regular strongly local Dirichlet space. Then,
\begin{enumerate}[$(i)$]
\item\label{i:p:EDominant:1} $\domdom$ is dense in~$\dom$;
\item\label{i:p:EDominant:2} every minimal $\mcE$-dominant~$\mu$ does not charge~$\mcE$-polar sets.
\end{enumerate}
\end{proposition}
\begin{proof}
\iref{i:p:EDominant:1} If~$(\mbbX,\mcE)$ is regular, the statement is~\cite[Prop.~2.7]{Hin09}.
For an arbitrary Dirichlet space~$(\mbbX,\mcE)$ let~$j\colon (\mbbX,\mcE)\rar (\mbbX^\sharp,\mcE^\sharp)$ be a quasi-homeomorphism to a regular Dirichlet space~$(\mbbX^\sharp,\mcE^\sharp)$, and~$\seq{F_n}_n$, $\tseq{F^\sharp_n}_n$ be nests witnessing the quasi-homeomorphism property~\iref{i:QH:1} of~$j$. Set~$X_0\eqdef \cup_n F_n$ and recall that
\begin{align*}
\sq{j^*f^\sharp,j^*g^\sharp} A\eqdef \sq{f^\sharp,g^\sharp}^\sharp j(A\cap X_0) \comm \qquad A\in\Bo{\T} \comm \qquad f^\sharp,g^\sharp\in\dom^\sharp\fstop
\end{align*}

Let~$X^\sharp\eqdef j(X_0)=\cup_n F_n^\sharp$ and note that~$X^\sharp\setminus X_0^\sharp$ is $\mcE^\sharp$-polar. Since~$\sq{f^\sharp,g^\sharp}^\sharp\in\Mbp(\Bo{\T^\sharp},\Ne{\mcE^\sharp})$ does not charge $\mcE^\sharp$-polar sets, we have that~$\sq{f,g}=(j^{-1})_\pfwd \sq{j^{-1*}f, j^{-1*}g}^\sharp\mrestr{X_0^\sharp}$ as measures on~$X_0$, where~$j^{-1}$ denotes the inverse of~$j$ defined on~$X_0^\sharp$.
Since the push-forward of measures preserves absolute continuity,~$\sq{j^*f^\sharp}$ is minimal $\mcE$-dominant if and only if~$\sq{f^\sharp}^\sharp$ is minimal $\mcE^\sharp$-dominant. In particular,~$\domdom=j^*\domdom^\sharp$ as a consequence of the isomorphism~\iref{i:QH:3}.

\iref{i:p:EDominant:2} Let~$\mu$ be minimal $\mcE$-dominant. By definition,~$\mu\ll\sq{f}$ for every~$f\in \domdom$. The conclusion follows since~$\sq{f}$ does not charge $\mcE$-polar sets.
\end{proof}

\subsubsection{Moderance}\label{sss:Moderance}
Let~$(X,\T)$ be a paracompact completely regular Hausdorff space, and~$\kappa\subset \Bo{\T}$ be any family of Borel subsets of $X$. We say that a Borel measure~$\mu$ on~$(X,\T)$ is $\kappa$-\emph{moderate} if there exists a $\mu$-negligible set~$N\subset X$ and a countable cover~$\seq{A_n}_n\subset \kappa$ of~$X\setminus N$ so that~$\mu A_n<\infty$. As usual, if~$\kappa$ is closed under finite unions, we may arrange so that~$\seq{A_n}_n$ be increasing.
When~$\kappa=\Bo{\T}$, the definition of $\kappa$-moderance reduces to that of $\sigma$-finiteness.
When~$\kappa=\T$ and~$N=\emp$, it is standard to say that~$\mu$ is \emph{moderate}, cf.~e.g.~\cite[\S{IX}.1.9, D{\'e}finition~12, p.~21]{Bou69}.

A sequence of functions~$\seq{\phi_n}_n$ is an \emph{algebraic approximation to the identity} if
\begin{align*}
0\leq \phi_n\leq \phi_{n+1}\nearrow_n 1\comm \qquad \seq{\phi_n}_n\subset \Cb(\T)\comm
\end{align*}
in which case~$\nlim f\phi_n=f$ for every~$f\in \Cb(\T)$. Up to relabeling~$\phi_n$ by~$n\phi_n\wedge 1$ and possibly passing to a subsequence, we may and shall assume with no loss of generality that~$U_n\eqdef \inter_\T\set{\phi_n =1}\neq \emp$ defines an open covering of~$(X,\T)$.

Analogously, a sequence of functions~$\seq{\psi_n}_n$ is a \emph{latticial approximation to the identity} if
\begin{align*}
0\leq \psi_n \leq \psi_{n+1} \nearrow_n \infty\comm \qquad \seq{\psi_n}_n\subset \Cb(\T)\comm
\end{align*}
in which case~$\nlim (-\psi_n)\vee (f\wedge \psi_n)=f$ for every~$f\in \Cb(\T)$. Up to relabeling~$\psi_n$ by~$\psi_n\wedge n$ and possibly passing to a subsequence, we may and shall assume with no loss of generality that~$V_n\eqdef \inter_\T\set{\psi_n =n}\neq \emp$ defines an open covering of~$(X,\T)$ and~$\psi_n\leq n$ for every~$n\in \N$.

In this way, a bijective correspondence between algebraic and latticial approximations to the identity is induced by letting~$\psi_n\eqdef n\phi_n$.
As a consequence of paracompactness, the following are equivalent:
\begin{enumerate}[$(a)$]
\item $\mu$ is moderate; 
\item\label{i:Moderation:2} there exists an algebraic approximation to the identity~$\seq{\phi_n}_n$ so that~$\mu\phi_n<\infty$ for every~$n$;
\item\label{i:Moderation:3} there exists a latticial approximation to the identity~$\seq{\psi_n}_n$ so that~$\mu\psi_n<\infty$ for every~$n$.
\end{enumerate}
This suggests that moderance may be expressed in terms of a family~$\msK$ of real-valued functions on~$X$, rather than of a family~$\kappa$ of subsets of~$X$.
Informally,~$\mu$ is `moderate' if the elements of~$\msK$ may be `localized' to sets of finite $\mu$-measure by approximations to the identity.
Different notions of moderance may arise depending on the chosen localization procedure; for instance, if~$\msK$ is endowed with additional structure, e.g., if it is an algebra, as $\Cb(\T)$ in~\iref{i:Moderation:2}, or a lattice, as~$\Cb(\T)$ in~\iref{i:Moderation:3}.

\paragraph{Energy moderance} Further abstracting away from families of sets, an analogous concept of moderance may be given for functionals on~$\msK$ more general than $\mu$-integration. This is usually the case in describing local Dirichlet spaces, where several notions of moderance naturally appear. 
Before discussing the literature, let us give some precise definitions.
 
\begin{definition}\label{d:Nest}
 Let~$(\mbbX,\mcE)$ be a quasi-regular Dirichlet space. A countable family of Borel sets $A_\bullet\eqdef \seq{A_n}_n$ is \emph{$\mcE$-moderate} if for every~$n$ there exists~$e_n\in\dom$ so that~$\reptwo e_n=1$ $\mssm$-a.e.\ on~$A_n$.
\end{definition}
 
\begin{remark} \label{r:AriHino}
Definition~\ref{d:Nest} is modeled after~\cite{AriHin05}:
An $\mcE$-moderate Borel $\mcE$-nest is but a `\emph{nest}' in the sense of~\cite[Dfn.~2.1]{AriHin05}.
If a family of Borel sets~$A_\bullet$ is $\mcE$-moderate, we may take~$e_n\in \domb$ so that~$0\leq \reptwo e_n\leq 1$ without loss of generality, by the Markov property of~$\mcE$,~\cite[Rmk.~2.2]{AriHin05}, and additionally so that~$\reptwo e_n=1$ $\mcE$-q.e.\ on~$A_n$.
Indeed, since~$\reptwo e_n$ is $\mcE$-quasi-continuous, then~$\reptwo e_n^{-1}(\set{1})$ is $\mcE$-quasi-closed, and we have~$A_n\subset \cl_\mcE A_n\subset \reptwo e_n^{-1}(\set{1})$, where all inclusions hold up to $\mcE$-polar sets.

As a consequence of the quasi-regularity of~$(\mbbX,\mcE)$, in the language of Choquet capacities,~$A_\bullet$ is $\mcE$-moderate if and only if each~$A_n$ has finite $\mcE_1$-capacity, cf.~\cite[Thm.~2.1.5]{FukOshTak11}.
\end{remark}

For an $\mcE$-quasi-open~$E\subset X$, write
\begin{equation}\label{eq:G0}
\begin{aligned}
\msG_0(E)\eqdef&\ \set{G_\bullet \in\msG(E): G_\bullet \text{~is $\mcE$-moderate}} \comm
\\
\msG_c(E)\eqdef&\ \set{G_\bullet \in\msG_0(E) : \cl_\T G_n \text{~is $\T$-compact for all~$n$}}\fstop
\end{aligned}
\end{equation}
For~$G_\bullet\in\msG_0(E)$, we write~$e_\bullet\eqdef\seq{e_n}_n$ for any sequence of functions witnessing the $\mcE$-moderance of~$G_\bullet$. When the sequence~$e_\bullet$ is relevant, we write as well~$(G_\bullet,e_\bullet)\in\msG_0(E)$. As usual, we omit the specification of~$E=X$.
Since~$\car\in\dotloc{\dom}$ by~\eqref{eq:E(1)=0}, then~$\msG_0\neq \emp$. Clearly,~$\msG_c\subset \msG_0\subsetneq \msG$ as in~\eqref{eq:Xi0}, cf.~\cite[Lem.~3.5(iii)]{Kuw98}. On the other hand, since $\mcE$-moderance (resp.\ compactness) is hereditary (resp.\ hereditary with respect to\ closed sets), both~$\msG_0$ and~$\msG_c$ are $\cap$-ideals of~$\msG$, i.e.\ if~$G_{*,\bullet}\in\msG_*$ and~$G_\bullet\in\msG$, then~$G'_\bullet\eqdef \seq{G_{*,n}\cap G_n}_n$ satisfies~$G_\bullet'\in\msG_*$ for $*=0$ or~$c$. In particular, we have the following.

\begin{lemma}\label{l:nests}
Let~$(\mbbX,\mcE)$ be a quasi-regular Dirichlet space. Then, for every~$G_\bullet\in\msG$ there exists~$G_{c,\bullet}\in\msG_c$ so that~$\dotloc{\dom}(G_\bullet)=\dotloc{\dom}(G_{c,\bullet})$.
\end{lemma}
\begin{proof}
Let~$G_\bullet\in\msG$ and~$f\in \dotloc{\dom}(G_\bullet)$ be witnessed by~$f_\bullet$. Further let~$G_\bullet'\in\msG_0$ be witnessing that~$\car\in\dotloc{\dom}$, and~$F_\bullet$ be a $\T$-compact increasing $\mcE$-nest given by~\iref{i:QR:1}. We have that~$G_n'' \eqdef \intE F_n$ is $\mcE$-q.e.\ increasing in~$n$ and $\mcE$-quasi-open. Furthermore,~$X=\cup_n G_n''$ $\mcE$-q.e.\ by Lemma~\ref{l:Kuwae}, thus $G_\bullet''\eqdef\seq{G_n''}_n\in\msG$.
Set~$G_{c,n}\eqdef G_n\cap G_n'\cap G_n''$. It is readily verified that~$G_{c,\bullet}\in\msG_c$. Furthermore,~$f_n= f$ $\mssm$-a.e.\ on~$G_{c,n}$ and so~$f_\bullet$ witnesses that~$f\in\dotloc{\dom}(G_{c,\bullet})$ as well.
\end{proof}

As a consequence of the previous lemma,~$\msG_c\neq\emp$, and
\begin{align}\label{eq:DotLocGc}
\dotloc{\dom}=\bigcup_{G_\bullet\in \msG_c} \dotloc{\dom}(G_\bullet) \fstop
\end{align}

A minimal $\mcE$-dominant measure~$\mu$ may not be moderate in any reasonable sense, despite its minimality. In order to ensure the existence of sufficiently many $\mu$-integrable functions of interest, we shall need the following definition.

\begin{definition}[$\mcE$-moderance]\label{d:Moderance}
Let~$(\mbbX,\mcE)$ be a quasi-regular Dirichlet space. For a measure~$\mu\in\Msp(\Bo{\T},\Ne{\mcE})$ we say that $(G_\bullet,e_\bullet)\in\msG_0$ is \emph{$\mu$-moderated} if~$e_\bullet$ is additionally so that~$\mu \reptwo e_n <\infty$ for every~$n$.
We say that~$\mu$ is:
\begin{itemize}
\item \emph{$\mcE$-moderate} if there exists a $\mu$-moderated $G_\bullet\in\msG_0$;
\item \emph{absolutely $\mcE$-moderate} if for every $G_\bullet\in\msG_0$ there exists~$e_\bullet$ so that~$(G_\bullet, e_\bullet)$ is $\mu$-moderated.
\end{itemize}

We denote by~$\mcM$ (resp.~$\mcM_0$) the space of all (absolutely) $\mcE$-moderate measures.
\end{definition}

Since $\mcE$-moderate measures do not charge $\mcE$-polar sets by definition, then~$\mu(X\setminus \cup_n G_n)=0$ for every $G_\bullet\eqdef\seq{G_n}_n\in\msG$.
Therefore, $\mcE$-moderate measures are $\kappa$-moderate for the family~$\kappa$ of all $\mcE$-quasi-open subsets of~$X$, which motivates the terminology.
If~$(G_\bullet,e_\bullet)$ is $\mu$-moderated, then $e_\bullet$ is an algebraic approximation to the identity in the $\mcE$-q.e.\ sense ---~analogously to the purely topological case discussed in the previous paragraph.
The absolute $\mcE$-moderance of~$\mssm$ is implicit in the definition of~$\msG_0$:
if~$(G_\bullet,e_\bullet)\in\msG_0$, then~$(G_\bullet, e_\bullet^2)$ is $\mssm$-moderated.
Finally, since~$2\,\sq{f}\reptwo h=\sqf{f}(h)$ is finite for every~$f,h\in\domb$, by Proposition~\ref{p:EDominant} there exist (plenty of) $\mcE$-moderate minimal $\mcE$-dominant (Radon) measures, which makes the definition non-void.

\subsubsection{Smoothness}\label{sss:Smoothness}
Let us compare $\mcE$-moderance with the following well-known definition.

\begin{definition}\label{d:Smooth}
Let~$(\mbbX,\mcE)$ be a quasi-regular Dirichlet space.
A measure~$\mu\in\Msp(\Bo{\T}, \Ne{\mcE})$ is:
\begin{itemize}
\item \emph{of finite $\mcE$-energy integral} if there exists a constant~$c>0$ so that, for any $\mcE$-quasi-continuous $\mssm$-representative $\reptwo u$ of~$u$,
\begin{align}\label{eq:DongMa}
\int \abs{\reptwo u} \diff \mu \leq c\, \sqrt{\mcE_1(u)} \comm \qquad u\in\dom\semicolon
\end{align}
\item $\mcE$-\emph{smooth} if there exists a compact $\mcE$-nest~$F_\bullet$ so that~$\mu F_n<\infty$ for every~$n$.
\end{itemize}
We denote by~$\mcS_0$, resp.~$\mcS$, the space of all measures of finite $\mcE$-energy integral, resp.\ $\mcE$-smooth measures.
\end{definition}

The above definition of `measure of finite $\mcE$-integral' is taken from~\cite[Dfn.~1]{DonMa93}. The fact that~$\mu$ does not charge $\mcE$-polar sets is in fact a consequence of~\eqref{eq:DongMa} rather than part of the definition; cf.~\cite[Lem.~2.2.3, p.~79]{FukOshTak11} in the regular case.
Contrary to the definition for regular Dirichlet spaces~\cite[p.~77]{FukOshTak11}, one does not require a measure of finite $\mcE$-energy integral to be Radon.
The above definition of `$\mcE$-smooth measure' is taken from~\cite[Dfn.~2.2]{RoeSch95} or~\cite[Dfn.~2]{DonMa93}. It coincides with the more standard definition~\cite[p.~123]{MaRoe92} on any Dirichlet space satisfying~\iref{i:QR:1}.
It is a standard result in the theory that $\mcE$-smooth measures are in one-to-one correspondence with \emph{positive continuous additive functionals} on the Dirichlet space~$(\mcE,\dom)$, e.g.~\cite[\S{VI.2} and Thm.~VI.2.4]{MaRoe92}.

For every $\mcE$-smooth~$\mu$, a quadratic form~$(\mcE^\mu,\dom^\mu)$ is induced on~$X$ by setting
\begin{align}\label{eq:Perturbed}
\dom^\mu\eqdef \tset{f\in\dom : \reptwo f\in \mcL^2(\mu)}\comm \qquad \mcE^\mu(f,g)\eqdef \mcE(f,g)+\int \reptwo f \,\reptwo g\, \diff\mu \fstop
\end{align}
The form~$(\mcE^\mu,\dom^\mu)$ is in fact a Dirichlet form, called the $\mu$-perturbation of~$(\mcE,\dom)$, and, if~$(\mbbX,\mcE)$ is a quasi-regular strongly local Dirichlet space, then~$(\mbbX,\mcE^\mu)$ is so as well, for every $\mcE$-smooth~$\mu$, e.g.~\cite[Prop.~2.3]{RoeSch95}.
By definition of~$\dom^\mu$, there is a natural inclusion~$\dom^\mu\hookrightarrow\dom$ and therefore~$\dotloc{(\dom^\mu)}\hookrightarrow \dotloc{\dom}$, by which we shall always mean that~$\dom^\mu\subset\dom$ up to the choice of suitable representatives.

\begin{proposition}\label{p:Smooth}
Let~$(\mbbX,\mcE)$ be a quasi-regular Dirichlet space. Then,
\begin{equation*}
\begin{aligned}
\xymatrix@=10pt{\mcM_0 \ar@{}[r]|-*{\subset} & \mcM_{\phantom{0}} \\ \mcS_0 \ar@{}[u]|-*{\cup} \ar@{}[r]|-*{\subset} & \mcS_{\phantom{0}} \ar@{}[u]|-*{\cup\,\,}
}
\end{aligned} \fstop
\end{equation*}
\end{proposition}
\begin{proof}
The inclusion~$\mcS_0\subset \mcM_0$ is straightforward, while~$\mcM_0\subset \mcM$ holds by definition. The inclusion~$\mcS_0\subset \mcS$ is standard in the regular case, e.g.~\cite[p.~84]{FukOshTak11}; see~\cite{DonMa93} for the quasi-regular case.
In order to show that~$\mcS\subset\mcM$, let~$\mu\in\mcS$.
Since~$(\mcE,\dom)$ is quasi-regular, the $\mu$-perturbed form $(\mcE^\mu,\dom^\mu)$ in~\eqref{eq:Perturbed} is quasi-regular as well, and for any sequence~$F_\bullet$ of $\T$-closed sets,
\begin{align}\label{eq:p:Smooth:1}
\text{$F_\bullet$ is an $\mcE$-nest} \qquad \text{if and only if} \qquad \text{$F_\bullet$ is an $\mcE^\mu$-nest} \comm
\end{align}
see the proof of~\cite[Prop.~2.3]{MaRoe92}, also cf.~\cite[Lem.~IV.4.5]{MaRoe92} for the regular case.
It follows that the family~$\msG^\mu$ as in~\eqref{eq:Xi0} relative to~$(\mcE^\mu,\dom^\mu)$ is in fact independent of~$\mu\in\mcS$.
We may therefore write~$\msG$ in place of~$\msG^\mu$.

By~\cite[Lem.~3.5(iii)]{Kuw98}, there exists~$G'_\bullet\in\msG$ and a sequence of functions~$\reptwo e_n^\mu\in \dom^\mu$ satisfying~$0\leq \reptwo e_n^\mu\leq 1$ and~$\reptwo e_n^\mu=1$ $\mcE^\mu$-q.e.\ on~$G'_n$. Setting~$F'_n\eqdef \cl_\T G'_n$, we have that~$F'_\bullet$ is an $\mcE^\mu$-nest, since~$G'_\bullet\in\msG$. 
Let~$\msG_0^\mu$ be defined as in~\eqref{eq:G0} with~$\mcE^\mu$ in place of~$\mcE$. Since~\iref{i:QR:1} holds for~$(\mcE^\mu,\dom^\mu)$, by~\cite[Lem.~3.2]{Kuw98} there exists a compact $\mcE^\mu$-nest~$F_\bullet$ refining~$F'_\bullet$.
Thus, by Lemma~\ref{l:Kuwae},~$G_n\eqdef \inter_\mcE F_n$ satisfies~$(G_\bullet, \reptwo e^\mu_\bullet)\in\msG_0^\mu$, and~$G_\bullet\in\msG$. In fact, since~$\dom^\mu \hookrightarrow \dom$, we additionally  have that~$G_\bullet\in \msG_0$.
The proof is concluded if we show that~$(G_\bullet, \reptwo e^\mu_\bullet)$ is $\mu$-moderated. Up to possibly replacing~$\reptwo e^\mu_\bullet$ by its square, we have~$\mu\reptwo e^\mu_n<\infty$, since~$\dom^\mu\subset L^2(\mu)$ by definition.
\end{proof}

Let us note that many inclusions in Proposition~\ref{p:Smooth} become trivial if one further assumes the involved measures to be Radon. In particular, it is clear from the proof that every $\mcE$-moderate Radon measure is~$\mcE$-smooth. For examples of non-Radon $\mcE$-smooth measures, see Example~\ref{ese:SingularPerturb} below.

\subsection{Intrinsic distances} 
In this section we introduce the generalized intrinsic distance induced 
by a Borel measure~$\mu$. We start with the definition of local domains of functions with $\mu$-bounded energy measure.

\subsubsection{Local domains of bounded-energy}\label{sss:LocDom}
Let~$(\mbbX,\mcE)$ be a quasi-regular strongly local Dirichlet space, and~$\mu\in\Msp(\Bo{\T},\Ne{\mcE})$. For $\mcE$-quasi-open~$E\subset X$ and any $G_\bullet\in\msG_0(E)$, set
\begin{align*}
\DzLoc{\mu}(E,G_\bullet)\eqdef& \set{f\in \dotloc{\dom}(E,G_\bullet): \sq{f}\leq \mu}\comm \qquad \DzLoc{\mu}(E)\eqdef \tset{f\in \dotloc{\dom}(E): \sq{f}\leq \mu}\comm
\\
\Dz{\mu}(E)\eqdef& \DzLoc{\mu}(E)\cap \dom\comm \quad \ \DzLoc{\mu,\T}(E)\eqdef \DzLoc{\mu}(E)\cap \Cont(E,\T) \comm \quad \ \Dz{\mu,\T}(E)\eqdef \DzLoc{\mu,\T}(E)\cap \dom \fstop
\end{align*}
As usual, we omit~$E$ from the notation whenever~$E=X$. For~$G_\bullet\in\msG_0(E)$ we additionally denote by
\begin{align*}
\DzLocB{\mu}(E,G_\bullet)\eqdef \DzLoc{\mu}(E,G_\bullet)\cap L^\infty(\mssm)
\end{align*}
the space of $\mssm$-essentially uniformly bounded functions in~$\DzLoc{\mu}(E,G_\bullet)$.
Let the analogous definitions for~$\DzLocB{\mu}(E)$,~$\DzB{\mu}(E)$,~$\DzLocB{\mu,\T}(E)$, and~$\DzB{\mu,\T}(E)$ be given.

The next results are an adaptation to our setting of~\cite[Lem.~3.8, Prop.~3.9]{AriHin05}.
\begin{lemma}\label{l:AriyoshiHino}
Let~$(\mbbX,\mcE)$ be a quasi-regular strongly local Dirichlet space, and~$\mu\in\Msp(\Bo{\T},\Ne{\mcE})$. Further let~$G_\bullet\in\msG_0$, and fix~$f\in\DzLocB{\mu}(G_\bullet)$ and~$g\in\domb$ with~$\reptwo g\in L^2(\mu)$. Then~$fg\in\dom$ and $\norm{fg}_\dom\leq \sqrt{2}\norm{f}_{L^\infty(\mssm)}\norm{g}_\dom+\norm{\reptwo g}_{L^2(\mu)}$.
\end{lemma}
\begin{proof}
Let~$\seq{f_n}_n$ be witnessing that~$f\in\DzLocB{\mu}(G_\bullet)$, and note that we may take~$\abs{f_n}\leq \norm{f}_{L^\infty(\mssm)}$ without loss of generality. By~\eqref{eq:Leibniz},
\begin{equation}\label{eq:l:AriyoshiHino:1}
\begin{aligned}
2\, \mcE(f_n g)\leq& \int \reptwo g^2\diff\sq{f_n}+\int \reptwo f_n^2 \diff\sq{g}\leq \norm{\reptwo g}_{L^2(\mu)}^2+2\norm{f}_{L^\infty(\mssm)}^2\, \mcE(g) \comm
\\
\norm{f_n g}_{L^2(\mssm)}\leq& \norm{f}_{L^\infty(\mssm)}\norm{g}_{L^2(\mssm)} \fstop
\end{aligned}
\end{equation}
In particular,~$\sup_n \norm{f_n g}_\dom<\infty$. Since additionally $\mssm$-a.e.-$\nlim f_n g=fg$, then~$L^2(\mssm)$-$\nlim f_n g= fg$ by Dominated Convergence, with dominating function~$\norm{f}_{L^\infty(\mssm)} \abs{g}$. Thus, by \cite[Lem.~I.2.12]{MaRoe92}, $f g\in\dom$ and
\begin{align*}
\norm{fg}_\dom^2 \leq \nliminf \norm{f_n g}_\dom^2 \leq \tfrac{1}{2}\norm{\reptwo g}_{L^2(\mu)}^2+\norm{f}_{L^\infty(\mssm)}^2\, \mcE(g)+\norm{f}_{L^\infty(\mssm)}^2\norm{g}_{L^2(\mssm)}^2
\end{align*}
by~\eqref{eq:l:AriyoshiHino:1}, whence the conclusion follows.
\end{proof}

We are now ready to show that~$\DzLocB{\mu}(G_\bullet)$ does not, in fact, depend on~$G_\bullet$. 

\begin{proposition}\label{p:AriyoshiHino}
Let~$(\mbbX,\mcE)$ be a quasi-regular strongly local Dirichlet space,~$\mu\in\Msp(\Bo{\T},\Ne{\mcE})$. Then, for any pair of $\mu$-moderated $G_\bullet$,~$G_\bullet'\in\msG_0$,
\begin{align*}
\DzLocB{\mu}(G_\bullet)=\DzLocB{\mu}(G_\bullet') \fstop
\end{align*} 
\end{proposition}
\begin{proof}
Let~$f\in\DzLocB{\mu}(G_\bullet')$. It suffices to show that~$f\in\DzLocB{\mu}(G_\bullet)$. To this end, let~$e_\bullet\eqdef\seq{e_n}_n\subset \domb$ be witnessing that $G_\bullet$ is $\mu$-moderated, and set~$f_n\eqdef f\cdot e_n^2$. Note that~$e_n^2\in L^1(\mssm)\cap L^\infty(\mssm)$, and therefore~$e_n^2\in L^2(\mssm)$ for all~$n$. Since~$f_n\in\dom$ by Lemma~\ref{l:AriyoshiHino}, and since~$f=f_n$ $\mssm$-a.e.\ on~$G_n$ by definition, then~$f\in\dotloc{\mcF}(G_\bullet)$, and the conclusion follows.
\end{proof}

\begin{remark}\label{r:AriyoshiHino}
The independence of~$\DzLocB{\mu}(G_\bullet)$ on~$G_\bullet$ is a feature of the subspace of \emph{bounded} functions in the local domain of bounded energy. In general,~$\DzLoc{\mu}(G_\bullet)$ \emph{does depend} on~$G_\bullet$, as shown by~\cite[Ex.~2.9(iii)]{AriHin05}.
\end{remark}

\begin{definition}
Let~$\mu\in\Msp(\Bo{\T},\Ne{\mcE})$. The \emph{intrinsic distance generated by~$\mu$} is the extended pseudo-distance~$\mssd_\mu\colon X^{\times 2}\rar [0,\infty]$ defined as
\begin{align}\label{eq:IntrinsicD}
\mssd_\mu(x,y)\eqdef \sup\tset{f(x)-f(y) : f\in \DzLocB{\mu,\T}} \fstop
\end{align}
\end{definition}

Note that~$\mssd_\mu$ is always $\T^\tym{2}$-l.s.c., hence $\Bo{\T^\tym{2}}$- and $\A^{\otym 2}$-measurable, for it is the supremum of a family of $\T^\tym{2}$-continuous functions. Furthermore,~$\mssd_\mu$ is $\T$-admissible by definition, as witnessed by the bounded uniformity
\begin{align*}
\UP_\mu\eqdef \set{\mssd_f(x,y)\eqdef \abs{f(x)-f(y)}: f\in\DzLocB{\mu,\T}}\fstop
\end{align*}

\begin{remark}
Intrinsic distances are occasionally defined by means of a family~$\mcC$ of continuous functions, satisfying the symmetry condition~$-\mcC\subset\mcC$, see e.g.~\cite{Ebe96,Kuw96}. Whereas we do not explicitly allow for such generality, all of the results in the following remain valid by further assuming, when appropriate, any of the following:
\begin{enumerate*}[$(a)$]
\item $\mcC$ separates points in~$X$;
\item $\mcC$ is a $(\vee,\wedge)$-lattice;
\item $\mcC$ is an algebra.
\end{enumerate*}
\end{remark}

By lower semi-continuity,~$\cl_{\mssd_\mu} A\subset \cl_{\T}A$ for every~$A\in X$, yet the inclusion may be a strict one, even when~$\mssd_\mu$ is an extended distance (as opposed to: extended pseudo-distance).
If~$\mssd_\mu$ is an extended distance, then~$(X,\mssd_\mu\wedge r)$ is a metric space for every~$r>0$, yet in general it is \emph{not} separable.
Thus, possibly,~$\mssd_\mu(\emparg, A)< \inf_n \mssd(\emparg, y_n)$ for every countable family~$\seq{y_n}_n\subset A$.
Finally,~$\mssd_\mu(\emparg, A)$ needs \emph{not} be $\Bo{\T}$-measurable, even if~$A\in \Bo{\T}$.
In the following, this fact motivates the use of nets in place of sequences.

\smallskip \emph{Caveat:} The Monotone and Dominated Convergence Theorems do \emph{not} hold for nets of functions in~$L^2(\mssm)$, even if the net is bounded and consisting of continuous compactly supported functions. A substitute is provided by the next lemma.

\begin{lemma}[{\cite[\S{I.6(b)}, Prop.~5, p.~42]{Sch73}}]\label{l:MonotoneNets1}
Let~$\mbbX$ be satisfying~\iref{ass:Hausdorff}, and additionally be so that~$\mssm$ is Radon. 
Then,
\begin{enumerate*}[$(a)$]
\item if~$\seq{f_\alpha}_\alpha$ is an upwards-directed family of non-nega\-tive $\T$-l.s.c.\ functions, then, there exists~$\mssm(\sup_\alpha f_\alpha)= \sup_\alpha \mssm f_\alpha$;
and
\item if~$\seq{f_\alpha}_\alpha$ is a downwards-directed family of non-negative $\T$-u.s.c.\ functions so that~$\mssm f_\alpha<\infty$ for some~$\alpha$, then there exists $\mssm (\inf_\alpha f_\alpha)=\inf_\alpha \mssm f_\alpha $.
\end{enumerate*}
\end{lemma}

As a consequence of Proposition~\ref{p:AriyoshiHino}, our definition~\eqref{eq:IntrinsicD} of intrinsic distance generalizes the standard notions, in particular~\cite[Eqn.~(1.3)]{Stu94} for strongly local regular Dirichlet spaces, as we now show.

\begin{proposition}[Cf.{~\cite[Rmk.~2.8]{AriHin05}}]\label{p:BoundedDist}
Let~$(\mbbX,\mcE)$ be a regular strongly local Dirichlet space. Then
\begin{equation}\label{eq:p:ConsistencyD:0}
\begin{aligned}
\mssd_\mssm(x,y)=&\sup\set{f(x)-f(y) : f\in \domloc \cap \Cb(\T) \comm \sq{f}\leq \mssm}
\\
=&\sup\set{f(x)-f(y) : f\in \domloc \cap \Cont(\T) \comm \sq{f}\leq \mssm} \fstop
\end{aligned}
\end{equation}
\end{proposition}
\begin{proof}
Since~$\mssm$ is absolutely $\mcE$-moderate, then
\begin{align}\label{eq:p:ConsistencyD:1}
\DzLocB{\mssm,\T}=\set{f\in \dotloc{\dom}(G_\bullet) \cap \Cb(\T) : \sq{f}\leq \mssm}
\end{align}
for any fixed~$G_\bullet\in\msG_0$ by Proposition~\ref{p:AriyoshiHino}. Since~$(X,\T)$ is locally compact Polish and~$\mssm$ is Radon, there exists~$G_\bullet'\in\msG_c$ consisting of relatively compact open sets, and thus~$\dotloc{\dom}(G_\bullet')=\domloc$. Together with~\eqref{eq:p:ConsistencyD:1}, this shows the first equality in~\eqref{eq:p:ConsistencyD:0}. The second is a straightforward consequence of the Markov property of~$(\mcE,\dom)$.
\end{proof}

The above proposition clarifies that our definition of intrinsic distance coincides with the usual one in all classical settings, including on Riemannian manifolds, as detailed in the next example.

\begin{example}[Riemannian manifolds] Let~$(M,g)$ be any smooth connected Riemannian manifold, and define the ---~regular strongly local~--- canonical Dirichlet form~$(\mcE,\dom)$ of~$(M,g)$ as the closure of the pre-Dirichlet form
\begin{align*}
\mcE(\phi,\psi)\eqdef \int_M g(\diff\phi,\diff\psi) \ \dvol_g\comm \qquad \phi,\psi\in \Czinfty(M)\comm
\end{align*}
see e.g.~\cite[Thm.~4.2]{AlbBraRoe89}. Then,
\begin{equation}\label{ese:Riemannian:3}
\begin{aligned}
\mssd_{\vol_g}(x,y)=& \ \sup\set{f(x)-f(y): f\in \domloc\cap \Cb(M)\comm g(\diff f,\diff f)\leq 1 \as{\vol_g}}
\\
=&\ \sup\set{f(x)-f(y): f\in \mcC^1_b(M)\comm g(\diff f,\diff f)\leq 1}
\\
=&\ \sup\set{f(x)-f(y): f\in \mcC^\infty_b(M)\comm g(\diff f,\diff f)\leq 1}
\end{aligned}
\end{equation}
and, if~$(M,\mssd_g)$ is additionally complete, then it holds as well that
\begin{align}
\label{ese:Riemannian:4}
\mssd_{\vol_g}(x,y)=&\ \mssd_g(x,y)\eqdef \inf \set{\ell(\gamma) : \gamma\in \mcC^1([0,1],M)\comm \gamma_0=x, \gamma_1=y} \fstop
\end{align}
Indeed: the first equality in~\eqref{ese:Riemannian:3} holds by Proposition~\ref{p:BoundedDist}, the third equality in~\eqref{ese:Riemannian:3} is straightforward, the second equality in~\eqref{ese:Riemannian:3} is claimed in~\cite[Rmk.~3, p.~1859]{Stu97}. A proof can be adapted from that of the same statement for Alexandrov spaces, in~\cite[Thm.~7.1]{KuwMacShi01}.
If~$\mssd_g$ is complete, then the equality of~\eqref{ese:Riemannian:3} and~\eqref{ese:Riemannian:4} is standard, e.g.~\cite[p.~151, Ex.~17]{Pet06}.
Despite a claim to this fact in~\cite[\S5.1.1 Prop.~3]{Stu97}, the latter equality may not hold if~$\mssd_g$ is not complete. Indeed,~\cite[\S5.1.1 Prop.~3]{Stu97} relies on~\cite[\S2.1.2 Prop.~2]{Stu97} ---~i.e.\ the equality~$\rho=\tilde\rho$ right before~\cite[Lem.~3.4]{DeCPal91}~--- which holds under the completeness assumption in~\cite[Ass.~(1.13) ii)]{DeCPal91}.
\end{example}

If~$(\mbbX,\mcE)$ is a quasi-regular strongly local Dirichlet space, the case~$\mu=\mssm$ is singled out as giving important information about the associated diffusion process. However, it is possible that~$\mssd_\mssm$ is identically vanishing. Let~$\mbbX$ be satisfying~\ref{ass:Hausdorff}. In the next example, we say that a measure~$\mu$ on~$(X,\Bo{\T})$ is \emph{nowhere-Radon} if~$\mu U=\infty$ for every non-empty open~$U\subset X$.

\begin{example}[Singular perturbations]\label{ese:SingularPerturb}
In~\cite{AlbMa91}, S.~Albeverio and Z.-M.~Ma construct a very large class of quasi-regular perturbations of Dirichlet forms by nowhere-Radon measure.
The domain of any such form contains no continuous function but the zero function, and thus~$\mssd_\mssm$ vanishes identically.

In particular, let~$(\mbbX,\mcE)$ be a regular strongly local Dirichlet space and assume that each singleton in~$X$ is $\mcE$-polar. In this case, for each $\mcE$-smooth~$\mu$, we may find an $\mcE$-smooth (infinite) measure~$\nu$ on~$(X,\Bo{\T})$, equivalent to~$\mu$ and additionally nowhere Radon,~\cite[Thm.~IV.4.7]{MaRoe92}.
\end{example}

\begin{example}[Liouville Brownian motion]\label{ese:LQG}
In~\cite{GarRhoVar14}, C.~Garban, R.~Rhodes, and V.~Vargas construct the regular strongly local Dirichlet form~$(\mcE,\dom)$ associated to the Liouville Brownian motion on a simply connected domain~$D\subset \R^2$,~\cite[Thm.~1.7]{GarRhoVar14}. The form, defined on the~$L^2$-space of the Liouville quantum gravity measure~$M$ on~$D$, has identically vanishing intrinsic distance~$\mssd_\mssm\eqdef \mssd_M$,~\cite[Prop.~3.1]{GarRhoVar14}.
On the other hand, letting~$\mu\eqdef \Leb^2\mrestr D$, the intrinsic distance~$\mssd_\mu$ satisfies
\begin{align*}
\mssd_\mu(x,y)=&\ \sup\set{f(x)-f(y): f\in \dom\comm \abs{\nabla f}^2\leq 1 \as{\mu}}
\\
\geq & \ \sup\set{f(x)-f(y): f\in\Cbinfty(D) \comm \abs{\nabla f}^2\leq 1}=\abs{x-y} \comm
\end{align*}
and is therefore non-trivial.
\end{example}

\begin{lemma}\label{l:ConvMeasure}
Let~$(\mbbX,\mcE)$ be a quasi-regular strongly local Dirichlet space, and~$\mu\in\Msp(\Bo{\T},\Ne{\mcE})$ be $\mcE$-moderate. Further let~$\ttseq{f_\alpha}_\alpha\subset{(\dotloc{\dom})}_b$ be so that
\begin{enumerate*}[$(i)$]
\item $\sup_\alpha\norm{f_\alpha}_{L^\infty(\mssm)} \leq r$ for some~$r>0$;
\item $\seq{f_\alpha}_\alpha\subset \DzLocB{\mu}$;
\item there exists~$f\in L^\infty(\mssm)$ so that~$f_\alpha \rar f$ weakly* in~$L^\infty(\mssm)$.
\end{enumerate*}
Then,~$f\in \DzLocB{\mu}$.
\end{lemma}
\begin{proof} The statement is well-posed by Proposition~\ref{p:EDominant}\iref{i:p:EDominant:2} since~$\mu\in\Msp(\Bo{\T}, \Ne{\mcE})$.
Since~$\car\in\dotloc{\dom}$ and~$\sq{\car}\equiv 0$ by~\eqref{eq:E(1)=0}, we may and shall assume with no loss of generality that~$f_\alpha\geq 0$ for every~$\alpha$.
Further note that we may arrange so that an $\mcE$-quasi-continuous representative~$\reptwo f_\alpha$ of~$f_\alpha$ satisfies~$0\leq \reptwo f_\alpha \leq r$ for every~$\alpha$.

\paragraph{Step 1} Let~$(G_\bullet, e_\bullet)\in\msG_0$ be $\mu$-moderated, and fix~$n\in \N$. By assumption~$\seq{f_\alpha}_\alpha$ converges to~$f$ weakly* in~$L^\infty(\mssm)$, thus~$\seq{e_n f_\alpha}_\alpha$ converges weakly* in~$L^2(\mssm)$ to~$e_n f$.
Reasoning as in~\eqref{eq:l:AriyoshiHino:1} with~$f_\alpha$ in place of~$f_n$ and~$e_n$ in place of~$g$, one has~$\sup_\alpha\mcE_1^{1/2}(e_n f_\alpha)<\infty$ and so~$\seq{e_n f_\alpha}_\alpha$ has a (non-relabeled) subnet weakly* convergent in~$\dom$. Since $\norm{\emparg}_{L^2(\mssm)}\leq \norm{\emparg}_{\dom}$, the subnet~$\seq{e_n f_\alpha}_\alpha$ converges weakly* in~$L^2(\mssm)$, and so it converges to the weak* limit~$e_n f$ of the net~$\seq{e_n f_\alpha}_\alpha$. Since the subnet was arbitrary, in fact the net~$\seq{e_n f_\alpha}_\alpha$ converges weakly* in~$\dom$ to~$e_n f$.
Since~$n\in \N$ was arbitrary,~$f\in\dotloc{\dom}$.

\paragraph{Step 2} Fix again~$n\in\N$ and note that~$e_n^2\in\domb$, since the latter is an algebra, and that~$\reptwo e_n^2 \in L^2(\mu)$ by interpolation, since~$\reptwo e_n\in L^2(\mu)\cap L^\infty(\mu)$.
Repeat \emph{verbatim} the reasoning in \emph{Step 1} with~$e_n^2$ in place of~$e_n$ to conclude that~$\seq{e_n^2 f_\alpha}_\alpha$ converges weakly* in~$\dom$ to~$e_n^2 f$ for every~$n\in \N$.
Since every norm-closed convex subset of a Hilbert space is weakly closed,~$e_n^2 f\in \cl_\dom \conv \set{e_n^2 f_\alpha}_\alpha$. Therefore, there exists~$\seq{g_m}_m$ so that each~$g_m$ is a convex combination of finitely many elements of~$\set{e_n^2 f_\alpha}_\alpha$ and~$\dom$-$\mlim g_m=e_n^2 f$.
Again by the Leibniz rule~\eqref{eq:Leibniz}, for arbitrary~$\alpha$,
\begin{align*}
\sq{e_n \cdot e_n f_\alpha}=&\reptwo e_n^2\, \sq{e_n f_\alpha}+\reptwo e_n^2 \reptwo f_\alpha^2 \, \sq{e_n} + 2\cdot \reptwo e_n^2\, \reptwo f_\alpha \,\sq{e_n, e_n f_\alpha} \comm
\\
\car_{G_n}\sq{e_n \cdot e_n f_\alpha}=&\car_{G_n} \reptwo e_n^2 \, \sq{e_n f_\alpha}+\car_{G_n}\reptwo e_n^2\, \reptwo f_\alpha^2\, \sq{e_n} + 2\cdot \car_{G_n}\, \reptwo e_n^2\, \reptwo f_\alpha\, \sq{e_n, e_n f_\alpha} = \car_{G_n} \sq{f_\alpha} \comm
\end{align*}
where the last equality holds by strong locality~\eqref{eq:SLoc:2}, since~$\reptwo e_n$ is constant $\mcE$-q.e.\ on the $\mcE$-quasi-open set~$G_n$. Thus,~$\car_{G_n} \sq{e_n^2 f_\alpha}\leq \car_{G_n}\mu$. By convexity of~$\emparg \mapsto \sq{\emparg}$ on~$\dom$, the same holds for~$g_m$ in place of~$e_n^2 f_\alpha$, that is
\begin{align}\label{eq:l:ConvMeasure:1}
\car_{G_n}\sq{g_m}\leq \car_{G_n} \mu\comm \qquad m\in \N\fstop
\end{align}

\paragraph{Step 3} Let~$V$ be $\mcE$-quasi-open with~$V\subset G_n$ $\mcE$-q.e., and~$F$ be $\mcE$-quasi-closed with~$F\subset V$ $\mcE$-q.e. By~\cite[Lem~3.5(i)]{Kuw98} there exists a sequence~$\seq{v_k}_k\subset \domb$ so that~$0\leq \reptwo v_k \leq \reptwo v_{k+1}\leq 1$ $\mcE$-q.e.,~$\klim \reptwo v_k=1$ $\mcE$-q.e.\ on~$F$, and~$\reptwo v_k=0$ $\mcE$-q.e.\ on~$V^\complement$ for every~$k\in \N$. Set~$\rep v\eqdef \sup_k \reptwo v_k$, satisfying~$0\leq \rep v\leq 1$ $\mcE$-q.e., and~$\rep v=0$ $\mcE$-q.e.\ on~$V^\complement$.
Since~$\reptwo v_k\leq \car_{G_n}$ $\mcE$-q.e., it follows from~\eqref{eq:l:ConvMeasure:1} that
\begin{align*}
\int \reptwo v_k \diff\sq{g_m}\leq \int \reptwo v_k \diff\mu \comm \qquad k,m\in \N\fstop
\end{align*}
By Lemma~\ref{l:Continuity}, we may take the limit as~$m\rar\infty$, to obtain
\begin{align*}
\int \reptwo v_k \diff\sq{e_n^2 f} \leq \int \reptwo v_k \diff\mu\comm \qquad k\in \N\comm
\end{align*}
and, by Monotone Convergence, the limit as~$k\rar\infty$, to obtain
\begin{align*}
\int_V \rep v\diff\sq{e_n^2 f}\leq \int_V \rep v \diff\mu\comm \qquad n\in \N\comm
\end{align*}
hence conclude from the properties of~$\rep v$ that
\begin{align}\label{eq:ConvMeasure:3}
\sq{e_n^2 f} F \leq \int_V \rep v \diff \sq{e_n^2 f} \leq& \ \mu F +\int_{V\setminus F} \rep v\diff\mu
\leq \mu F + \mu(V\setminus F)\comm \qquad n\in \N\fstop
\end{align}

On the other hand, since~$G_\bullet$ is $\mu$-moderated,~$\mu G_n$ is finite for every~$n\in \N$. In particular, the measures~$\car_{G_n}\mu$ and~$\car_{G_n}\sq{e_n^2 f}$ are totally finite measures.
Since~$(X,\T)$ is strongly Lindel\"of,~$G_n$ is strongly Lindel\"of as well, and therefore perfectly normal,~\cite[Ex.~3.8A(c), p.~194]{Eng89}. In particular, $G_n$~is hereditarily Lindel\"of and regular Hausdorff, thus both~$\car_{G_n}\mu$ and~$\car_{G_n}\sq{e_n^2 f}$ are inner regular with respect to\ closed sets by~\cite[Prop.~II.7.2.2(iv) and~(i)]{Bog07}, and \emph{a fortiori} inner regular with respect to\ $\mcE$-quasi-closed sets.
By arbitrariness of~$F$, and since~$V\subset G_n$ $\mcE$-q.e., we conclude from~\eqref{eq:ConvMeasure:3} that
\begin{align*}
\car_{G_n}\sq{e_n^2 f} V \leq \car_{G_n} \mu V \comm \qquad V\subset G_n \qe{\mcE}\comm V \text{~$\mcE$-quasi-open}\comm \qquad n\in \N\fstop
\end{align*}
By strong locality~\eqref{eq:SLoc:2},
\begin{align}\label{eq:ConvMeasure:4}
\car_{G_n}\sq{f} V= \car_{G_n}\sq{e_n^2 f} V \leq \car_{G_n} \mu V \comm \qquad V\subset G_n \qe{\mcE}\comm V \text{~$\mcE$-quasi-open}\comm \qquad n\in \N\fstop
\end{align}
Since both~$\car_{G_n}\sq{f}$ and~$\car_{G_n}\mu$ are \emph{finite} Borel measures, and since the family of $\mcE$-quasi-open subsets~$V$ of~$G_n$ is closed with respect to finite intersections and generates the Borel $\sigma$-algebra on~$G_n$, the inequality~\eqref{eq:ConvMeasure:4} suffices to establish that~$\car_{G_n}\sq{f}\leq \car_{G_n}\mu$ for every~$n\in \N$ by a standard monotone class argument, e.g.~\cite[Lem.~I.1.9.4]{Bog07}.
Finally, let~$A\in \Bo{\T}$.
Since both~$\sq{f}$ and~$\mu$ do not charge $\mcE$-polar sets, it follows from the monotonicity of both measures along the nested countable family~$A\cap G_n$ that
\begin{align*}
\sq{f} A=\nlim \car_{G_n} \sq{f} A \leq \nlim \car_{G_n}\mu A =\mu A \comm \qquad A\in\Bo{\T}\comm
\end{align*}
which concludes the proof.
\end{proof}

\begin{remark}\label{r:weak*}
Suppose~$\seq{f_n}_n\subset L^\infty(\mssm)$ with~$\sup_n \norm{f_n}_{L^\infty(\mssm)}\leq r$ for some~$r>0$. Further let~$f\in L^\infty(\mssm)$, and assume that~$\mssm$-a.e.-$\nlim f_n=f$.
Then,~$\nlim f_n=f$ weakly* in~$L^\infty(\mssm)$. Indeed, for every~$g\in L^1(\mssm)$ we have~$\nlim \mssm(f_n g)=\mssm(f g)$ by Dominated Convergence with dominating function~$r\abs{g}\in L^1(\mssm)$.
\end{remark}

\subsubsection{Invariance vs. accessibility}
\blue{Let us briefly recall the notion of $\mcE$-invariance of set~$A\subset X$, and its relation with $\mssd$-accessibility.}

\begin{definition}[$\mcE$-invariance]
Let~$(\mbbX,\mcE)$ be a Dirichlet space, and $\A$ be a $\sigma$-algebra satisfying~$\Bo{\T}\subset \A\subset \Bo{\T}^\mssm$. A set~$A\in\A$ is \emph{$\mcE$-invariant} if
\begin{align*}
T_t (\car_A f)= \car_A T_t f\comm \qquad f\in L^2(\mssm)\fstop
\end{align*}
\end{definition}

The following characterization of $\mcE$-invariance is standard. See~\cite[Thm.~1.6.1 and Cor.~4.6.3]{FukOshTak11} for the regular case. The quasi-regular case follows by the transfer method.

\begin{proposition}\label{p:InvarChar}
Let~$(\mbbX,\mcE)$ be a quasi-regular strongly local Dirichlet space, and~$A\in\Bo{\T}$. Then, the following are equivalent:
\begin{enumerate}[$(a)$]
\item\label{i:p:InvarChar:1} $A$ is $\mcE$-invariant;
\item\label{i:p:InvarChar:2} $A^\complement$ is $\mcE$-invariant;
\item\label{i:p:InvarChar:3} $\car_A f\in\dom$ and~$\mcE(f)=\mcE(\car_A f)+\mcE(\car_{A^\complement} f)$ for every~$f\in \dom$;
\item\label{i:p:InvarChar:4} $\car_A f\in\domext$ and~$\mcE(f)=\mcE(\car_A f)+\mcE(\car_{A^\complement} f)$ for every~$f\in \domext$;
\item\label{i:p:InvarChar:5} there exists~$\tilde A\subset X$ so that~$\mssm (A\triangle \tilde A)=0$ and~$\tilde A$ is both $\mcE$-quasi-open and $\mcE$-quasi-closed.
\end{enumerate}
\end{proposition}

\blue{
\begin{lemma}\label{l:InvarSq}
Let~$(\mbbX,\mcE)$ be a quasi-regular strongly local Dirichlet space, and~$A\in\Bo{\T}$ be $\mcE$-invariant. Then,~$\car_A\in\dotloc{\dom}$ and~$\sq{\car_A}\equiv 0$.
\begin{proof}
Let~$(G_\bullet, e_\bullet)\in\msG_0$. 
Since~$e_n\in\domb$ for each~$n\in\N$, and since~$A$ is $\mcE$-invariant, we have that~$e^A_n\eqdef \car_A e_n\in \domb$ by Proposition~\ref{p:InvarChar}\iref{i:p:InvarChar:3}, and~$e^A_n= 1$ $\mssm$-a.e.\ on~$G_n\cap  A$ and~$e^A_n=0$ $\mssm$-a.e.\ on~$G_n\cap A^\complement$, hence~$e^A_n= \car_A$ $\mssm$-a.e.\ on~$G_n$.
Thus,~$(G_\bullet,e^A_\bullet)\in\msG_0$ witnesses that~$\car_A\in\dotloc{\dom}$.
Let~$\tilde A$ be the $\mcE$-quasi-open and $\mcE$-quasi-closed $\mssm$-version of~$A$ provided by Proposition~\ref{p:InvarChar}\iref{i:p:InvarChar:5}.
Since~$\tilde A\cup \tilde A^\complement=X$, we have that~$\car_{\tilde A}+\car_{\tilde A^\complement}\equiv\car$, and combining~\eqref{eq:SLoc:2} and~\eqref{eq:E(1)=0} readily yields that~$\sq{\car_A}=\sq{\car_{\tilde A}}\equiv 0$.
\end{proof}
\end{lemma}
}

\begin{proposition}\label{p:Invariant}
Let~$(\mbbX,\mcE)$ be a quasi-regular strongly local Dirichlet space, $\mu\in\Msp(\Bo{\T})$, and~$A=B^{\mssd_\mu}_\infty(x_0)$ be the $\mssd_\mu$-accessible component of~$x_0$ for some~$x_0\in X$. Then, the following are equivalent:
\begin{enumerate}[$(a)$]
\item\label{i:p:Invariant:1} $A$ is $\mcE$-invariant;
\item\label{i:p:Invariant:2} $\rep\rho_A\colon y\mapsto \mssd_\mu(y, A)\wedge 1$ satisfies~$\rho_A\in \DzLocB{\mu}$.
\end{enumerate}
\end{proposition}
\begin{proof}
Since~$\mssd_\mu$ is $\T^\tym{2}$-l.s.c., $\mssd_\mu$-accessible components are $\Bo{\T}$-measurable, and the statement is well-posed. Further note that~$\rep\rho_A=\car_{A^\complement}$.
Assume~\iref{i:p:Invariant:1}. 
Since~$A$ is $\mcE$-invariant, so is~$A^\complement$.
By Lemma~\ref{l:InvarSq} applied to~$A^\complement$, we have that~$\rho_A=\car_{A^\complement}\in \dotloc{\dom}$ and~$\sq{\rho_A}=0$, and therefore~$\rho_A\in\DzLocB{\mu}$.
Assume~\iref{i:p:Invariant:2}. Since~$\rho_A\in\DzLocB{\mu}$, it admits an $\mcE$-quasi-continuous representative~$\reptwo\rho_A$. By $\mcE$-quasi-continuity,~$\tilde A\eqdef \reptwo\rho_A^{-1}\ttonde{(-1,1)}=\reptwo\rho_A^{-1}(\set{0})$ is both $\mcE$-quasi-open and $\mcE$-quasi-closed. Since~$\reptwo\rho_A=\rep\rho_A$ $\mssm$-a.e.,~$\mssm(\tilde A\triangle A)=0$, and therefore~$A$ is $\mcE$-invariant by Proposition~\ref{p:InvarChar}.
\end{proof}

For an extended pseudo-distance~$\mssd$ on~$X$, $\mssd$-accessible components may range within the two extrema, there being spaces all of which $\mssd$-accessible components are $\mssm$-negligible. A meaningful example is as follows.

\begin{example}[Configuration spaces I]\label{ese:Config1}
Let~$\Upsilon(\R^d)$ denote the configuration space over~$\R^d$, endowed with the vague topology, the induced Borel $\sigma$-algebra, and the Poisson measure~$\pi_d$ with intensity the Lebesgue measure on~$\R^d$. Since~$\Upsilon(\R^d)$ is Polish, the probability measure~$\pi_d$ is Radon.
The canonical Dirichlet form~$(\mcE,\dom)$ on~$L^2(\Upsilon(\R^d))$ is the quasi-regular strongly local Dirichlet form constructed in~\cite{AlbKonRoe98}. It was shown in~\cite[Thm.~1.5(ii)]{RoeSch99} that the intrinsic distance~$\mssd_{\pi_d}$ coincides with the $L^2$-transportation extended distance~$W_2$.
It is not difficult to show, by translation-invariance of the Lebesgue measure and standard properties of Poisson measures, that all $W_2$-accessible components are $\pi_d$-negligible.
In particular, for any such component~$A$, we have that~$W_2(\emparg, A)\equiv +\infty$ $\pi_d$-a.e., hence~$W_2(\emparg, A)\wedge 1\equiv \car$ $\pi_d$-a.e.
\end{example}

\section{The Rademacher property}
Let us introduce the first property of our interest. Write
\begin{align*}
\Lipu(\mssd)\eqdef& \tset{\rep f\in\Lip(\mssd) : \Li[\mssd]{\rep f}\leq 1} \comm & \bLipu(\mssd)\eqdef& \tset{\rep f\in\bLip(\mssd) : \Li[\mssd]{\rep f}\leq 1} \comm
\\
\Lipu(\mssd,\A)\eqdef& \tset{\rep f\in\Lip(\mssd,\A) : \Li[\mssd]{\rep f}\leq 1}\comm & \bLipu(\mssd,\A)\eqdef& \tset{\rep f\in\bLip(\mssd,\A) : \Li[\mssd]{\rep f}\leq 1} \fstop
\end{align*}

\blue{
In the next definition, and in some of the results of this section, we will be concerned with Dirichlet spaces that are not necessarily quasi-regular.
In fact, it will be shown in Proposition~\ref{p:Consistency} below, that the Rademacher property for a Dirichlet space~$(\mbbX,\mcE)$ ---~together with some additional assumptions~--- \emph{implies} the quasi-regularity of~$(\mbbX,\mcE)$.
}

\begin{definition}[Rademacher]\label{d:Rad}
Let~$\mbbX$ be satisfying~\ref{ass:Luzin},~$(\mbbX,\mcE)$ be a strongly local Dirichlet space, $\mu\in\Msp(\Bo{\T})$, and~$\mssd\colon X^\tym{2}\rar [0,\infty]$ be an extended pseudo-distance on~$X$. 
We say that~$(\mbbX,\mcE,\mssd,\mu)$ has:
\begin{itemize}
\item the \emph{Rademacher property} if
\begin{align}\label{eq:Rad}
\tag{$\Rad{\mssd}{\mu}$} \rep f\in \Lipu(\mssd,\A) \comm f\in\dotloc{L^\infty(\mssm)} \qquad \implies \qquad f\in \DzLoc{\mu}\semicolon
\end{align}
\item the \emph{bounded-support Rademacher property} if
\begin{align}\label{eq:bsRad}
\tag{$\Rad{\mssd}{\mu}^{bs}$} \rep f\in \Lipu(\mssd,\A) \comm f\in\dotloc{L^\infty(\mssm)} \comm  \textrm{$\supp[f]$ $\mssd$-bounded} \qquad \implies \qquad f\in \DzLoc{\mu}\semicolon
\end{align}
\item the \emph{distance-Rademacher property} if
\begin{align}\label{eq:dRad}
\tag{$\dRad{\mssd}{\mu}$} \mssd\leq \mssd_\mu\fstop
\end{align}
\end{itemize}
\end{definition}

It will be apparent from the proof of Theorem~\ref{t:Lenz} that we might equivalently define
\begin{align}\tag{$\Rad{\mssd}{\mu}^b$}
\rep f\in \bLipu(\mssd,\A) \qquad \implies \qquad f\in \DzLocB{\mu} \fstop
\end{align}

\begin{remark}\label{r:RadBS}
The bounded-support Rademacher property, together with Proposition~\ref{p:BoundedSupp} below, is introduced for comparison with~\cite{AmbGigSav15}. It is clear that~\eqref{eq:Rad} implies~\eqref{eq:bsRad}. The converse implication does not hold in general. Furthermore~\eqref{eq:bsRad} might be trivial, even if~\eqref{eq:Rad} is not.
Example~\ref{ese:Config2} below provides a quadruple~$(\mbbX,\mcE,\mssd_\mssm,\mssm)$ on which every $\mssd_\mssm$-Lipschitz function with bounded support is $\Bo{\T}^\mssm$-measurable and coincides $\mssm$-a.e.\ with the $\zero$-function, so that~\eqref{eq:bsRad} trivially holds.
\end{remark}

\begin{remark}
In the discussion of any Rademacher property for extended metric spaces, the $\A$-measurability of~$\rep f\in \Lipu(\mssd)$ is essential, as shown by the next example. Furthermore, the property \emph{depends} on the chosen $\sigma$-algebra.
\end{remark}

\begin{example}[Non-measurable Lipschitz functions]\label{ese:LipNonMeas}
Let~$\mbbX$ be satisfying~\ref{ass:Luzin}, and~$\mssd\colon X^\tym{2}\rar [0,\infty]$ be an extended distance on~$X$ with uncountably many $\mssd$-accessible components.
Assume that every $\mssd$-accessible component is $\mssm$-negligible. Then there exists a bounded $\mssd$-Lipschitz function~$\rep f$ that is not $\Bo{\T}^\mssm$-measurable. In particular, there exists no $\Bo{\T}$- (or even~$\Bo{\T}^\mssm$-) measurable function~$\rep g$ with~$\rep f=\rep g$ $\mssm$-a.e.
\end{example}
\begin{proof}
Let~$\msN$ be the $\sigma$-ideal of $\mssm$-negligible sets and~$\msA$ the family of $\mssd$-accessible components.
By assumption,~$\msA\subset \msN$ and~$\msA$ is a partition of~$X$ (i.e.\ sets in~$\msA$ are pairwise disjoint and~$\cup \msA=X$). In particular,~$\cup \msA\notin \msN$.
By the Four Poles Theorem~\cite{CicMorRalRyl07} applied to~$\msA$ and~$\msN$, there exists a subfamily~$\msA'\subset \msA$ so that~$A\eqdef \cup \msA'$ is not $\Bo{\T}^\mssm$-measurable.
By definition of~$\msA$ we have~$\mssd(x,y)=+\infty$ for every~$x\in A$ and~$y\in A^\complement$.
As a consequence, the characteristic function~$\car_A$ is $\mssd$-Lipschitz.
\end{proof}
Note that the configuration spaces discussed in Example~\ref{ese:Config1} satisfy the previous assumptions.

We postpone a proof of the next proposition until later in this section.

\begin{proposition}\label{p:BoundedSupp}
Let~$\mbbX$ be satisfying~\ref{ass:Luzin},~$(\mbbX,\mcE)$ be a strongly local Dirichlet space, $\mu\in\Msp(\Bo{\T})$, and~$\mssd\colon X^\tym{2}\rar [0,\infty]$ be a $\T^\tym{2}$-continuous extended pseudo-distance on~$X$.
Then,~$(\mbbX,\mcE,\mssd,\mu)$ possesses~\eqref{eq:bsRad} if and only if it possesses~\eqref{eq:Rad}.
\end{proposition}

The interplay between~\eqref{eq:Rad} and~\eqref{eq:dRad} is discussed in the next lemma.

\begin{lemma}\label{l:RadDRad}
Let~$(\mbbX,\mcE)$ be a strongly local Dirichlet space,~$\mu\in\Msp(\Bo{\T},\Ne{\mcE})$, and~$\mssd\colon X^\tym{2}\rar [0,\infty]$ be a $\T$-admissible extended pseudo-distance on~$X$ in the sense of Definition~\ref{d:AES}.
If~$(\mbbX,\mcE,\mssd,\mu)$ possesses~$(\Rad{\mssd}{\mu})$, then it possesses~$(\dRad{\mssd}{\mu})$.
\end{lemma}
\begin{proof}
By Definition~\ref{d:AES} and Remark~\ref{r:AES}, there exists a family~$\set{\mssd_\alpha}_{\alpha\in \mcA}$ of $\T^\tym{2}$-\emph{continuous} bounded pseudo-distances $\mssd_\alpha\leq \mssd$ so that~$\mssd=\lim_\alpha \mssd_\alpha$. Furthermore,~$\abs{\mssd_\alpha(x,y)-\mssd_\alpha(x,z)}\leq \mssd_\alpha(y,z)\leq \mssd(y,z)$, thus~$\mssd_\alpha(x,\emparg)$ is $\mssd$-Lipschitz and so~$\mssd_\alpha\in \DzLocB{\mu,\T}$ by assumption. Therefore,
\begin{align*}
\mssd(x,y)=&\ \mssd(x,y)-\mssd(x,x)=\lim_\alpha \mssd_\alpha(x,y)-\mssd_\alpha(x,x)\leq \sup\tset{f(y)-f(x): f\in \DzLocB{\mu,\T}}\\=&\ \mssd_\mu(x,y) \qedhere \fstop
\end{align*}
\end{proof}

\blue{
As a first application of the Rademacher property, let us show how it implies to the completeness of intrinsic distances.
\begin{proposition}\label{p:RadCompleteness}
Let~$(\mbbX,\mcE)$ be a quasi-regular strongly local Dirichlet space,~$\mssd\colon X^\tym{2}\rar [0,\infty]$ be an extended pseudo-distance on~$X$ and $\mu\in\Msp(\Bo{\T},\Ne{\mcE})$ be $\mcE$-moderate.
Further assume that~$(\mbbX,\mcE)$ possesses~$(\dRad{\mssd}{\mu})$.
Then, the following assertions hold:
\begin{enumerate}[$(i)$]
\item if~$(X,\T,\mssd)$ is an extended metric-topological space (Dfn.~\ref{d:AES}), then so is~$(X,\T,\mssd_\mu)$;
\item if~$(X,\T,\mssd)$ is additionally complete, then so is~$(X,\T,\mssd_\mu)$.
\end{enumerate}

\begin{proof}
The $\T$-admissibility of~$\mssd_\mu$ was already noted after~\eqref{eq:IntrinsicD}.
Since~$\mssd\leq \mssd_\mu$ and~$\mssd$ separates points by assumption, we conclude that~$\mssd_\mu$ too separates points, hence~$(X,\T,\mssd_\mu)$ is an extended metric-topological space.
In order to show completeness, let~$\seq{x_n}_n$ be a $\mssd_\mu$-fundamental sequence.
Since~$\mssd\leq \mssd_\mu$, the sequence~$\seq{x_n}_n$ is as well $\mssd$-fundamental.
By completeness of~$(X,\mssd)$, it $\mssd$-converges to a limit point~$x\in X$.
Since~$(X,\T,\mssd)$ is an extended metric-topological space, $\nlim \mssd(x_n,x)=0$ implies that $\T$-$\nlim x_n=x$.
Fix~$\eps>0$, and let~$n_\eps\in\N$ be so that~$\mssd_\mu(x_n,x_m)<\eps$ for $m,n\geq n_\eps$.
Then, by $\T$-lower semi-continuity of~$\mssd_\mu$,
\begin{align*}
\mssd_\mu(x_n,x)\leq \mliminf \mssd_\mu(x_n,x_m) <\eps\comm \qquad n\geq n_\eps\fstop
\end{align*}
This shows that $\seq{x_n}_n$ $\mssd_\mu$-converges to~$x$ as well, hence~$(X,\mssd_\mssm)$ is a complete extended metric space.
\end{proof}
\end{proposition}
}

\medskip

Let us now provide some examples.

\begin{example}[Intrinsic distances of perturbed forms]\label{ese:Perturbed}
Let~$(\mbbX,\mcE)$ be a quasi-regular strongly local Di\-richlet space, and~$\mu\in\mcS$ be $\mcE$-smooth in the sense of Definition~\ref{d:Smooth}. 
Further let~$\msG^\mu$, resp.~$\msG_0^\mu$,~$\msG_c^\mu$, be defined as in~\eqref{eq:Xi0}, resp.~\eqref{eq:G0}, with~$\mcE^\mu$ in place of~$\mcE$.
By~\eqref{eq:p:Smooth:1}, we have that~$\msG^\mu=\msG$,~$\msG_0^\mu\subset \msG_0$, hence~$\msG_c^\mu\subset \msG_c$ as well.
Because of the inclusion of sets~$\dom^\mu\hookrightarrow\dom$, we have as well the inclusions of sets~$\dotloc{(\dom^\mu)}(G_\bullet)\hookrightarrow \dotloc{\dom}(G_\bullet)$, for every~$G_\bullet\in\msG^\mu=\msG$, and therefore~$\dotloc{(\dom^\mu)}\hookrightarrow \dotloc{\dom}$ by~\eqref{eq:DotLocGc}.
In particular,
\begin{align}\label{eq:DotLocMu}
\tset{f\in\dotlocb{(\dom^\mu)}: \sq{f}\leq \mu}\subset \DzLocB{\mu} \fstop
\end{align}

Now, let~$(\mcE^\mu,\dom^\mu)$ be the $\mu$-perturbed form defined in~\eqref{eq:Perturbed}, and denote by~$\sq{f}^{\mcE^\mu}$ the $\mcE^\mu$-energy measure of~$f\in\dom^\mu$.
By~\eqref{eq:Representation} and definition of~$(\mcE^\mu,\dom^\mu)$, we have that
\begin{align}\label{eq:EnergyMu}
\sq{f}^{\mcE^\mu}=\sq{f}+2\reptwo f^2\cdot \mu\comm
\end{align}
which we may extend to~$f\in\dotloc{(\dom^\mu)}$ by~\eqref{eq:SLoc:2} with~$\sq{f}$ being replaced by~$\sq{f}^{\mcE^\mu}$. Combining~\eqref{eq:DotLocMu} and~\eqref{eq:EnergyMu}, we thus have
\begin{align*}
\DzLocB{\mcE^\mu,\mu}\eqdef \tset{f\in\dotlocb{(\dom^\mu)}: \sq{f}^{\mcE^\mu} \leq \mu}\subset \DzLocB{\mu}\comm
\end{align*}
and therefore, letting~$\mssd^{\mcE^\mu}_\mu$ be the intrinsic metric of~$(\mcE^\mu,\dom^\mu)$, one has~$\mssd^{\mcE^\mu}_\mu\leq \mssd_\mu$, that is, the distance-Rademacher property holds for~$(\mbbX,\mcE,\mssd^{\mcE^\mu}_\mu,\mu)$.
\end{example}

\begin{example}[Rademacher property for extended distances]
In the case when the intrinsic distance~$\mssd_\mssm$ is extended, meaningful examples of Dirichlet spaces satisfying~$(\Rad{\mssd_\mssm}{\mssm})$ typically have infinite-dimensional underlying space~$\mbbX$. Examples include: configuration spaces, see Examples~\ref{ese:Config1}, \ref{ese:Config2}, \ref{ese:Config3}, \ref{ese:Config4}; Wiener spaces \cite{EncStr93}; and locally convex riggings of normed spaces~\cite{BogMay96}.
\end{example}

\begin{example}[Metric Measure Spaces]\label{ese:MMS}
A wide class of examples of Dirichlet spaces satisfying the Rademacher property is given by metric measure spaces~$(X,\mssd,\mssm)$.

A triple~$(X,\mssd,\mssm)$ is a \emph{metric measure space} if~$(X,\mssd)$ is a complete and separable metric space, and~$\mssm$ is a Borel measure on~$(X,\mssd)$ with full support and finite on $\mssd$-balls.
For a function~$f\in \Lip(\mssd)$, define the slope of~$f$ at~$x$ by
\begin{align*}
\slo{f}(x)\eqdef \limsup_{y\rar x} \frac{\abs{f(y)-f(x)}}{\mssd(x,y)}\comm \qquad x\in X\comm
\end{align*}
where, conventionally,~$\slo{f}(x)=0$ if $x$ is isolated.
The \emph{Cheeger energy}~\cite[Eqn.~(4.11)]{AmbGigSav14} on a metric measure space~$(X,\mssd,\mssm)$ is the functional
\begin{align*}
\Ch[\mssd,\mssm](f)\eqdef \inf\set{\nliminf \int \slo{f_n}^2\diff\mssm : f_n\in \Lip(\mssd)\comm L^2(\mssm)\text{-}\nlim f_n=f} \comm \qquad \inf\emp\eqdef+\infty\fstop
\end{align*}
A metric measure space~$(X,\mssd,\mssm)$ is called \emph{infinitesimally Hibertian} if~$\Ch[\mssd,\mssm]$ is quadratic, in which case, it is a strongly local Dirichlet form, satisfying~$(\Rad{\mssd}{\mssm})$ by construction.
\end{example}

\subsection{Sufficient conditions} Under the assumption of strong locality, the next statement is an extension to the quasi-regular Dirichlet spaces of~\cite[Thm.~4.9]{FraLenWin14}, the proof of which we adapt to our setting. Concerning the assumptions, see~\cite[Dfn.~4.1, Rmk.~4.2]{FraLenWin14}.

\begin{theorem}\label{t:Lenz}
Let~$\mbbX$ be satisfying~\ref{ass:Luzin},~$(\mbbX,\mcE)$ be a quasi-regular strongly local Dirichlet space, $\mssd\colon X^\tym{2}\rar [0,\infty]$ be an extended pseudo-distance on~$X$, and $\mu\in\Msp(\Bo{\T},\Ne{\mcE})$ be $\mcE$-moderate.
Further assume that:
\begin{enumerate}[$(a)$]
\item\label{i:t:Lenz:1} $\mssd(\emparg, A)$ is $\A$-measurable for every~$A\in\A$;
\item\label{i:t:Lenz:2} $\mssd(\emparg, A)\wedge r\in\DzLocB{\mu}$ for every~$A\in\A$ and every~$r>0$.
\end{enumerate}

Then,~$(\mbbX,\mcE,\mssd,\mu)$ possesses~$(\Rad{\mssd}{\mu})$.
\end{theorem}

Before proving the Theorem, we collect some remarks on the assumptions.
\begin{remark}\label{r:Lenz} Note that:
\begin{enumerate}[$(i)$]
\item\label{i:r:Lenz:1} If~$\T_\mssd$ is separable, then~\iref{i:t:Lenz:1}, \iref{i:t:Lenz:2} in Theorem~\ref{t:Lenz} may be respectively substituted by:
\begin{enumerate}[$(a')$]
\item\label{i:r:Lenz:1.1} $\mssd(\emparg, x_0)$ is $\A$-measurable for every~$x_0\in X$;
\item\label{i:r:Lenz:1.2} $\mssd(\emparg, x_0)\wedge r\in\DzLocB{\mu}$ for every~$x_0\in X$ and every~$r>0$.
\end{enumerate}
\begin{proof}
Let~$A\subset X$ be any subset. Since~$\T_\mssd$ is separable, there exists a countable set~$\set{x_i}_i$ $\T_\mssd$-dense in~$A$. Therefore,
\begin{align}\label{eq:r:Lenz:1:1}
\mssd(\emparg, A)\wedge r=\mssd(\emparg, \cl_\mssd A)\wedge r=\nlim \min_{i\leq n} \mssd(\emparg, x_i) \wedge r\comm \qquad r>0 \fstop
\end{align}

Thus,~$\mssd(\emparg, A)$ is the limit of a sequence of $\A$-measurable functions, by~\iref{i:r:Lenz:1.1}, and therefore it is $\A$-measurable. Furthermore,~$\min_{i\leq n} \mssd(\emparg, x_i) \wedge r\in \DzLocB{\mu}$ by~\iref{i:r:Lenz:1.2} and~\eqref{eq:TruncationLoc}, and thus we conclude~\iref{i:t:Lenz:2} by~\eqref{eq:r:Lenz:1:1}, Remark~\ref{r:weak*}, and Lemma~\ref{l:ConvMeasure}.
\end{proof}

\item\label{i:r:Lenz:2} If~$\mssd$ is $\T^\tym{2}$-continuous, then~$\T_\mssd$ is separable (since~$\T$ is so by~\ref{ass:Luzin}), and it suffices to assume~\iref{i:r:Lenz:1.2};

\item\label{i:r:Lenz:3.0} If we substitute~\iref{i:t:Lenz:1},~\iref{i:t:Lenz:2} with the stronger assumptions:
\begin{enumerate}[$(a'')$]
\item\label{i:r:Lenz:3} $\mssd(\emparg, A)$ is $\A$-measurable for every~$A\subset X$ (possibly: $A\not\in\A$);
\item\label{i:r:Lenz:4} $\mssd(\emparg, A)\wedge r\in\DzLocB{\mu}$ for every~$A\subset X$;\end{enumerate}
then we may relax the assumptions on~$\rep f$ in~$(\Rad{\mssd}{\mu})$ in that we do not need to assume \emph{a priori} that~$\rep f$ be $\A$-measurable. Indeed, it is shown in the proof that, under~\iref{i:r:Lenz:3},~\iref{i:r:Lenz:4}, there exists~$\reptwo f\in \DzLoc{\mu}$ with~$\reptwo f=\rep f$ $\mcE$-quasi-everywhere.
\end{enumerate}
\end{remark}

\begin{proof}[Proof of Theorem~\ref{t:Lenz}]
Let~$\rep f\in \Lip(\mssd,\A)$ with~$\Li[\mssd]{\rep f}\leq 1$.
By Lemma~\ref{l:FrankLenz} we may and shall assume with no loss of generality that~$0\leq f\leq r$ $\mssm$-a.e.\ for some~$r>0$.
For~$s\geq 0$ set~$A_s(\rep f)\eqdef \ttset{\rep f\geq s}$.
Let~$\rep d_{m,n}\eqdef \mssd\ttonde{\emparg, A_{m/n}(\rep f)}$ and note that~$\rep d_{m,n} \wedge r\in\DzLocB{\mu}$ for every~$r>0$ by~\iref{i:t:Lenz:2}. 
Further set
\begin{align*}
\rep f_n&\colon x\longmapsto \max_{1\leq m\leq n r} \tonde{\tfrac{m}{n}- \rep d_{m,n}(x)}_+\comm 
\qquad n\in \N \fstop
\end{align*}
By the truncation property~\eqref{eq:TruncationLoc} we have~$f_n \in \DzLocB{\mu}$ for every~$n\in \N$.

\blue{
Since~$\rep f$ is $\mssd$-Lipschitz, we have that, for every~$n\in N$, every~$m\leq n$, every~$x\in A_{m/n}(\rep f)^\complement$, and every~$m'$ with~$1\leq m'\leq nr$,
\begin{align*}
\rep d_{m',n}(x)\eqdef \inf_{y\in A_{m'/n}(\rep f)} \mssd(x,y) \geq \inf_{y\in A_{m'/n}(\rep f)} \rep f(y) - \rep f(x) > \inf_{y\in A_{m'/n}(\rep f)} \rep f(y) - \tfrac{m}{n} \fstop
\end{align*}
As a consequence,
\begin{align*}
\tfrac{m'}{n} - \rep d_{m',n}(x) < \tfrac{m'}{n}+\tfrac{m}{n}-\inf_{y\in A_{m'/n}(\rep f)} \rep f(y) \leq \tfrac{m}{n}\comm
\end{align*}
and therefore, maximizing over~$m'$,
\begin{align}\label{eq:t:Lenz:0.1}
\rep f_n(x)\eqdef \max_{1\leq m\leq nr} \tonde{\tfrac{m'}{n}- \rep d_{m',n}(x)}_+ \leq \tfrac{m}{n} \comm \qquad x\in  A_{m/n}(\rep f)^\complement \comm \qquad n\in \N\fstop
\end{align}

In a similar way, one can show that
\begin{align}\label{eq:t:Lenz:0.2}
\rep f_n(x) \geq \tfrac{m-1}{n}\comm \qquad x\in A_{(m-1)/n}(\rep f)\comm \qquad n\in \N\fstop
\end{align}
}

Combining~\eqref{eq:t:Lenz:0.1} and~\eqref{eq:t:Lenz:0.2},
\begin{align}\label{eq:t:Lenz:1}
\tfrac{m-1}{n} \leq \rep f_n(x) \leq \tfrac{m}{n}\comm \qquad x\in A_{(m-1)/n}(\rep f)\setminus A_{m/n}(\rep f)\comm \qquad n\in \N\fstop
\end{align}
It thus follows that~$\rep f_n$ converges to~$\rep f$ pointwise on~$X$.
In particular,~$\rep f_n$ converges to~$\rep f$ $\mssm$-a.e.\ on~$X$, hence~$\nlim f_n=f$ weakly* in~$L^\infty(\mssm)$ by Remark~\ref{r:weak*}.
By Lemma~\ref{l:ConvMeasure} this concludes the proof.
\end{proof}

The assumptions in Theorem~\ref{t:Lenz} are usually difficult to check.
It is however worth to spell out one result in the case when~$\mssd=\mssd_\mu$. This was shown by K.~Kuwae in~\cite{Kuw96} for not necessarily local forms assuming the $\T$-continuity of~$\mssd_\mu$.
We adapt the proof of~\cite{Kuw96} to our more general definition of intrinsic distance, postponing a thorough comparison with~\cite{Kuw96} to Remark~\ref{r:Kuw96} below.

\begin{theorem}\label{t:KuwaeProposition}
Let~$(\mbbX,\mcE)$ be a quasi-regular strongly local Dirichlet space,~$\mu\in\Msp(\Bo{\T}, \Ne{\mcE})$ be $\mcE$-moderate. Further assume that~$\T_{\mssd_\mu}$ is separable. Then, the conditions~\iref{i:r:Lenz:3} and~\iref{i:r:Lenz:4} in Remark~\ref{r:Lenz}\ref{i:r:Lenz:3.0} hold for~$\mssd_\mu$.
\end{theorem}
\begin{proof}
We show the statement for~$A=\set{x_0}$. The assertion for~$\mssd_\mu(\emparg, A)$ with arbitrary~$A\subset X$ follows by Remark~\ref{r:Lenz}\iref{i:r:Lenz:1}. For fixed~$x_0\in X$ and~$r>0$, set~$\rep\rho_{x_0}\eqdef\mssd_\mu(\emparg, x_0)\wedge r$.

\paragraph{Step 1} Suppose first that~$\mssd_\mu$ is everywhere finite. Then~$(X,\mssd_\mu)$ is a separable pseudo-metric space, and therefore it is second countable.
In particular, there exists a countable set~$\set{y_i}_i$ so that, setting~$B_{n,i}\eqdef B^{\mssd_\mu}_{1/n}(y_i)$, then~$\set{B_{n,i}}_i$ is a $\T_{\mssd_\mu}$-open covering of~$X$ for every~$n\in \N$. 
By definition of~$\mssd_\mu$, for every~$n$, $i\in \N$ and every fixed~$x\in X$ there exists~$f_{n,i,x}\in \DzLoc{\mu,\T}$ so that
\begin{align}
\label{eq:p:Kuw:1}
f_{n,i,x}(x)-f_{n,i,x}(y_i)\geq&\ \mssd_\mu(x,y_i)-\tfrac{1}{n}\comm && n,i\in \N\comm \qquad x\in X\comm
\\
\label{eq:p:Kuw:1.5}
f_{n,i,x}(y)\geq&\ f_{n,i,x}(x)-\mssd_\mu(x,y) \comm && n,i\in \N\comm \qquad x,y\in X\comm
\\
\label{eq:p:Kuw:2}
f_{n,i,x}(y)\leq&\ f_{n,i,x}(y_i)+\tfrac{1}{n} \comm && n,i\in \N\comm \qquad y\in B_{n,i}\fstop
\intertext{
Combining~\eqref{eq:p:Kuw:1} and~\eqref{eq:p:Kuw:2} with the triangle inequality
}
\label{eq:p:Kuw:3}
\mssd_\mu(x,y_i)\geq&\ \mssd_\mu(x,y)-\tfrac{1}{n}\comm && n,i\in \N \comm \qquad x\in X\comm y\in B_{n,i}\comm
\intertext{
yields
}
\nonumber
f_{n,i,x}(y)\leq&\ f_{n,i,x}(x)-\mssd_\mu(x,y)+\tfrac{3}{n} \comm && n,i\in \N\comm \qquad x\in X\comm y\in B_{n,i}\fstop
\end{align}

Now, let~$g_{n,i,x}\colon y\mapsto 0\vee \ttonde{f_{n,i,x}(x)-f_{n,i,x}(y)} \wedge r$, and note that $g_{n,i,x}\in \DzLoc{\mu,\T}$ by \eqref{eq:TruncationLoc} for every~$n$, $i\in \N$, and every~$x\in X$, and that
\begin{align}
\label{eq:p:Kuw:4}
0\leq g_{n,i,x}(y) \leq&\ \mssd_\mu(x,y)\wedge r\comm && n, i\in \N\comm \qquad x,y\in X\comm
\\
\label{eq:p:Kuw:5}
g_{n,i,x}(y)\geq&\ \ttonde{\mssd_\mu(x,y)-\tfrac{3}{n}}\wedge r\comm && n, i\in \N\comm \qquad x\in X\comm y\in B_{n,i}\comm
\\
\nonumber
\abs{g_{n,i,x}(y)-g_{n,i,x}(z)}\leq&\ \mssd_\mu(y,z)\comm && n, i\in \N\comm \qquad y,z\in X\fstop
\end{align}

Let~$\rho_{n,m,x_0}\colon y\mapsto \max_{i\leq m} g_{n,i,x_0}(y)$, and note that~$\rho_{n,m,x_0}\in \DzLoc{\mu,\T}$ by~\eqref{eq:TruncationLoc} for~$n,m\in \N$ and every~$x_0\in X$. Set~$\rep\rho_{n,x_0}\eqdef \mlim \rho_{n,m,x_0}$. 
By Remark~\ref{r:weak*} and Lemma~\ref{l:ConvMeasure} we have~$\rep\rho_{n,x_0}\in \DzLocB{\mu}$ for every~$x_0\in X$. By~\eqref{eq:p:Kuw:4} and~\eqref{eq:p:Kuw:5}, we have~$\rep\rho_{x_0}(y)=\nlim \rep\rho_{n,x_0}(y)$ for all~$y\in X$. As a consequence,~$\rho_{x_0} \in \DzLocB{\mu}$ for every~$x_0\in X$ and every~$r>0$ again by an application of Remark~\ref{r:weak*} and Lemma~\ref{l:ConvMeasure}.

\paragraph{Step 2} Suppose now that~$\mssd_\mu$ is an extended pseudo-distance. Since~$(X,\T_{\mssd_\mu})$ is separable, $(X,\mssd_\mu)$ has up to countably many accessible components~$X_i\eqdef B^{\mssd_\mu}_\infty(x_i)$, $x_i\in X$, each an element of~$\Bo{\T}$ and a separable pseudo-metric space. Fix~$x_0\in X_i$ for some~$i$. Without loss of generality, up to relabeling,~$i=0$. Arguing as in \emph{Step~1} with~$X_0$ in place of~$X$, there exists a sequence~$\tseq{\rep f_{n,x_0}}_n\subset \mcL^\infty(\A)$ so that~$\seq{f_{n,x_0}}_n \subset \DzLocB{\mu}$ and
\begin{align}\label{eq:p:Kuw:8}
0\leq \rep f_{n,x_0}(y)\leq \rep\rho_{x_0}(y)\comm \qquad \nlim \rep f_{n,x_0}(z)=\rep\rho_{x_0}(z)\comm \qquad y\in X\comm z\in X_0\fstop
\end{align}
Again arguing as in \emph{Step~1}, for each~$i\in\N$ let~$\set{y_{i,k}}_k\subset X_i$ be a countable set, $\T_{\mssd_\mu}$-dense in~$X_i$. For each~$i,k\in \N$, there exists a sequence~$\tseq{g^{i,k}_n}_n\subset \DzLocB{\mu,\T}$ of functions, defined on the whole of~$X$, so that
\begin{align*}
\nlim g^{i,k}_n(y_{i,k}) - g^{i,k}_n(x_0)=\mssd_\mu(y_{i,k},x_0)= \infty\comm \qquad i,k,n\in \N\fstop
\end{align*}
Note that the construction of~$g^{i,k}_n$ on~$X$ as in \emph{Step~1} for~$f_{n,i,x}$ can be done in this generality, i.e.\ on the whole of~$X$.

Without loss of generality, up to subtracting the constant~$g^{i,k}_n(x_0)$ by~\eqref{eq:E(1)=0}, and possibly taking a (non-relabeled) subsequence in~$n$, we may and shall assume that 
\begin{align}\label{eq:p:Kuw:6}
g^{i,k}_n(x_0)=0\comm \quad g^{i,k}_n(y_{i,k})\geq n\comm \qquad i,k,n \in \N\comm
\end{align}
and thus, by~\eqref{eq:p:Kuw:1.5},
\begin{align*}
g^{i,k}_n(y)\wedge r =\ttonde{g^{i,k}_n(y)-g^{i,k}_n(x_0)}\wedge r\leq \rep\rho_{x_0}(y)\comm \qquad i,k,n\in \N\comm \qquad y\in X \fstop
\end{align*}
Furthermore
\begin{align}\label{eq:p:Kuw:9}
g^{i,k}_n(y)-g^{i,k}_n(x)\leq \mssd_\mu(x,y) \comm \qquad i,k,n\in \N\comm \qquad x,y\in X\fstop
\end{align}
Letting~$\rep g_{n,x}\colon y\mapsto 0\vee \sup_{i,k} \ttonde{g^{i,k}_n(y)-g^{i,k}_n(x)}\wedge r$ for fixed~$x\in X$, the function~$\rep g_{n,x}$ is~$\Bo{\T}$-measurable, and $\T_{\mssd_\mu}$-continuous by~\eqref{eq:p:Kuw:9}.
By $\T_{\mssd_\mu}$-density of~$\set{y_{i,k}}_k$ in~$X_i$ for every~$i$, $\T_{\mssd_\mu}$-continuity of~$\rep g_{n,x}$, and~\eqref{eq:p:Kuw:6},
\begin{align}\label{eq:p:Kuw:7}
\rep g_{n,x_0}(y)= r \comm \qquad n\geq r \comm \qquad y\in X\setminus X_0 \fstop
\end{align}
Since~$\rep g_{n,x}=\mlim \max_{i,k\leq m}\ttonde{g^{i,k}_n(\emparg)-g^{i,k}_n(x)}\wedge r$ pointwise on~$X$ for every fixed~$x\in X$, one has that~$g_{n,x}\in \DzLocB{\mu}$ for every~$x\in X$ by Remark~\ref{r:weak*} and Lemma~\ref{l:ConvMeasure}.

Finally, set~$\rep\rho_{n,x_0}\colon y\mapsto \rep f_{n,x_0}(y) \vee \rep g_{n,x_0}(y)$, and note that~$\rho_{n,x_0}\in \DzLocB{\mu}$ by~\eqref{eq:TruncationLoc} for every~$n\in \N$ and every~$x_0\in X$. Then, by~\eqref{eq:p:Kuw:8} and~\eqref{eq:p:Kuw:7},
\begin{align*}
\rep\rho_{x_0}(y)\geq& \ \rep\rho_{n,x_0}(y)\geq \rep f_{n,x_0}(y) \comm && n\in \N\comm \qquad y\in X_0\comm
\\
r=\rep\rho_{x_0}(y)\geq& \ \rep\rho_{n,x_0}(y)\geq \rep g_{n,x_0} \comm && n\in \N\comm \qquad y\in X\setminus X_0\fstop
\end{align*}
Thus,~$\rep\rho_{n,x_0}$ converges pointwise to~$\rep\rho_{x_0}$ everywhere on~$X$ by~\eqref{eq:p:Kuw:8} and~\eqref{eq:p:Kuw:7}, and the conclusion is implied by Remark~\ref{r:weak*} and Lemma~\ref{l:ConvMeasure}.
\end{proof}

\begin{remark}[Comparison with~\cite{Kuw96}]\label{r:Kuw96}
In~\cite[Thm.~3.1]{Kuw96} the conclusion of Theorem~\ref{t:KuwaeProposition} is shown for \emph{non-local} spaces~$(\mbbX,\mcE)$ admitting carr\'e du champ operator with a point-separating form core of continuous bounded functions. In particular~$(\mbbX,\mcE)$ is quasi-regular by~\cite[Lem.~2.2]{Kuw96}.
Additionally, it is assumed there that
\begin{enumerate*}[$(a)$]
\item[$(A'')$] $\mssd_\mu$ is $\T^\tym{2}$-continuous;
for some~$g\in\dom$ with~$0<g\leq 1$ $\mssm$-a.e.\ and~$\psi\in L^1(g^2\mssm)$, it holds that
\item[$(C)_g$] $\mu=\psi\mssm$ is absolutely continuous;
\item[$(D)_g$] $0<g\leq 1$ $\mcE$-q.e.
\end{enumerate*}
Since~$\T$ is separable,~$(A'')$ implies that~$\T_{\mssd_\mu}$ is separable as well. Letting~$G_n\eqdef  (\reptwo g^2)^{-1} \ttonde{(1/n,2)}$, then~$G_\bullet\in\msG$ by~$(D)_g$, and therefore~$\mu\ll\mssm$ is $\mcE$-moderate by~$(C)_g$.
\end{remark}

The main consequence of the results in this section is collected in the next corollary.
\begin{corollary}\label{c:Separable}
Let~$(\mbbX,\mcE)$ be a quasi-regular strongly local Dirichlet space,~$\mu\in\Msp(\Bo{\T},\Ne{\mcE})$ be $\mcE$-moderate, and assume that~$\T_{\mssd_\mu}$ is separable.
Then,~$(\mbbX,\mcE,\mssd_\mu,\mu)$ possesses~$(\Rad{\mssd_\mu}{\mu})$.
\end{corollary}
\begin{proof}
Consequence of Theorem~\ref{t:Lenz}, Remark~\ref{r:Lenz}\iref{i:r:Lenz:3.0} and Theorem~\ref{t:KuwaeProposition}.
\end{proof}

\begin{remark}[Comparison with~\cite{KosZho12,Stu94}]
Because of Remark~\ref{r:Lenz}\iref{i:r:Lenz:2}, Corollary~\ref{c:Separable} is a sensible generalization to quasi-regular strongly local Dirichlet spaces of several results in the literature, including e.g.,~\cite[Lem.~1, Lem.~$1'$]{Stu94},~\cite[Thm.~2.1]{KosZho12}, obtained for~$\mssd_\mssm$ on strongly regular Dirichlet spaces. For the definition of \emph{strong regularity}, see Remark~\ref{r:StrongRegularity} below.
\end{remark}

\begin{corollary} Let~$(\mbbX,\mcE)$ be a quasi-regular strongly local Dirichlet space,~$\mu\in\Msp(\Bo{\T},\Ne{\mcE})$ be $\mcE$-moderate, and assume that~$\mssd_\mu\colon X^\tym{2}\rar[0,\infty)$ is $\T^\tym{2}$-continuous and everywhere finite.
Further let~$\mssd\colon X^\tym{2}\rar [0,\infty]$ be an extended pseudo-distance. Then, the following are equivalent:
\begin{enumerate}[$(a)$]
\item $(\mbbX,\mcE,\mssd,\mu)$ possesses~\eqref{eq:Rad} and~$\mssd$ is $\T$-admissible;
\item $(\mbbX,\mcE,\mssd,\mu)$ possesses~\eqref{eq:dRad}.
\end{enumerate}
\end{corollary}
\begin{proof}
Assume~\eqref{eq:dRad}. Then:
\begin{enumerate*}[$(a)$]
\item\label{i:c:Loc:1} $\mssd$ is $\T^\tym{2}$-continuous, since~$\mssd_\mu$ is;
and
\item\label{i:c:Loc:2} $\mssd$ is everywhere finite, since~$\mssd_\mu$ is, and therefore

\item\label{i:c:Loc:3} $\mssd(x_0,\emparg)\in \dotloc{L^\infty(\mssm)}$ by Lemma~\ref{l:LinftyLoc} for every~$x_0\in X$;
\item $\T_{\mssd_\mu}$ and~$\T_\mssd$ are separable, since~$\T$ is separable;
\item $\mssd(x_0,\emparg)$ is $\mssd_\mu$-Lipschitz for every~$x_0\in X$, since~$\mssd\leq \mssd_\mu$.
\end{enumerate*}
By Corollary~\ref{c:Separable},~$(\Rad{\mssd_\mu}{\mu})$ holds, therefore the assumptions~\iref{i:r:Lenz:1.1},~\iref{i:r:Lenz:1.2} in Remark~\ref{r:Lenz}\iref{i:r:Lenz:1} hold for~$\mssd$. Thus, Remark~\ref{r:Lenz}\iref{i:r:Lenz:1} applies, and~$(\Rad{\mssd}{\mu})$ follows from Theorem~\ref{t:Lenz}.
By~\iref{i:c:Loc:1} and~\iref{i:c:Loc:2} above,~$\mssd$ is $\T$-admissible with~$\UP=\set{\mssd}$.

The reverse implication holds by Lemma~\ref{l:RadDRad}.
\end{proof}

\begin{proof}[Proof of Proposition~\ref{p:BoundedSupp}]
We show the equivalent statement that~$(\Rad{\mssd}{\mu}^{bs,b})$ implies~$(\Rad{\mssd}{\mu}^b)$.

Firstly, note that $\mssd$-accessible components are open and closed, therefore $\mcE$-quasi-open and $\mcE$-quasi-closed, and thus $\mcE$-invariant by Proposition~\ref{p:InvarChar}. Analogously to the proof of Proposition~\ref{p:Invariant}, for every $\mssd$-accessible component~$A\subset X$, we have therefore that $\mssd(\emparg, A)\wedge 1 \in \DzLocB{\mu,\T}$.
Hence, $\mssd(\emparg, A^\complement)\wedge 1= \car - \ttonde{\mssd(\emparg, A)\wedge 1}$ is an element of~$\DzLocB{\mu,\T}$ as well, by~\eqref{eq:E(1)=0}. Set~$\rep\rho_A\eqdef \mssd(\emparg, A^\complement)\wedge 1$.

Fix now~$\rep f\in \bLipu(\mssd)$, and note that~$\rep\rho_A \cdot \rep f\in \bLipu(\mssd)$ as well, since~$\rep\rho_A\equiv\car_A$ and~$\mssd(x,y)=+\infty$ for every~$x\in A$ and~$y\in A^\complement$.
Arguing as in \emph{Step 2} in the proof of Theorem~\ref{t:KuwaeProposition},~$X$ has up to countably many $\mssd$-accessible components, thus it suffices to show the statement in the case when $\mssd$ has exactly one accessible components, that is, when~$\mssd$ is a (everywhere finite) pseudo-distance.
In this case, fix~$x_0\in X$, and set~$\rep f_n\eqdef \rep f\cdot \ttonde{0\vee \ttonde{n-\mssd(x_0,\emparg)}\wedge 1}$. Then,~$\rep f_n\equiv \rep f$ on the ball~$B^\mssd_{n-1}(x_0)$, and~$\rep f_n\in \bLipu(\mssd)$ with bounded support in~$B^\mssd_n(x_0)$. By assumption,~$f_n\in \DzLocB{\mu,\T}$, and the conclusion is implied by Remark~\ref{r:weak*} and Lemma~\ref{l:ConvMeasure} letting~$n\rar\infty$.
\end{proof}

The next example shows that the separability of~$\T_\mssd$ is not necessary for the Rademacher property to hold.
\begin{example}[Configuration Spaces II]\label{ese:Config2}
Recall the setting of Example~\ref{ese:Config1}, and in particular that the intrinsic distance~$\mssd_{\pi_d}$ of the canonical Dirichlet form~$(\mcE,\dom)$ on~$L^2(\Upsilon(\R^d))$ coincides with the $L^2$-transportation extended distance~$W_2$. Since all $W_2$-accessible components are $\pi_d$-negligible, there exist more than countably many such components, thus the topology~$\T_2$ on~$\Upsilon(\R^d)$ induced by~$W_2$ is not separable.
The Rademacher property~$(\Rad{W_2}{\pi_d})$ is shown in~\cite[Thm.~1.3]{RoeSch99}.

Since every $W_2$-accessible components is $\pi_d$-negligible, every $W_2$-Lipschitz function with bound\-ed support is measurable with respect to\ the $\pi_d$-completion of the $\sigma$-algebra on~$\Upsilon(\R^d)$, and coincides $\pi_d$-a.e.\ with the $\zero$-function.
\end{example}

\subsection{The Rademacher property and quasi-regularity}\label{ss:RadQuasiReg}
Note that the definition of intrinsic distance is always well-posed for Dirichlet spaces that are not necessarily quasi-regular.
In particular, we may always discuss properties like~$(\Rad{\mssd}{\mssm})$ and~$(\dRad{\mssd}{\mssm})$ on any Dirichlet space satisfying~\iref{ass:Hausdorff}. Note however that the definition of strong locality is well-posed only if~\iref{ass:Luzin} holds.
A discussion of the interplay between~$\mssd$ and~$\T$ motivates the following definitions, mimicking that of \emph{strict locality}~\cite[p.~224]{Sto10}.

\begin{definition}[Strict locality]\label{d:StrictLoc}
Let~$\mbbX$ be satisfying~\ref{ass:Luzin}, and~$(\mbbX,\mcE)$ be a strongly local Dirichlet space. We say that~$(\mbbX,\mcE)$ is \emph{strictly local} if~$\T_{\mssd_\mssm}=\T$.
\end{definition}

\begin{remark}\label{r:StrongRegularity}
If~$(\mbbX,\mcE)$ is a regular strongly local Dirichlet space, then the definition of `strict locality' coincides with that of \emph{strong regularity}, e.g.~\cite[p.~74]{Stu95}.
\end{remark}

If~$(\mbbX,\mcE)$ is strictly local, then~$(X,\T)$ is metrizable, and~$\mssd_\mssm$ is an extended distance.
The importance of strict locality is evident from the following more general fact. Let~$(\mbbX,\mcE)$ be a quasi-regular strictly local Dirichlet space, and~$\mu\in\Msp(\Bo{\T},\Ne{\mcE})$ be $\mcE$-moderate. Further assume that~$\mssd_\mu$ is $\T$-continuous. By Corollary~\ref{c:Separable},
\begin{align*}
\bLipu(\mssd_\mu)=\bLipu(\mssd_\mu,\A)=\bLipu(\mssd_\mu,\T)\subset \DzLocB{\mu,\T}\fstop
\end{align*}
Furthermore,~$\DzLocB{\mu,\T}\subset \bLipu(\mssd_\mu)$ by definition of~$\mssd_\mu$. Therefore,
\begin{align}\label{eq:LipDzLoc}
\bLipu(\mssd_\mu)=\DzLocB{\mu,\T} \fstop
\end{align}

The next result was essentially shown by L.~Ambrosio, N.~Gigli, and G.~Savar\'e in~\cite[Lem.~6.7]{AmbGigSav14} in the case when~$\mssm X<\infty$, and by Savar\'e in the case of $\sigma$-finite~$\mssm$.

\begin{proposition}\label{p:Consistency}
Let~$\mbbX$ be satisfying~\ref{ass:Luzin},~$(\mbbX,\mcE)$ be a strongly local Dirichlet space, and $\mssd\colon X^\tym{2}\rar [0,\infty]$ be a distance. Further assume that:
\begin{enumerate}[$(a)$]
\item\label{i:p:Consistency2} $\T_\mssd=\T$;
\item\label{i:p:Consistency1} $\mssm B^\mssd_r(x)<\infty$ for every~$x\in X$ and~$r\in(0,\infty)$;

\item\label{i:p:Consistency3} $(X,\mssd)$ is a $\T$-locally complete metric space;

\item\label{i:p:Consistency4} $(\mbbX,\mcE)$ possesses~$(\Rad{\mssd}{\mssm})$.
\end{enumerate}

Then,~$(\mbbX,\mcE)$ satisfies~\iref{i:QR:1} and~\iref{i:QR:3}.
In particular, if~$(\mbbX,\mcE)$ additionally satisfies~\iref{i:QR:2}, then it is quasi-regular.
\end{proposition}
\begin{proof}
The proof follows exactly as in~\cite[Thm.~4.1]{Sav14}. It suffices to note that~\iref{i:p:Consistency1} (as opposed to the exponential bound~\cite[($\mfm$-$\exp$), p.~1655]{Sav14}), and local completeness (as opposed to completeness) are enough to the arguments there.
\end{proof}

\begin{remark}[Comparison with~\cite{AmbGigSav15} --- part I]\label{r:QuasiRegularity}
Let~$(\mbbX,\mcE)$ be satisfying assumptions~\iref{i:p:Consistency2}-\iref{i:p:Consistency4} of Proposition~\ref{p:Consistency}. Since~$(X,\T)$ is second countable, the definition of `strong locality' in the sense of~\cite{MaRoe92} adopted here is implied by the definition of `locality' in the sense of~\cite{BouHir91}, as noted in~\cite[p.~78]{tElRobSikZhu06}.
Further note that~\iref{i:p:Consistency2},~\iref{i:p:Consistency3} and~\iref{ass:Luzin} with~$\A\eqdef \Bo{\T}^\mssm$ together are~\cite[(MD.a), p.~358]{AmbGigSav15}; \iref{i:p:Consistency1} is~\cite[(MD.b), p.~358]{AmbGigSav15} for~$\mssd$; \cite[(ED.b), p.~369]{AmbGigSav15} implies~$(\Rad{\mssd}{\mssm}^{bs})$, which in turn yields~\iref{i:p:Consistency4} by Proposition~\ref{p:BoundedSupp}.
If~$\mssd=\mssd_\mu$ in Proposition~\ref{p:Consistency} for some $\mcE$-moderate~$\mu\in\Msp(\Bo{\T},\Ne{\mcE})$, 
then~\iref{i:p:Consistency4} 
is a consequence of~\iref{i:p:Consistency2} and Corollary~\ref{c:Separable}.

Now, let~$D\eqdef \dom\cap \Cont(\T)$, and note that the form~$(\mcE,D)$ is closable. Its closure~$(\mcE_0,\dom_0)$ is a strongly local Dirichlet form satisfying~\iref{i:QR:2} by definition. 
By definition of intrinsic metric, the intrinsic metric of the form~$(\mcE_0,\dom_0)$ coincides with the intrinsic metric~$\mssd_\mssm$ of the original form.
As a consequence, the Dirichlet space~$(\mbbX,\mcE_0)$ is a strongly local Dirichlet space satisfying assumptions~\iref{i:p:Consistency2}-\iref{i:p:Consistency4} of Proposition~\ref{p:Consistency}, and it is therefore also a quasi-regular Dirichlet space. 
It follows that, under the assumptions of Proposition~\ref{p:Consistency}, we may assume~$(\mbbX,\mcE)$ to be additionally quasi-regular, with no loss of generality.
\end{remark}

\subsection{The Rademacher property and the length property}\label{ss:RadLength}
In this section we establish the length property for intrinsic distances of strictly local spaces, adapting the characterization in terms of sheaves given by P.~Stollman in~\cite{Sto10} for regular Dirichlet spaces.
We start with a preliminary Lemma.

\begin{lemma}\label{l:Sheaf}
Let~$(\mbbX,\mcE)$ be a quasi-regular strongly local Dirichlet space, and~$\mu\in\Msp(\Bo{\T},\Ne{\mcE})$ be $\mcE$-moderate. If~$\T_{\mssd_\mu}=\T$, then~$\DzLoc{\mu,\T}$ is a sheaf, i.e.\ for every~$f\in\Cont(\T)$ it holds that $f\in\DzLoc{\mu,\T}$ if and only if for every~$x\in X$ there exists $U\in\T$ with~$x\in U$ and so that~$f\restr_U\in \DzLoc{\mu,\T}(U)$.
\end{lemma}
\begin{proof}
Since~$\T_{\mssd_\mu}=\T$, $\mssd_\mu$-accessible components are open and closed, therefore (Borel) $\mcE$-quasi-open and $\mcE$-quasi-closed, and thus $\mcE$-invariant by Proposition~\ref{p:InvarChar}. As a consequence,~$f\in\DzLoc{\mu,\T}$ if and only if~$f\restr_A\in\DzLoc{\mu,\T}(A)$ for each $\mssd_\mu$-accessible component~$A\subset X$.
Thus, since~$(X,\mssd_\mu)$ has at most countable $\mssd_\mu$-accessible components by separability of~$\T$, we may assume with no loss of generality, up to restricting to each such component, that~$\mssd_\mu$ be everywhere finite.

Assume~$f\in \DzLoc{\mu,\T}$ and let~$G_\bullet\in\msG_0$ and~$f_\bullet\subset \dom$ be witnessing that~$f\in \DzLoc{\mu,\T}$. Since~$U$ is open,~$G^U_n\eqdef G_n\cap U$ is $\mcE$-quasi-open for every~$n$, and therefore~$G^U_\bullet\eqdef\seq{G^U_n}_n$ satisfies~$G^U_\bullet\in\msG_0(U)$. Thus,~$G^U_\bullet$ and~$f_\bullet$ witness that~$f\restr_U\in \DzLoc{\mu,\T}(U)$.

Vice versa, assume that for every~$x\in X$ there exists~$U_x\in \T$ so that~$x\in U_x$ and~$f\restr_{U_x}\in \DzLoc{\mu,\T}(U_x)$.
Since~$(X,\T)$ is Lindel\"of, there exists a countable set~$\set{x_n}_n\subset X$ so that~$U_\bullet\eqdef\seq{U_n}_n$, with~$U_n\eqdef U_{x_n}$, is an open covering of~$X$.
For every~$n$, there exist~$G^n_\bullet\eqdef\seq{G^n_k}_k\in\msG(U_n)$ and~$f^n_\bullet\eqdef \seq{f^n_k}_k\subset \dom\cap \Cont(U_n,\T)$ witnessing that~$f\restr_{U_n}\in \DzLoc{\mu,\T}(U_n)$.
Analogously to~\eqref{eq:LipDzLoc}, we have that~$f^n_k\in \Lipu(G^n_k,\mssd_\mu,\T)$ for every~$n$ and~$k$.

Set~$G_n\eqdef \cup_{k\leq n} G^n_k$ and note that~$G_n$ is $\mcE$-quasi-open.
Since~$U_\bullet$ is a covering of~$X$ and since $\cup_k G^n_k=U_n$ $\mcE$-q.e.\ for every~$n$, then~$\cup_n G_n=X$ $\mcE$-q.e., and thus~$G_\bullet\in\msG$.
Further note that, if~$G^n_k\cap G^m_h\neq \emp$ for some choice of the indices, then~$f^n_k=f=f^m_h$ $\mssm$-a.e.\ on~$G^n_k\cap G^m_h$, therefore everywhere on~$G^n_k\cap G^m_h$, since all functions involved are $\T$-continuous and~$\supp[\mssm]=X$.
As a consequence,~$f\restr_{G_n}$ is $\mssd_\mu$-Lipschitz on~$G_n$ with~$\Li[\mssd_\mu]{f\restr_{G_n}}\leq 1$.

Now, since~$\mssd_\mu$ is everywhere finite, the (e.g., lower) McShane extension~$\rep f_n$ of~$f\restr_{G_n}$ to~$X$ satisfies~$\rep f_n\in \Lipu(\mssd_\mu)$ for every~$n$, and~$\rep f_n\restr_{G_n}=f$ everywhere (hence $\mssm$-a.e.) on~$G_n$, hence~$f\in\dotloc{\dom}$ by Lemma~\ref{l:LocLoc}.
Since $\mssd_\mu$ metrizes~$\T$, we may write~$f_n$ in place of~$\rep f_n$ (i.e.: $\rep f_n$ is $\T$-continuous), and we have that~$f_n\in \DzLoc{\mu,\T}$ by Corollary~\ref{c:Separable}.

Finally, for every fixed~$r>0$ set~$f_{n,r}\eqdef (-r)\vee f_n \wedge r$ and~$f_r\eqdef (-r)\vee f \wedge r$. Since~$\cup_n G_n=X$ $\mcE$-q.e., then~$\cup_n G_n=X$ $\mssm$-a.e., and therefore~$\mssm$-a.e.-$\nlim f_{n,r}=f_r$, since~$G_n$ is $\mcE$-q.e.\ increasing and~$f_{n,r}=f_r$ on~$G_n$.
By Remark~\ref{r:weak*} and Lemma~\ref{l:ConvMeasure},~$f_r\in \DzLoc{\mu}$, hence~$f_r\in \DzLoc{\mu,\T}$ by continuity, for every~$r>0$. The conclusion follows from~$f\equiv f_r$ on~$\set{\abs{f}<r}$ and~\eqref{eq:SLoc:2}, letting~$r\to\infty$.
\end{proof}

\begin{theorem}\label{t:Stollmann}
Let~$(\mbbX,\mcE)$ be a quasi-regular strictly local Dirichlet space. If~$(X,\mssd_\mssm)$ is locally complete, then it is a length space.
\end{theorem}
\begin{proof}
We adapt the proof of~\cite[Thm.~5.2]{Sto10}.
We substitute~$\mcA^1$ in~\cite{Sto10} with~$\DzLoc{\mssm,\T}$.
Step~1 in the proof of~\cite[Thm.~5.2]{Sto10} relies on~\cite[Lem.~5.2]{Sto10}. Lemma~5.2(1) is substituted by Lemma~\ref{l:Sheaf} above. Lemma~5.2(2) is substituted by Lemma~\ref{l:ConvMeasure} above.
Step~4 applies to the quasi-regular case as well, having care to use  Proposition~\ref{p:Invariant}.
We have~$(\Rad{\mssd_\mssm}{\mssm})$ by Corollary~\ref{c:Separable}, which substitutes~\cite[Thm.~5.1]{Sto10} in the proof of Step~5.
\end{proof}

\begin{remark}[Comparison with~\cite{Sto10,AmbGigSav15}]\label{r:Stollmann}
Theorem~\ref{t:Stollmann} extends~\cite[Thm.~5.2]{Sto10} to the quasi-regular case, and~\cite[Thm.~3.10]{AmbGigSav15} to the locally complete non-complete case, without the necessity of~\cite[Dfn.~3.6(a)]{AmbGigSav15}. 
In the regular case, the choice to replace~$\mcA^1$ by~$\DzLoc{\mssm,\T}$ is justified by Proposition~\ref{p:BoundedDist}.
\end{remark}

The importance of the locally complete non-complete case is discussed in~\cite[Rmk.~3.4]{Sto10}, from which we borrow the next example.

\begin{example}[Stollman]\label{ese:Stollmann}
Let~$X\subset \R^d$ be open, and denote by~$\mssm$ the restriction of the standard Lebesgue measure on~$X$. On $L^2(\mssm)$ consider the Dirichlet form~$(\mcE,\dom)$ generated by the Laplacian with Dirichlet boundary condition. It is shown in~\cite[Prop.~3.3]{Sto10} that the intrinsic distance~$\mssd_{\mssm}$ induced by~$(\mcE,\dom)$ coincides with the length distance induced by the Euclidean distance on~$X$. As noted in~\cite{Sto10}, the latter space is a locally complete metric space. It is complete if and only if~$X=\R^d$.
\end{example}

\subsection{Localization}\label{ss:Localization}
In this section, we introduce another sufficient condition for a quasi-regular strongly local Dirichlet space to satisfy the Rademacher property. We are strongly inspired by~\cite[Defn.~3.6(a)]{AmbGigSav15}, which we reinterpret in the setting of~\S\ref{sss:Moderance}.

\begin{definition}\label{d:Loc}
Let~$(\mbbX,\mcE)$ be a strongly local Dirichlet space, and~$\mu\in\Msp(\Bo{\T},\Ne{\mcE})$.
We say that~$(\mbbX,\mcE,\mu)$ is \emph{$\mu$-uniformly latticially $\T$-localizable}, in short: $(\mbbX,\mcE,\mu)$ satisfies \eqref{eq:Loc}, if $\mu$ is $\mcE$-moderate, and there exists a latticial approximation to the identity~$\seq{\theta_n}_n$ uniformly bounded by~$\mu$ in energy measure, viz.
\begin{align}\tag{$\Loc{\mu,\T}$}\label{eq:Loc}
0\leq \theta_n\leq \theta_{n+1} \nearrow_n \infty\comm \qquad \theta_n\in\DzB{\mu,\T} \;(\,\subset \Cb(\T)\,)\fstop
\end{align}
\end{definition}

\begin{remark}\label{r:WlogTheta}
If~$\seq{\theta_n}_n$ witnesses~\eqref{eq:Loc} for~$(\mbbX,\mcE,\mu)$, we may and shall assume with no loss of generality that:
\begin{itemize}
\item $V_n\eqdef \inter_\T\set{\theta_n =n}\neq \emp$ defines an open covering of~$(X,\T)$ (up to passing to a subsequence);
\item $\theta_n\leq n$ for every~$n\in \N$ (by~\eqref{eq:Truncation} and~\eqref{eq:SLoc:1}, up to relabeling~$\theta_n$ as~$\theta_n\wedge n$).
\end{itemize}
\end{remark}

Dirichlet spaces satisfying~\eqref{eq:Loc} are also `algebraically localizable', in the following sense.
If~$f\in L^0(\mssm)$ is so that~$f\theta_n\in \dom$ for every~$n$, then~$\seq{V_n}_n$ as above and~$\ttseq{\tfrac{1}{n}\theta_n f}_n\subset \dom$ witness that~$f\in\dotloc{\dom}$.

\begin{remark}\label{r:AGS}
Consider a family of $1$-Lipschitz truncations~$S_r\in\Cb^1(\R)$, $r>0$, defined by
\begin{align*}
S_r(t)\eqdef r S(t/r) \comm \qquad \text{where~~} S(t)=\begin{cases} 1 & \text{if } \abs{t}\leq 1 \\ 0 &\text{if } \abs{t}\geq 3\end{cases} \quad \text{and}\quad \abs{S'(t)}\leq 1\fstop
\end{align*}
If~$\car\in\dom$, then~\eqref{eq:Loc} is trivially satisfied letting~$\theta_n\eqdef S_n\circ\car$.
Since $\mssm$ is absolutely $\mcE$-moderate, we may always choose $\mu=\mssm$, in which case~$(\Loc{\mssm,\T})$ on quasi-regular strongly local Dirichlet spaces coincides with~\cite[Dfn.~3.6(a)]{AmbGigSav15} on strongly local Dirichlet spaces over Polish spaces.
\end{remark}

Under~\eqref{eq:Loc} the intrinsic distance~\eqref{eq:IntrinsicD} coincides as well with the one defined in~\cite[Eqn.~(1.9)]{AmbGigSav15} for strongly local Dirichlet forms on (possibly not locally compact) Polish spaces.

\begin{proposition}\label{p:IntrinsicAGS}
Let~$(\mbbX,\mcE,\mu)$ be satisfying~\eqref{eq:Loc}. Then,
\begin{align}\label{eq:IntrinsicDom}
\mssd_\mu(x,y)=\sup\set{f(x)-f(y) : f \in \DzB{\mu,\T}} \fstop
\end{align}
\end{proposition}
\begin{proof}
For fixed~$x$ and~$y$ let~$\seq{f_k}_k\subset \DzLocB{\mu,\T}$ be so that~$\klim f_k(x)=\mssd_\mu(x,y)$ and~$f_k(y)=0$ for every~$k$. Without loss of generality, by~\eqref{eq:TruncationLoc}, $0\leq f_k(x)\leq \mssd_\mu(x,y)$ for every~$k$. 

If~$\mssd_\mu(x,y)<\infty$, then~$\theta_n(x) \wedge f_k(x)=f_k(x)$ for all~$n\geq \mssd_\mu(x,y)$, and~$\theta_n \wedge f_k\in\DzB{\mu,\T}$ by~\eqref{eq:TruncationLoc}. Thus~$\mssd_\mu(x,y)=\klim \theta_n(x)\wedge f_k(x)$.

If otherwise~$\mssd_\mu(x,y)=\infty$, then the right-hand side of~\eqref{eq:IntrinsicDom} is estimated from below by~$\theta_n(x)\wedge f_k(x)-\theta_n(y)\wedge f_k(y)=\theta_n(x)\wedge f_k(x)$.
Letting~$n\to\infty$ and~$k\to\infty$ yields~\eqref{eq:TruncationLoc}.
\end{proof}

\begin{remark}[Comparison with~{\cite{AmbGigSav15}} --- part II]
Together with Remark~\ref{r:QuasiRegularity}, Proposition~\ref{p:IntrinsicAGS} shows that we may compare our results with~\cite[Thm.~3.9]{AmbGigSav15}. In particular:
\begin{itemize}
\item Corollary~\ref{c:Separable} generalizes the implication `energy measure space implies $(\mathrm{ED.b})$' in~\cite[Thm.~3.9]{AmbGigSav15}, substituting the assumption~\cite[Dfn.~3.6(b)]{AmbGigSav15}: ``$\mssd_\mssm$ is a (everywhere finite) distance on~$X^\tym{2}$ which induces the topology~$\T$, and~$(X,\mssd_\mssm)$ is a complete metric space'' with ``the topology induced on~$X$ by~$\mssd_\mssm$ is separable''.
\end{itemize}
\end{remark}

\begin{example}[Configuration Spaces III]\label{ese:Config3}
Recall the setting of Example~\ref{ese:Config1}, and in particular that the intrinsic distance~$\mssd_{\pi_d}$ of the canonical Dirichlet form~$(\mcE,\dom)$ on~$L^2(\Upsilon(\R^d))$ coincides with the $L^2$-transportation extended distance~$W_2$.
Denote by~$\T_2$ the topology on~$\Upsilon(\R^d)$ induced by~$W_2$, and recall that it is not separable, since there exist uncountably many $W_2$-accessible components.
Since~$\pi_d$ is a finite measure, the Dirichlet space~$(\Upsilon(\R^d), \mcE)$ satisfies both~$(\Loc{\pi_d,\T_2})$ and~$(\Loc{\pi_d,\T})$.
A more explicit example of the latticial approximation~$\seq{\theta_n}_n$ is provided by the sequence of functions~$\theta_n\eqdef n+\rho_{\omega,n}$, where~$\rho_{\omega, n}$ is defined as in~\cite[Lem.~4.2]{RoeSch99}.
\end{example}

\begin{lemma}\label{l:RadLocality}
Let~$(\mbbX,\mcE,\mssd,\mu)$ be satisfying~\eqref{eq:Rad}. Further assume that~$\mssd\colon X^\tym{2}\rar [0,\infty)$ is $\T$-continuous and everywhere finite, and that~$\mu$ is $\mcE$-moderate. Then,~$(\mbbX,\mcE,\mu)$ satisfies~\eqref{eq:Loc}.
\end{lemma}
\begin{proof}
Fix~$x_0\in X$ and set~$\rho\eqdef \mssd(\emparg,x_0)$. By $\T$-continuity of~$\mssd$, the function~$\rho$ is $\T$-continuous, hence both $\Bo{\T}$- and $\A$-measurable, and everywhere finite, thus in~$\dotloc{L^\infty(\mssm)}$ by Lemma~\ref{l:LinftyLoc}. 
By~\eqref{eq:Rad} we have~$\rho_r\eqdef \rho\wedge r\in \DzLocB{\mu,\T}\subset \dotloc{\dom}$ for every~$r>0$. Since~$\rho\in \dotloc{L^\infty(\mssm)}$, then~$\rho\in\DzLoc{\mu,\T}$ by Lemma~\ref{l:FrankLenz}.
Set~$\theta\eqdef \reptwo \rho=\rho$, and let~$\theta_n$ be defined as in Remark~\ref{r:AGS}. Apply~\eqref{eq:ChainRuleLoc} to obtain
\begin{equation*}
\sq{\theta_n}= \ (S_n'\circ \reptwo \rho)^2 \cdot \sq{\rho} \leq \abs{S_n'}\sq{\rho}\leq \mu\comm\qquad n\in\N\fstop \qedhere
\end{equation*}
\end{proof}

\begin{example}[Perturbation by Coulomb densities]\label{ese:Coulomb1}
Consider the standard Euclidean space~$X=\R^d$, $d\geq 3$, endowed with the Euclidean distance~$\mssd$.
Let~$\seq{x_i}_{i\leq k}$ be $k$ distinct points, and~$\seq{p_i}_{i\leq k}$ be so that~$p_i\geq d$ for each~$i\leq k$.
Further let~$(\mcE,\dom)$ be the standard form on~$\R^d$ and set
\begin{align*}
\mssm\eqdef\phi^2\Leb^d\comm \qquad \phi(x)\eqdef \sum_{i=1}^k \mssd(x,x_i)^{-p_i/2}\fstop
\end{align*}

The $\mssm$-perturbed Dirichlet space~$(\mbbX,\mcE^\mssm)$ defined as in~\eqref{eq:Perturbed} is a quasi-regular strictly local Dirichlet space, thus satisfying~$(\Rad{\mssd_\mssm}{\mssm})$ by Corollary~\ref{c:Separable}.
Furthermore~$(\R^d,\mssd_\mssm)$ is locally complete, and thus a length space by Theorem~\ref{t:Stollmann}.

We note that the reference measure~$\mssm$ is \emph{neither} Radon \emph{nor} locally finite, and, therefore, the $\mssm$-perturbed Dirichlet space~$(\mcE^\mssm,\dom^\mssm)$ is not regular.
Thus, this example does not fall within the scope of either Frank--Lenz--Wingert~\cite[Thm.~4.9]{FraLenWin14} and Stollmann~\cite[Thm.~5.2]{Sto10}, assuming the regularity of~$(\mathcal E, \mathcal F)$, nor Ambrosio--Gigli--Savar\'e~\cite[Thm.~3.10]{AmbGigSav15}, requiring the finiteness of $\mssm$ on ($\mssd_\mssm$-)metric balls.
\end{example}
\begin{proof}
Set~$K\eqdef \set{x_1,\dotsc,x_k}$, and let~$K_\eps\eqdef \set{\mssd_K\leq \eps}$ for all~$\eps>0$. 
For~$i\leq k$, we denote by $V_i$  the Voronoi $\mssd$-cell $\set{x: \mssd_{x_i}(x)=\mssd_K(x)}$ of~$x_i$.

The measure~$\mssm$ is $\sigma$-finite on~$(X,\Bo{\T})$, yet not Radon, since it is infinite on every neighborhood of~$K$ in view of the choice of the exponents~$p_i$'s.
Since~$d\geq 3$, the set~$K$ is $\mcE$-polar.
As a consequence, the sets~$F_n\eqdef \mssd_K^{-1}\ttonde{[\tfrac{1}{n}, n]}$ form a compact $\mcE$-nest.
Since~$\mssm F_n<\infty$ for each~$n$, the measure~$\mssm$ is $\mcE$-smooth in the sense of Definition~\ref{d:Smooth}.
It follows that the $\mssm$-perturbed form~$(\mcE^\mssm,\dom^\mssm)$ in~\eqref{eq:Perturbed} is a quasi-regular strongly local Dirichlet form on~$\R^d$, non-regular, since~$\mssm$ is not Radon.
By~\cite[Cor.~6.2]{Kuw98}, the density~$\phi$ satisfies the assumptions of~\cite[Thm.~1.2]{Ebe96}, by which~$(\mcE^\mssm,\dom^\mssm)$ coincides with the closure of the form
\begin{align*}
\mcE^\mssm(f,g)\eqdef \int_{\R^d} \nabla f \cdot \nabla g\,  \phi^2 \diff\Leb^d\comm \qquad f,g\in \bigcup_{\eps>0}\Cc^\infty(\R^d\setminus K_\eps) \fstop
\end{align*}
The latter closure admits carr\'e du champ~$\cdc^\mssm$ satisfying~$\cdc^\mssm(f,g)=\nabla f\cdot \nabla g$ on e.g.~$\Cc^\infty(\R^d\setminus K_\eps)$ for each~$\eps>0$, cf.~\cite[Rmk., p.~510]{Ebe96}.

On the one hand, we have that~$\dotloc{(\dom^\mssm)}\subset \dotloc{\dom}$ (see~\eqref{eq:Perturbed}), and thus (also cf.\ Example~\ref{ese:Perturbed})
\begin{align*}
\mssd_\mssm(x,y)=\sup\set{f(x)-f(y) : f\in\dotloc{(\dom^\mssm)}\cap \Cb(\R^d)\comm \cdc^\mssm(f)\leq 1 \as{\Leb^d}}\leq \mssd(x,y)\comm
\end{align*}
whence $\mssd_\mssm$ is a separable pseudo-distance.
On the other hand,
\begin{align*}
\mssd_\mssm(x,y)\geq \sup\set{f(x)-f(y): f\in \bLip(\mssd)\comm \Li[\mssd]{f}\leq 1\comm \exists \eps>0: f\restr_{K_\eps}\equiv 0} \comm
\end{align*}
Therefore, if~$x_i\in K$, and setting~$f_\eps\eqdef \mssd_{K_\eps}$,
\begin{align*}
0\vee \ttonde{\mssd(x_i,y)-\eps}=\mssd_{K_\eps}(y) = f_\eps(y)-f_\eps(x_i)\leq \mssd_\mssm(x_i,y)\leq \mssd(x_i,y)\comm \qquad y\in V_i \comm
\end{align*}
hence, letting~$\eps$ to~$0$, we have that~$\mssd_\mssm(x_i,\emparg)=\mssd(x_i,\emparg)$ on the cell~$V_i$.
If otherwise~$x\notin K$, setting~$f_\eps \eqdef \ttonde{\mssd(x,K_\eps)-\mssd(x,\emparg)}\vee 0$,
\begin{align*}
\mssd(x,y) = f_\eps(x)-f_\eps(y)\leq\mssd_\mssm(x,y)\leq \mssd(x,y) \comm \qquad y\in B_{\mssd_{K_\eps}(x)}(x)\comm
\end{align*}
hence, letting~$\eps$ to~$0$, we have that~$\mssd_\mssm(x,\emparg)=\mssd(x,\emparg)$ on~$B_{\mssd(x,K)}(x)$.
As a consequence,~$\mssd_\mssm$ is locally equivalent to~$\mssd$ on~$\R^d$, and therefore locally complete, and inducing the standard topology on~$\R^d$.
Since~$\R^d$ is also $\mssd_\mssm$-closed, then~$\mssd_\mssm$ is in fact complete.
\end{proof}

Since the set~$K$ of singular points for~$\mssm$ is $\mcE^\mssm$-polar, removing it from the space yields a new Dirichlet space~$(\tilde\mbbX,\mcE^{\tilde\mssm})$, with~$\tilde X\eqdef \R^d\setminus K$ and~$\tilde\mssm\eqdef\mssm\mrestr{\tilde X}$, quasi-homeomorphic to~$(\mbbX,\mcE^\mssm)$.
This new space is in fact a regular Dirichlet space, so that the aforementioned results in~\cite{FraLenWin14,Sto10} apply.
However, this is purely a consequence of the finite-dimensionality of~$\tilde\mbbX$, as shown in the next Example~\ref{ese:Coulomb2}, extending Example~\ref{ese:Coulomb1} to separable Hilbert spaces.
Finally, the analogous results in~\cite{AmbGigSav15} are still not applicable, since~$(\tilde X,\mssd_{\tilde\mssm})$ is not complete.

\begin{example}\label{ese:Coulomb2}
If we replace the Lebesgue measure by a log-concave probability measure~$\gamma$, the same constructions in Example~\ref{ese:Coulomb1} apply as well to any separable infinite-dimensional Hilbert space~$(H,\norm{\emparg})$.
Indeed, let~$(\mcE,\dom)$ be the quasi-regular strongly local Dirichlet form on~$(H,\gamma)$ constructed in~\cite[Thm.~1.2]{AmbSavZam09}.
Fix~$K=\set{h_i}_{i\leq k}\subset H$, and set $\phi(h)\eqdef \sum_{i=1}^k \norm{h-h_i}^{-p_i/2}$ and~$\mssm\eqdef \phi^2\cdot \gamma$.
Since~$K$ is polar and the form is quasi-regular, there exists a compact $\mcE$-nest~$\seq{F_n}_n$ with~$F_n\subset \mssd_K^{-1}\ttonde{[n^{-1},n]}$.
As in the previous example, the perturbed form~$(\mcE^\mssm,\dom^\mssm)$ is quasi-regular strongly local.
By~\cite[Thm.~5.1]{AmbSavZam09}, also cf.~\cite[p.~1409]{AmbGigSav14b}, the space~$\mbbX\eqdef (H,\gamma, \mssd)$ satisfies the~$\RCD(0,\infty)$ condition and~$(\mcE,\dom)$ coincides with the Cheeger energy of~$(H,\gamma,\mssd)$.
In particular,~$(\mbbX,\mcE)$ possesses~$(\Rad{\mssd}{\gamma})$.
The rest of the argument follows similarly to Example~\ref{ese:Coulomb1}.
\end{example}

\section{Varadhan short-time asymptotics}

The second property of our interest is stated in the next definition.

\begin{definition}[Sobolev-to-Lipschitz]\label{d:StoL}
Let~$(\mbbX,\mcE)$ be a quasi-regular strongly local Dirichlet space, $\mu\in\Msp(\Bo{\T},\Ne{\mcE})$ be $\mcE$-dominant, and~$\mssd\colon X^\tym{2}\rar[0,\infty]$ be an extended pseudo-metric. We say that the quadruple~$(\mbbX,\mcE,\mssd,\mu)$ has:
\begin{itemize}
\item the \emph{Sobolev--to--continuous-Lipschitz property} if
\begin{align}\tag{$\ScL{\mu}{\T}{\mssd}$}\label{eq:ScL}
f\in\DzLoc{\mu} \qquad \implies \qquad \text{$\exists\, \rep f\in\Lip^1(\mssd,\T)$}\semicolon
\end{align}

\item the \emph{Sobolev--to--Lipschitz property} if
\begin{align}\tag{$\SL{\mu}{\mssd}$}\label{eq:SL}
f\in\DzLoc{\mu} \qquad \implies \qquad \text{$\exists\, \rep f\in\Lip^1(\mssd,\A)$}\semicolon
\end{align}

\item the \emph{continuous-Sobolev--to--Lipschitz property} if
\begin{align}\tag{$\cSL{\T}{\mu}{\mssd}$}\label{eq:cSL}
f\in \DzLoc{\mu,\T} \qquad \implies \qquad \text{$\exists\, \rep f\in\Lip^1(\mssd,\T)$}\semicolon
\end{align}

\item the \emph{distance Sobolev-to-Lipschitz property} if
\begin{align}\tag{$\dSL{\mssd}{\mu}$}\label{eq:dSL}
\mssd\geq \mssd_\mu \fstop
\end{align}
\end{itemize}
\end{definition}

Arguing by truncation, we might equivalently define, e.g.,
\begin{align}\tag{$\ScL{\mu}{\T}{\mssd}^b$}
f\in\DzLocB{\mu} \qquad \implies \qquad \exists\, \rep f\in\bLipu(\mssd,\T)\fstop
\end{align}

Since a $\mssd$-Lipschitz function is automatically $\T_\mssd$-continuous,~$(\ScL{\mu}{\T_\mssd}{\mssd})$ coincides with $(\SL{\mu}{\mssd})$. However, since~$\T$-continuity and~$\T_\mssd$-continuity are unrelated,~$(\cSL{\T}{\mu}{\mssd})$ is in general different from $(\cSL{\T_\mssd}{\mu}{\mssd})$. An example of a space satisfying~$(\cSL{\T_\mssd}{\mu}{\mssd})$, and for which~$(\SL{\mssd}{\mu})$ is not known, is given in Example~\ref{ese:Config4} below.

\blue{Since~$\mssm$ has full $\T$-support, the $\T$-continuous $\mssd$-Lipschitz representative~$\rep f$ in~\eqref{eq:cSL} always coincide with the given $\T$-continuous representative~$f$.
As usual for continuous functions, this will be reflected in the notation by omitting the notation for representatives altogether.}

\blue{Analogously,} the $\mssd$-Lipschitz representative~$\rep f$ in~\eqref{eq:ScL} is always uniquely identified. In general the same does \emph{not} hold for~\eqref{eq:SL}, that is, if~\eqref{eq:SL} holds, then there might exist even uncountably many $\mssd$-Lipschitz representatives~$\rep f$ of~$f$, see Example~\ref{ese:Config4}.

Finally, note that, if~\eqref{eq:cSL} holds, then~\eqref{eq:ScL} is equivalent to the requirement that
\begin{align}\label{eq:L=Lc}
\DzLoc{\mu,\T}=\DzLoc{\mu} \fstop
\end{align}

\begin{proposition}\label{p:StoL}
The following implications hold.
\begin{align*}
(\ScL{\mu}{\T}{\mssd}) \implies (\SL{\mu}{\mssd}) \implies (\cSL{\T}{\mu}{\mssd}) \iff (\dSL{\mssd}{\mu})
\end{align*}
\end{proposition}
\begin{proof}
The implications~$(\ScL{\mu}{\T}{\mssd}) \implies (\SL{\mu}{\mssd}) \implies (\cSL{\T}{\mu}{\mssd})$ hold by definition.
Assume $(\cSL{\T}{\mu}{\mssd})$. Then, for every~$f\in \DzLocB{\mu,\T}$ with $\mssd$-Lipschitz representative $\rep f$, one has $\abs{f(x) - f(y)}\leq \mssd(x,y)$, whence~$(\dSL{\mssd}{\mu})$ holds by taking the supremum over all~$f$ as above.
Assume~$(\dSL{\mssd}{\mu})$. Then, for every~$f\in \DzLocB{\mu,\T}$ with continuous representative $\rep f$, one has $\ttabs{\rep f(x)- \rep f(y)}\leq \mssd_\mu(x,y)$ by definition of~$\mssd_\mu$, whence~$(\cSL{\T}{\mssd}{\mu})$ holds by~$(\dSL{\mssd}{\mu})$.
\end{proof}

\begin{remark}\label{r:SLScL}
If~$\mssd$ is $\T^\tym{2}$-continuous, then every~$\rep f\in \Lip(\mssd)$ is $\T$-continuous as well. Thus,~\eqref{eq:ScL} and~\eqref{eq:SL} coincide. In particular, if~$\mssd_\mu$ is $\T^\tym{2}$-continuous, then~$(\SL{\mu}{\mssd_\mu})$ coincides in this case with~$(\ScL{\mu}{\T}{\mssd_\mu})$, and, combining~\eqref{eq:LipDzLoc} with~\eqref{eq:L=Lc}, we have under~$(\SL{\mu}{\mssd_\mu})$ that
\begin{align}\label{eq:Lip=L}
\bLipu(\mssd_\mu)=\DzLocB{\mu}=\DzLocB{\mu,\T}\fstop
\end{align}
\end{remark}

\begin{remark}[Comparison with~{\cite{AmbGigSav15}} --- part III]
Since~$(\dSL{\mssd_\mu}{\mu})$ always holds by definition, for every~$\mu$, we automatically have~$(\cSL{\T}{\mssm}{\mssd_\mssm})$ by Proposition~\ref{p:StoL}, which is~\cite[($\mathrm{ED.a}$)]{AmbGigSav15} when~$\mssd_\mssm$ is $\T^\tym{2}$-continuous.
\end{remark}

\begin{corollary}
Let~$(\mbbX,\mcE)$ be a quasi-regular strongly local Dirichlet space,~$\mssd\colon X^\tym{2}\rar[0,\infty]$ be an extended pseudo-distance, and~$\mu\in\Msp(\Bo{\T},\Ne{\mcE})$ be $\mcE$-moderate.
Assume that~$\T_\mssd$ is separable.
Then~\eqref{eq:dSL} implies~$(\Rad{\mssd_\mu}{\mu})$.
\end{corollary}
\begin{proof}
By~\eqref{eq:dSL}, the topology induced by~$\mssd_\mu$ is separable as well. The conclusion now follows by Corollary~\ref{c:Separable}.
\end{proof}

Let us now collect some examples.
\begin{example}[Triviality]
Without information on the broad local space~$\DzLocB{\mu,\T}$, the continuous-Sobolev--to--Lipschitz property can easily trivialize. For instance, this is the case for the Liouville Brownian motion described in Example~\ref{ese:LQG}, or for any other Dirichlet space for which~$\mssd_\mu$ vanishes identically, in which case the only $\mssd_\mu$-Lipschitz functions are constant ones.
\end{example}

\begin{example}[Connexion vs.\ quasi-connexion]
For~$i\in \set{\pm 1}$ let~$\mbbX_i\eqdef (X_i,\T_i,\A_i,\mssm_i)$ be the standard closed unit disk in~$\R^2$ centered at~$(i,0)\in \R^2$, and~$(\mcE_i,\dom_i)$ be the Dirichlet form generated by the Neumann Laplacian on~$X_i$.
Further set~$\mbbX\eqdef \mbbX_{-1}\cup \mbbX_{+1}$ and~$(\mcE,\dom)\eqdef (\mcE_{-1},\dom_{-1})\oplus (\mcE_{+1},\dom_{+1})$, endowed with the measure~$\mssm\eqdef\mssm_{-1}+\mssm_{+1}$. 
Since~$X_{-1}\cap X_{+1}=\set{\zero_2}$ is both $\mcE_{+1}$- and $\mcE_{-1}$-polar, it is not difficult to show that~$(\mcE,\dom)$ is a regular strongly local Dirichlet space.
Both~$(\mbbX_i,\mcE_i)$ satisfy~$(\ScL{\mssm_i}{\T_i}{\mssd_i})$ for the standard topology and the Euclidean distance~$\mssd_i\eqdef\mssd_{\mssm_i}$ on~$X_i$.
Let~$\mssd_\mssm$ be the intrinsic distance of~$(\mbbX,\mcE)$. By strong locality,~$\mssd_\mssm$ coincides with~$\mssd_i$ on each~$\mbbX_i$.
By triangle inequality,~$\mssd_\mssm(x_i,x_j)\leq \mssd_i(x_i,\zero_2)+\mssd_j(x_j,\zero_2)$ for every~$i\neq j\in \set{\pm 1}$ and every~$x_i\in X_i$,~$x_j\in X_j$. Thus,~$\mssd_\mssm$ is a finite distance.
The function~$\rep f\eqdef \car_{X_{+1}}$ satisfies~$f\in\dom$ with~$\sq{f}=0$, but it is not $\mssd_\mssm$-Lipschitz. Therefore~$(\mbbX,\mcE)$ does not satisfy~$(\SL{\mssm}{\mssd_\mssm})$.

On the other hand, let us consider the space~$\mbbX_\circ\eqdef\mbbX\setminus\set{\zero_2}$, i.e.\ the set~$X\setminus\set{\zero_2}$ endowed with the subspace topology~$\T_\circ$, the trace $\sigma$-algebra~$\A_\circ\eqdef \A\cap (X\setminus\set{\zero_2})$, and the restriction~$\mssm_\circ\eqdef\mssm\mrestr X\setminus\set{\zero_2}$.
Since~$\set{\zero_2}$ is $\mssm$-negligible, $L^2(\mssm)$~is latticially isometrically isomorphic to~$L^2(\mssm_\circ)$, and we may define a form~$\mcE_\circ$ on~$L^2(\mssm_\circ)$ by letting~$\mcE_\circ\eqdef\mcE$.
Now, the function~$\car_{X_{+1}}$ is $\T_\circ$-continuous, since~$\mbbX_\circ$ is $\T_\circ$-disconnected, and therefore we have, cf.~\cite[Prop.~3.3]{Sto10},
\begin{align*}
\mssd_\circ(x,y)\eqdef\mssd_{\mssm_\circ}(x,y)=\begin{cases} \mssd_i(x,y) & \textrm{if~} x,y\in \mbbX_i\setminus\set{\zero_2}
\\
+\infty & \text{otherwise}
\end{cases} \fstop
\end{align*}

As a consequence,~$(\mbbX_\circ,\mcE_\circ)$ possesses~$(\ScL{\mssm_\circ}{\T_\circ}{\mssd_\circ})$.
Finally, note that the abstract completion~$\overline{\mbbX_\circ}$ of~$\mbbX_\circ$ with respect to~$\mssd_\circ$ does not coincide with~$(\mbbX,\mssd_\mssm)$, not even as a set. One has instead that~$\overline{\mbbX_\circ}= \mbbX_{+1}\sqcup \mbbX_{-1}$.
\end{example}

\begin{example}[Wiener spaces I]\label{ese:Wiener1}
Let~$(X,H,\mssm)$ be an abstract Wiener space, endowed with the (extended) Cameron--Martin distance~$\mssd_H(x,y)\eqdef \norm{x-y}_H$ and its canonical Dirichlet form~$(\mcE,\dom)$, see e.g.~\cite[Eqn.~(1.6)]{RenRoe05} in the case when $\mu=\mssm$ is the Wiener measure.
By~\cite{EncStr93}, the Dirichlet space~$(\mbbX,\mcE)$ associated to~$(X,H,\mssm)$ possesses both~$(\Rad{\mssd_H}{\mssm}^b)$ and~$(\SL{\mssm}{\mssd_H})$.
By~\cite[Lem.~1.3]{RenRoe05}, we have
\begin{align*}
\mssd_H(x,y)=\sup\set{f(x)-f(y): f\in \mathcal{F}C^\infty_b, \norm{\cdc(f)}_\infty \leq 1}\leq\mssd_\mssm(x,y)\comm
\end{align*}
i.e.~$(\dRad{\mssd_H}{\mssm})$ holds. By Proposition~\ref{p:StoL}, we have~$(\dSL{\mssd_H}{\mssm})$ as well. Thus,~$\mssd_\mssm=\mssd_H$ and~$\mssd_H$ is $\T$-admissible by definition of~$\mssd_\mssm$.
\end{example}

\begin{example}[Configuration spaces IV]\label{ese:Config4}
Recall the setting of Example~\ref{ese:Config3}.
The form~$(\mcE,\dom)$ possesses~$(\cSL{\T_2}{\pi_d}{W_2})$ by~\cite[Thm.~1.5(i)]{RoeSch99}. It is conjectured in~\cite[Rmk., p.~331]{RoeSch99} that~$(\mcE,\dom)$ possesses in fact the stronger property~$(\SL{\pi_d}{W_2})$.

Recall further, again from Example~\ref{ese:Config1}, that, for every $W_2$-accessible component~$A\subset \Upsilon(\R^d)$, the characteristic function of~$A^\complement$ is a Borel measurable $W_2$-Lipschitz representative of the $\pi_d$-equivalence class of~$\car$ in $L^2(\Upsilon(\R^d))$.
As a consequence, if the $\pi_d$-equivalence class~$f$ of some function in~$\DzLocB{\pi_d}=\DzB{\pi_d}$ has a $W_2$-Lipschitz $\pi_d$-representative, then it has in fact uncountably many different such representatives.
\end{example}

\begin{example}[Wasserstein diffusion]
Even in the case of regular strictly local Dirichlet spaces, \eqref{eq:ScL} ---~i.e.~\eqref{eq:SL}~--- may be beyond reach. In such cases,~\eqref{eq:cSL} turns out to be a useful surrogate of~\eqref{eq:ScL}. Apart for the case of configuration spaces detailed in Example~\ref{ese:Config4}, another example of this fact is provided by the Dirichlet space of the \emph{Wasserstein diffusion}~\cite{vReStu09}.
Indeed, let~$(\mcE,\dom)$ be the form~\cite[Dfn.~7.24]{vReStu09} defined by integration of the squared $L^2$-Wasserstein gradient with respect to\ the entropic measure~$\mbbP^\beta$~\cite[Dfn.~3.3]{vReStu09} on the space of probability measure~$\msP(\mbbS^1)$ over the unit circle, endowed with the narrow topology~$\T_\mrmn$.
Then,~$(\mcE,\dom)$ is a $\T_\mrmn$-regular strictly local Dirichlet form on~$\msP_2(\mbbS^1)$ with intrinsic distance the Wasserstein distance~$W_2$, see~\cite[Thm.~7.25 and Cor.~7.29]{vReStu09}, satisfying the continuous-Sobolev--to--Lipschitz property by~\cite[Prop.~7.26(ii)]{vReStu09}.
\end{example}

\begin{example}[Sobolev-to-Lipschitz on metric measure spaces]\label{ese:RCD}
A main example of Dirichlet spaces satisfying the Sobolev-to-Lipschitz property is provided by metric measure spaces $(X,\mssd,\mssm)$ as in Example~\ref{ese:MMS}.
In this setting, sufficient conditions are, for instance that: 
$(X,\mssd,\mssm)$ has synthetic Riemannian Ricci curvature bounds in the sense of e.g.~\cite{Gig12a}, see~\cite[Thm.~6.2]{AmbGigSav14b};
$(X,\mssd,\mssm)$ is $2$-thick geodesic infinitesimally doubling (and infinitesimally Hilbertian), see~\cite[Dfn.s~1.3,~1.6, Thm.~1.7]{CreSou19}.
In both cases, we have~$(\dSL{\mssd}{\mssm})$ by Proposition~\ref{p:StoL}. Since~$\T=\T_\mssd$ by assumption,~$\mssd$ is trivially $\T$-admissible. Since we consider the Dirichlet space of the Cheeger energy, $(\Rad{\mssd}{\mssm})$~holds by construction, as well as~$(\dRad{\mssd}{\mssm})$, by Lemma~\ref{l:RadDRad}. As a consequence, in both cases we have~$\mssd_\mssm=\mssd$.
\end{example}

\subsection{Sobolev--to--Lipschitz-type properties and completeness}\label{ss:StoLCompleteness}
As shown in Proposition~\ref{p:Consistency} and Theorem~\ref{t:Stollmann}, under the assumption of strict locality of a Dirichlet space~$(\mbbX,\mcE)$, the completeness of the intrinsic distance~$\mssd_\mssm$ plays an important role. In this section, we draw a comparison between a strictly local Dirichlet space and its image on the metric completion of the underlying space endowed with the intrinsic metric.
We mostly expand on~\cite[Rmk 3.7]{AmbGigSav15}, by showing that if~$(\mbbX,\mcE)$ is a quasi-regular strictly local Dirichlet space satisfying~$(\SL{\mssm}{\mssd_\mssm})$, then we may assume that~$(X,\mssd_\mssm)$ is additionally complete, with no loss of generality. We do not, however, assume that~$(X,\T)$ is a priori Polish.
Let us start with some topological considerations.

\smallskip

Let~$(\mbbX,\mcE)$ be a Dirichlet space, and~$\mssd\colon X^\tym{2}\rar [0,\infty]$ be an extended distance. Further let~$(X^\iota,\mssd^\iota)$ be the abstract completion of~$(X,\mssd)$ and denote by~$\iota$ the completion embedding~$\iota\colon X\rar X^\iota$.
If~$\iota(X)$ is a Borel subset of~$X^\iota$, then~$\iota$ is $\Bo{\T_{\mssd}}/\Bo{\T_{\mssd^\iota}}$-measurable, and the image form~$(\mcE^\iota,\dom^\iota)$ in~\eqref{eq:Quasihomeo} is well-defined on the image space~$\mbbX^\iota$.

\begin{proposition}\label{p:Completion}
Let~$(X,\T)$ be satisfying~\ref{ass:Luzin} and~$(\mbbX,\mcE)$ be a Dirichlet space satisfying~\iref{i:QR:1}. 
Further let~$\mssd\colon X^\tym{2}\rar [0,\infty]$ be an extended distance generating~$\T$.
Then, the Dirichlet spaces~$(\mbbX,\mcE)$ and~$(\mbbX^\iota, \mcE^\iota)$ are quasi-homeomorphic.
In particular, $(\mbbX^\iota,\mcE^\iota)$ is quasi-regular if and only if~$(\mbbX,\mcE)$ is so.
\end{proposition}
\begin{proof} Since~$(X,\T)\cong (X,\T_\mssd)$ by strict locality, we may denote by~$\T^\iota$ the topology on~$X^\iota$ induced by~$\mssd^\iota$, with no risk of confusion.
Since~$\mbbX$ is metrizable Luzin,~$\mbbX^\iota$ is Polish, and~$\iota(X)\in\Bo{\T^\iota}$ by~\cite[Thm.~6.8.6]{Bog07}. Thus,~$(\mcE^\iota,\dom^\iota)$ is well-defined.
Since~\iref{i:QH:2} and~\iref{i:QH:3} hold by construction, it suffices to show~\iref{i:QH:1} for~$(\mbbX^\iota,\mcE^\iota)$.
By~\iref{i:QR:1} for~$(\mbbX,\mcE)$, there exists a $\T$-compact $\mcE$-nest~$F_\bullet$.
By continuity of~$\iota$, $F^\iota_n\eqdef \iota(F_n)$ is $\T^\iota$-compact. By injectivity of~$\iota$, and since~$F^\iota_n$ is Hausdorff,~$\iota\restr_{F_n}$ is a homeomorphism onto~$F^\iota_n$ for every~$n$, e.g.~\cite[Thm.~3.1.13]{Eng89}. Thus, it suffices to show that~$F^\iota_\bullet$ is an $\mcE^\iota$-nest. This follows by definition of~$\mcE$-nest and~\iref{i:QH:3}.
\end{proof}

The assertion of Proposition~\ref{p:Completion} may be equivalently rephrased by saying that~$X^\iota\setminus \iota(X)$ is $\mcE^\iota$-polar.

\smallskip

Proposition~\ref{p:Completion} is not of great interest when considering Dirichlet spaces up to quasi-homeomor\-phism. Indeed, if~$(\mbbX,\mcE)$ is additionally quasi-regular (as opposed to: only satisfying~\iref{i:QR:1}), then one should rather consider a quasi-homeomorphic regular Dirichlet space.
However, the proposition is insightful in the case when~$\mssd=\mssd_\mssm$ is the intrinsic distance of a strictly local quasi-regular Dirichlet space, as we now show.

\begin{proposition}
Let~$(\mbbX,\mcE)$ be a quasi-regular strictly local Dirichlet space satisfying~$(\SL{\mssm}{\mssd_\mssm})$. Then, $(X^\iota,\mssd_\mssm^\iota)$ and~$(X^\iota,\mssd_{\mssm^\iota})$ are isometric.
\end{proposition}
\begin{proof}
Let~$f^\iota\in \DzLocB{\mssm^\iota,\T^\iota}$. Then,~$f\eqdef f^\iota\restr_{\iota(X)}$ satisfies~$f\in \DzLocB{\mssm,\T}$ and therefore~$\mssd_{\mssm^\iota}\leq \mssd_\mssm^\iota$.
Vice versa, let~$f\in\DzLocB{\mssm,\T}$. By~$(\SL{\mssm}{\mssd_\mssm})$, and since~$f$ is bounded, there exists~$r>0$ so that~$f$ is $\mssd_\mssm\wedge r$-Lipschitz. Since~$\mssd_\mssm\wedge r$ is a distance,~$f$ is uniformly continuous, and therefore extends uniquely to a continuous function~$f^\iota$ on~$X^\iota$.
By Proposition~\ref{p:Completion}, and in particular by~\iref{i:QH:3}, one has that~$f^\iota\in \DzLocB{\mssm^\iota,\T^\iota}$, and the reverse inequality follows.
\end{proof}

\subsection{Sobolev--to--Lipschitz-type properties and Varadhan asymptotics}\label{ss:StoLMaxFunc}
Under the assumption of the Sobolev-to-Lipschitz property, we may compare point-to-set distance functions with their `maximal representatives' in~$\dotloc{\dom}$.

\paragraph{Maximal functions}
We start by recalling the following result of T.~Ariyoshi and M.~Hino~\cite{AriHin05}, extending M.~Hino and J.~Ram\'irez~\cite{HinRam03} to the case of $\sigma$-finite measure, adapted to our setting. Set
\begin{align*}
\Dz{\mu,A}_{\loc,r}\eqdef \set{f\in\DzLoc{\mu}: f=0 \as{\mssm} \text{~on~} A\comm \abs{f}\leq r \as{\mssm}}\subset \DzLocB{\mu}\comm \qquad r>0\fstop
\end{align*}

\begin{proposition}[{\cite[Prop.~3.11]{AriHin05}, cf.~\cite[Thm.~1.2]{HinRam03}}]\label{p:Hino}
Let~$(\mbbX,\mcE)$ be a quasi-regular strongly local Di\-richlet space,~$\mu\in\Msp(\Bo{\T},\Ne{\mcE})$ be $\mcE$-moderate. For each~$A\in\A$ there exists an $\mssm$-a.e.\ unique $\A$-measurable function~$\hr{\mu,A}\colon X\rar [0,\infty]$ so that~$\hr{\mu,A}\wedge r$ is the $\mssm$-a.e.\ maximal element of~$\Dz{\mu,A}_{\loc, r}$.
\end{proposition}
\begin{proof}
The statement is well-posed since~$0\in \Dz{\mu,A}_{\loc, r}$ for every~$A\in \A$ and~$r>0$.
Let~$\nu\sim \mssm$ be a probability measure on~$(X,\A)$ and set~$a\eqdef \sup \tset{\norm{f}_{L^1(\nu)} : f\in \Dz{\mu,A}_{\loc, r}} \leq r$.
By definition of~$a$, there exists a sequence~$\seq{g_k}_k\subset \Dz{\mu,A}_{\loc, r}$ so that~$\klim \nu g_k=a$.
Set~$f_n\eqdef \vee_{k\leq n} g_k$ and note that~$\seq{f_n}_n\subset \Dz{\mu,A}_{\loc, r}$ as well, by~\eqref{eq:TruncationLoc}.
Since~$\nu$ is a probability measure, up to choosing a suitable non-relabeled subsequence,~$\seq{f_n}_n$ converges to some~$f^{\sym r}$, $\nu$-, hence~$\mssm$-, a.e., and~$\nlim \nu f_n=\nu f^{\sym r}=a$ by Dominated Convergence in~$L^1(\nu)$ with dominating function~$r$.
Furthermore,~$f^{\sym r}=0$ $\mssm$-a.e.\ on~$A$, since the same holds true for~$f_n$ for every~$n$. Therefore, it follows by Lemma~\ref{l:ConvMeasure} that~$f^{\sym r}\in \Dz{\mu,A}_{\loc, r}$.
Additionally,~$f^\sym{r_2}\wedge r_1\equiv f^\sym{r_1}$ for every~$0<r_1<r_2$.
Maximality and uniqueness of~$f^{\sym r}$ are straightforward.
The existence of~$\hr{\mu,A}=\mssm\text{-a.e.-}\lim_{r\rar\infty} f^\sym{r}$ follows by consistency, as in the proof of~\cite[Prop.~3.11]{AriHin05}.
\end{proof}

We call the function~$\hr{\mu,A}$ constructed in Proposition~\ref{p:Hino} the \emph{maximal function} of~$A\in\A$. Note that~$\hr{\mu,A}$ is generally not an element of~$\dotloc{L^\infty(\mssm)}$.

\paragraph{Comparison results} Maximal functions should be compared with point-to-set distances induced by intrinsic distances. Before doing so in the next two lemmas, we note why this comparison is non-trivial.

\begin{remark}\label{r:IntrinsicVsHR}
In the case~$\mu=\mssm$, one might be tempted to identify~$\hr{\mssm, A}=\mssd_\mssm(\emparg, A)$.
We will show in Remark~\ref{r:dOpen} below that this identification does not hold for general~$A\in \A$.
Before discussing the details, let us note that ---~heuristically~--- this identification would require some kind of Minimax Theorem to hold. Indeed, by (the proof of) Proposition \ref{p:Hino}, 
\begin{align*}
\hr{\mssm,A}=\sup_{r>0}\mssm\text{-}\esssup_{x \in X} \sup\tset{f(x):\ f\in\Dz{\mu,A}_{\loc,r} }\comm
\end{align*}
whereas
\begin{align*}
\mssd_\mssm(\emparg, A)=\inf_{y\in A} \sup\set{f(y)-f(\emparg): f\in \DzLoc{\mu,\T}} \fstop
\end{align*}
\end{remark}

\begin{lemma}\label{l:StoLdA}
Let~$(\mbbX,\mcE)$ be a quasi-regular strongly local Dirichlet space,~$\mssd\colon X^\tym{2}\rar[0,\infty]$ be an extended pseudo-distance, and~$\mu\in\Msp(\Bo{\T},\Ne{\mcE})$ be $\mcE$-moderate. Further assume that:
\begin{enumerate}[$(a)$]
\item\label{i:l:StoLdA:1} $(\mbbX,\mcE,\mssd,\mu)$ possesses~\eqref{eq:Rad};
\item\label{i:l:StoLdA:2} $\mssd(\emparg, A)$ is $\A$-measurable for every~$A\in\A$.
\end{enumerate}
Then,
\begin{align*}
\mssd(\emparg, A)\leq \hr{\mu,A} \as{\mssm}\comm \qquad A\in \A \fstop
\end{align*}
\end{lemma}
\begin{proof} Since~$\mssd(\emparg, A)$ is $\mssd$-Lipschitz with~$\Li[\mssd]{\mssd(\emparg, A)}\leq 1$ for all~$A\subset X$, it holds that~$\mssd(\emparg, A)\wedge r\in \DzLocB{\mu}$ for all~$r>0$ by~\iref{i:l:StoLdA:2} and~$(\Rad{\mssd}{\mu})$. Since~$\mssd(\emparg, A)=0$ $\mssm$-a.e.\ on~$A$, then~$\mssd(\emparg, A)\wedge r \leq \hr{\mu,A}\wedge r$ $\mssm$-a.e.\ for every~$r>0$ by Proposition~\ref{p:Hino} and the maximality of~$\hr{\mu, A}\wedge r$.
The conclusion follows letting~$r\to\infty$.
\end{proof}

\blue{
Corollary~\ref{c:StoLdA2} below is a consequence of Lemma~\ref{l:StoLdA}. 
It is non-trivial only when~$(X,\mssd)$ is an extended metric space.
In this case, the statement conveys additional information on the compatibility between the Dirichlet space~$(\mbbX,\mcE)$ and the distance~$\mssd$; cf.\ Example~\ref{ese:Config1}.
Let us recall that~$(\mbbX,\mcE)$ is \emph{irreducible} if and only if every $\mcE$-invariant~$A\subset X$ satisfies either~$\mssm A=0$ or~$\mssm A^\complement=0$.

\begin{corollary}\label{c:StoLdA2}
Under the same assumptions as in Lemma~\ref{l:StoLdA}, let further~$\mu=\mssm$, and assume that the Dirichlet space~$(\mbbX,\mcE)$ be additionally irreducible.
Then,
\begin{align*}
A_1,A_2\in\A\comm \mssm A_1, \mssm A_2 >0 \implies \mssd(A_1,A_2)<\infty\fstop
\end{align*}

\begin{proof}
By irreducibility we have~$P_t(A_1,A_2)>0$ for all~$t>0$, hence~$\hr{\mssm}(A_1,A_2)<\infty$ by~\cite[Thm.~5.1]{AriHin05}, and the conclusion follows from Lemma~\ref{l:StoLdA}.
\end{proof}
\end{corollary}

\begin{remark}
We note that 
Corollary~\ref{c:StoLdA2} remains valid if assumption~\iref{i:l:StoLdA:2} in Lemma~\ref{l:StoLdA} is dropped, in which case the conclusion holds for all~$A_1,A_2\in\A$ so that~$\mssd(\emparg,A_i)$ is $\A$-measurable for either~$i=1,2$.
\end{remark}
}

\begin{lemma}\label{l:RaddA}
Let~$(\mbbX,\mcE)$ be a quasi-regular strongly local Dirichlet space,~$\mssd\colon X^\tym{2}\rar[0,\infty]$ be an extended pseudo-distance, and~$\mu\in\Msp(\Bo{\T},\Ne{\mcE})$ be $\mcE$-moderate.
Further assume that~$(\mbbX,\mcE,\mssd,\mu)$ possesses~$(\SL{\mu}{\mssd})$.
Then, for each~$A\in\A$ there exists~$\tilde A\in\A$ with~$\tilde A\subset A$ and so that~$A\triangle \tilde A$ is $\mssm$-negligible, and
\begin{align}\label{eq:l:RaddA:0}
\hr{\mu,A}=\hr{\mu,\tilde A}\leq \mssd(\emparg, \tilde A) \as{\mssm}
\end{align}
\end{lemma}
\begin{proof}
By~$(\SL{\mu}{\mssd})$, for every~$n\in \N$ there exists a $\A$-measurable $\mssd$-Lipschitz $\mssm$-represen\-ta\-tive~$\rep\rho_{A,n}$ of~$\hr{\mu,A}\wedge n$. Set~$\tilde A\eqdef \cap_n\,\rep\rho^{-1}_{A,n}(\set{0})\cap A\in \A$.
We have that~$A\triangle \tilde A$ is $\mssm$-negligible, and~$\rep\rho_{A,n}$ is identically vanishing everywhere on~$\tilde A$ for each~$n\in \N$. Thus,
\begin{align}
\nonumber
\rep\rho_{A,n}(x)\leq& \ \mssd(x,y)+\rep\rho_{A,n}(y) && x,y\in X \comm n\in \N\comm
\\
\label{eq:l:RaddA:1}
\rep\rho_{A,n}(x)\leq& \ \mssd(x,y) && x\in X\comm y\in \tilde A \comm n\in\N\fstop
\end{align}

Since~$n\mapsto \hr{\mu,A}\wedge n$ is monotone, there exists a set~$B\in\A$ of full $\mssm$-measure, and so that~$n\mapsto\rep\rho_{A,n}(x)$ is monotone for every~$x\in B$ and~$\seq{\rep\rho_{A,n}}_n$ is consistent in~$n$ on~$B$, in the sense that~$\rep\rho_{A,n}\equiv \rep\rho_{A,m}$ on the set~$B\cap\set{\rep\rho_{A,n}\leq m}$, for every~$m\leq n$.
Note that~$\tilde A\subset B$, since~$\rep\rho_{A,n}$ vanishes identically on~$\tilde A$ for each~$n\in\N$.
Therefore, taking the infimum over~$y\in \tilde A$ in~\eqref{eq:l:RaddA:1} and the limit superior in~$n$ to infinity,
\begin{align*}
\limsup_{n\rar\infty} \rep\rho_{A,n}(x)\leq& \inf_{y\in \tilde A}\mssd(x, y)\comm \qquad x\in X  \fstop
\end{align*}
Since~$B$ has full $\mssm$-measure,~$\rep\rho_A\eqdef \limsup_{n\rar\infty} \rep\rho_{A,n}$ is an $\mssm$-representative of~$\hr{\mu,A}$, and the conclusion follows.
\end{proof}

\begin{remark}[On the choice of~$\tilde A$]\label{r:dOpen}
\begin{enumerate*}[$(a)$]
\item\label{i:r:dOpen:1} Concerning the assertion of Lemma~\ref{l:RaddA}, one cannot replace~$\tilde A$ by~$A$ in~\eqref{eq:l:RaddA:0}, not even if~$A\in\Bo{\T}$.
This fact is most evident in the extreme case when~$A$ is both~$\mssm$-negligible and $\T_\mssd$-dense.
In this case,~$\hr{\mu,A}\equiv+\infty$, and therefore~$\rep\rho_{A,n}\equiv n$ everywhere on~$X$ for every~$n\in\N$. Thus,~$\tilde A=\emp$, and~\eqref{eq:l:RaddA:0} yields the (void) conclusion
\end{enumerate*}
\begin{align*}
+\infty\equiv \hr{\mu,A}\leq \inf_{y\in\emp}\mssd(\emparg, y)\equiv +\infty\fstop
\end{align*}
On the contrary,~$\mssd(\emparg, A)=\mssd(\emparg, \cl_{\mssd}A)=\mssd(\emparg,X)$ is identically vanishing.

\begin{enumerate*}[$(a)$]\setcounter{enumi}{1}
\item\label{i:r:dOpen:2} Under the assumptions of Lemma~\ref{l:RaddA}, suppose further that~\eqref{eq:ScL} holds. Then, we may choose~$\rep\rho_{A,n}$ to be additionally $\T$-continuous for every~$n$, and thus we may choose~$B=X$ in the proof. 
Furthermore, since~$\mssm$ has full $\T$-support, then~$C\eqdef \rep\rho_{A,1}^{-1}(\set{0})=\rep\rho_{A,n}^{-1}(\set{0})$ for all~$n$ by continuity.
Again by continuity of~$\rep\rho_{A,1}$, the set~$C$ is closed, and we have~$\mssm(A\triangle C)=0$ by definition of~$\rep\rho_{A,1}$.
Since~$\rep\rho_{A,1}\equiv 0$ $\mssm$-a.e.\ on~$\inter_\T A$, since both $\rep\rho_{A,1}$ and~$0$ are $\T$-continuous on~$\inter_\T A$, and since $\inter_\T A$ is open, then~$\rep\rho_{A,1}\equiv 0$ everywhere on~$\inter_\T A$ by Lemma~\ref{l:AeEquality}. Therefore~$C\supset \inter_\T A$. In fact, since~$C$ is $\T$-closed, then~$C\supset \cl_\T\inter_\T A$. As a consequence, we have~$\tilde A\supset A\cap \cl_\T\inter_\T A$.
Note that, even under~\eqref{eq:ScL}, it does not hold that~$C\supset A$, as it is readily seen by choosing~$A$ an $\mssm$-negligible singleton, so that~$C=\emp$.

\newline

\item\label{i:r:dOpen:2.5} Under the assumptions of Lemma~\ref{l:RaddA}, suppose further that~$\mssd$ is $\A^\otym{2}$-measurable, and that~$\mssm$ has full $\T_\mssd$-support. Since~$\SL{\mu}{\mssd}$ coincides with~$\ScL{\mu}{\T_\mssd}{\mssd}$, the same reasoning as in~\iref{i:r:dOpen:2} holds when replacing the $\T$-interior, resp.\ $\T$-closure with the~$\T_\mssd$-interior, resp.\ $\T_\mssd$-closure.
In particular, we may choose any~$\tilde A$ with~$\inter_\mssd A\subset \tilde A\subset \cl_\mssd \inter_\mssd A$.

\newline

\item\label{i:r:dOpen:3} Finally, since~$\mssd(\emparg, B)$ is increasing as~$B$ is decreasing, the assertion of Lemma~\ref{l:RaddA} remains true if we replace~$\tilde A$ by a smaller set. Thus, if additionally~\eqref{eq:ScL} holds, then we may always choose~$\tilde A=A\cap \cl_\T\inter_\T A$.
\end{enumerate*}
\end{remark}

Note that one is mostly interested in the case~$\mssd=\mssd_\mu$ (and, possibly,~$\mu=\mssm$), in which case the separability of~$\T_{\mssd_\mu}$ grants that the assumptions of Lemma~\ref{l:StoLdA} are satisfied, by Corollary~\ref{c:Separable}.

\begin{theorem}\label{t:Full}
Let~$(\mbbX,\mcE)$ be a quasi-regular strongly local Dirichlet space,~$\mssd\colon X^\tym{2}\rar[0,\infty]$ be an extended pseudo-distance, and~$\mu\in\Msp(\Bo{\T},\Ne{\mcE})$ be $\mcE$-moderate.
Further assume that:
\begin{enumerate}[$(a)$]
\item\label{i:t:Full:3} $(\mbbX,\mcE,\mssd,\mu)$ possesses~\eqref{eq:Rad} and~\eqref{eq:SL};

\item~$\mssd(\emparg, A)$ is $\A$-measurable for every~$A\in\A$, and~$\mssd$ is $\T$-admissible.
\end{enumerate}

Then, for each~$A\in\A$ there exists~$\tilde A\in\A$ with~$\tilde A\subset A$ and so that~$A\triangle \tilde A$ is $\mssm$-negligible, and
\begin{align}\label{eq:t:Full:0}
\hr{\mu,A}=\hr{\mu,\tilde A}=\mssd(\emparg, \tilde A)= \mssd_\mu(\emparg, \tilde A)
\quad \as{\mssm}
\end{align}
\end{theorem}
\begin{proof}
The property~\eqref{eq:dSL} follows by Proposition~\ref{p:StoL}. Since~$\mssd$ is $\T$-admissible
by assumption,~\eqref{eq:dRad} follows from~\eqref{eq:Rad} by Lemma~\ref{l:RadDRad}, which concludes that~$\mssd=\mssd_\mu$.

Furthermore, the assumptions of both Lemmas~\ref{l:StoLdA} and~\ref{l:RaddA} hold, and we have, for~$\tilde A$ as in Lemma~\ref{l:RaddA},
\begin{align*}
\hr{\mu,\tilde A}=&\ \mssd(\emparg, \tilde A)=\mssd_\mu(\emparg, \tilde A)\quad \as{\mssm}\comm
\end{align*}
which concludes the proof.
\end{proof}

In particular, we have the following.

\begin{corollary}\label{c:Special}
Let~$(\mbbX,\mcE)$ be a quasi-regular strongly local Dirichlet space,~$\mu\in\Msp(\Bo{\T},\Ne{\mcE})$ be $\mcE$-moderate. If~$\mssd_\mu$ is $\T^\tym{2}$-continuous and~$(\mbbX,\mcE,\mu)$ possesses~$(\SL{\mu}{\mssd_\mu})$, then for each~$A\in\A$ there exists~$\tilde A\in\A$ with~$\tilde A\subset A$ and so that~$A\triangle \tilde A$ is $\mssm$-negligible, and
\begin{align*}
\hr{\mu,A}=\hr{\mu,\tilde A}=\mssd_\mu(\emparg, \tilde A)\quad \as{\mssm}
\end{align*}

If additionally~$\mssm(A\setminus \inter_\T A)=0$ (in particular, if~$A$ is either $\T$-open, or a $\T$-continuity set for~$\mssm$, i.e.~$\mssm(\partial_\T A)=0$), then
\begin{align*}
\mssd_\mu(\emparg, \inter_\T A)=\hr{\mu,A}=\mssd_\mu(\emparg, \cl_\T \inter_\T A)\quad \as{\mssm}
\end{align*}
\end{corollary}
\begin{proof}
Since~$\mssd_\mu$ is $\T^\tym{2}$-continuous,~$\T_{\mssd_\mu}$ is separable. Combining Theorem~\ref{t:KuwaeProposition} with Theorem~\ref{t:Lenz} and Remark~\ref{r:Lenz}\iref{i:r:Lenz:3.0}, we have~$(\Rad{\mssd_\mu}{\mu})$, and the first assertion follows by Theorem~\ref{t:Full}. If~$\mssm(A\setminus \inter_\T A)=0$, then~$\hr{\mu,A}=\hr{\mu,\inter_\T A}$. The second assertion follows applying Theorem~\ref{t:Full} with~$\inter_\T A$ in place of~$A$, and thanks to Remark~\ref{r:dOpen}\iref{i:r:dOpen:3}, since we may choose~$\tilde A=\inter_\T A\cap \cl_\T\inter_\T A=\inter_\T A$.
\end{proof}

Whereas it is convenient to have the $\A$-measurability of~$\mssd(\emparg, \tilde A)$ for the same $\sigma$-algebra~$\A$ containing~$\tilde A$, this is not always possible in the applications.
In particular, this is the case for the distance on path groups, for which only the universal measurability of~$\mssd(\emparg, A)$ is known for every Borel set~$A$, see~\cite{AidZha02}, cf.~\cite[Thm.~1.4]{HinRam03}.
In order to partially address this issue, we state a slightly different version of Theorem~\ref{t:Full}, and subsequently provide another example in the case of the Wiener space.

\begin{theorem}\label{t:UniversallyM}
Let~$(\mbbX,\mcE)$ be a quasi-regular strongly local Dirichlet space,~$\mssd\colon X^\tym{2}\rar[0,\infty]$ be an extended pseudo-distance, and~$\mu\in\Msp(\Bo{\T},\Ne{\mcE})$ be $\mcE$-moderate.
Further assume that:
\begin{enumerate}[$(a)$]
\item\label{i:t:UniversallyM:1} $(\mbbX,\mcE,\mssd,\mu)$ possesses~\eqref{eq:Rad} for the $\sigma$-algebra $\Bo{\T}^\mssm$, and~\eqref{eq:SL} for the $\sigma$-algebra $\Bo{\T}^\mssm$;

\item\label{i:t:UniversallyM:2} $\mssd(\emparg, A)$ is $\Bo{\T}^\mssm$-measurable for every~$A\in\A$, and~$\mssd$ is $\T$-admissible.
\end{enumerate}
Then, the same conclusion holds as in Theorem~\ref{t:Full}.
\end{theorem}
\begin{proof}
As in the proof of Theorem~\ref{t:Full}, we may conclude that~$\mssd=\mssd_\mu$.
Lemma~\ref{l:RaddA}, which can be applied without condition~\iref{i:t:UniversallyM:2} here, yields the existence of~$\tilde A\in\A$ with~$\mssm(A\triangle \tilde A)=0$ and~$\hr{\mu,A}=\hr{\mu,\tilde A}\leq \mssd(\emparg, \tilde A)$ $\mssm$-a.e.
Thus, it suffices to show the opposite inequality for such~$\tilde A$.
To this end, note that, by~\iref{i:t:UniversallyM:2} there exists~$\rep f\colon X\rar [0,\infty]$, $\Bo{\T}$-measurable, and an $\mssm$-conegligible set~$Y\in\Bo{\T}$, so that~$\rep f=\mssd(\emparg, \tilde A)$ on~$Y$.
Thus, for all~$r>0$,
\begin{align*}
\rep f(x)\wedge r -\rep f(y) \wedge r=\mssd(x,\tilde A)\wedge r -\mssd(y,\tilde A)\wedge r\leq \mssd(x,y)\comm \qquad x,y\in Y\fstop
\end{align*}

By Lemma~\ref{l:McShane} applied to the bounded function~$\rep f\wedge r$, there exists a Lipschitz function~$\rep f_r$ satisfying~$\rep f_r=\rep f\wedge r$ on~$Y$. 
Since~$Y$ is $\mssm$-conegligible (and since the Borel $\sigma$-algebra on the real line is countably generated),~$\rep f_r$ is $\Bo{\T}^\mssm$-measurable for every~$r>0$.
Furthermore,~$\rep f_r\equiv 0$ $\mssm$-a.e.\ on~$A$.

By the Rademacher property~($\Rad{\mssd}{\mu}$) for the $\sigma$-algebra~$\Bo{\T}^\mssm$, the $\mssm$-class~$f_r$ of~$\rep f_r$ is an element of~$\Dz{\mu,A}_{\loc,r}$.
By maximality of~$\hr{\mu,\tilde A}\wedge r=\hr{\mu, A}\wedge r$ in~$\Dz{\mu,A}_{\loc,r}$ we have
\begin{align*}
f_r \leq \hr{\mu, A}\wedge r \qquad \as{\mssm}\comm
\end{align*}
and the conclusion follows letting~$r\rar \infty$.
\end{proof}

\begin{example}[Wiener spaces II]\label{ese:Wiener2}
In the same setting of Example~\ref{ese:Wiener1}, let~$A\in\Bo{\T}$. Then, there exists~$\tilde A\in\Bo{\T}$ so that~$A\triangle \tilde A$ is $\mssm$-negligible and
\begin{align*}
\hr{\mssm,\tilde A}=\mssd_H(\emparg, \tilde A) \qquad \as{\mssm}
\end{align*}
\end{example}
\begin{proof}
By a straightforward adaptation of~\cite[Prop.~4.13]{HinRam03}, the function~$\mssd_H(\emparg, \tilde A)$ is universally measurable, and thus in particular $\Bo{\T}^\mssm$-measurable.
The admissibility of~$\mssd_H$ is shown in Example~\ref{ese:Wiener1}.
The Rademacher property~($\Rad{\mssd_H}{\mssm}$) shown in~\cite[Thm., p.~27]{EncStr93} holds, with identical proof, for every $\Bo{\T}^\mssm$-measurable Lipschitz function.
The Sobolev-to-Lipschitz property~$(\SL{\mssm}{\mssd_H})$ for the completed Borel $\sigma$-algebra~$\Bo{\T}^\mssm$ is implied by that for the Borel $\sigma$-algebra~$\Bo{\T}$ (see Example~\ref{ese:Wiener1}). 
The conclusion now follows by Theorem~\ref{t:UniversallyM}.
\end{proof}

\paragraph{Varadhan-type short-time asymptotics} For a Dirichlet space~$(\mbbX,\mcE)$, we let~$T_\bullet\eqdef\seq{T_t}_{t>0}$ be the corresponding Markov semigroup, with heat kernel measure
\begin{align*}
\mssp_t(\emparg, A)\colon X\rar [0,1]\comm \qquad A\in \A \fstop
\end{align*}
We further define the heat kernel bi-measure
\begin{align}\label{eq:PtAB}
P_t(A_1,A_2)\eqdef \int_{A_2} T_t\car_{A_1} \diff\mssm = \int_{A_1} T_t \car_{A_2} \diff\mssm \comm \qquad A_1,A_2\in \A\comm \mssm A_1,\mssm A_2>0 \fstop
\end{align}

The maximal functions defined in Proposition~\ref{p:Hino} have appeared in~\cite{HinRam03,AriHin05} as a key tool in the study of the short-time asymptotics for the Markov semigroup~$T_\bullet$.
In the same setting of Proposition~\ref{p:Hino}, set
\begin{align*}
\bar{\mssd}_\mssm(A_1,A_2)\eqdef \mssm\text{-}\essinf_{x\in A_1} \hr{\mssm,A_2}(x)\comm \qquad A_1,A_2\in \A \fstop
\end{align*}
As a consequence of Remark~\ref{r:AriHino}, we may specialize results in~\cite{AriHin05} to our more restrictive (topological) setting. The next result is a particular case of~\cite[Thm.~2.7, Prop.~3.11]{AriHin05}.

\begin{theorem}[Ariyoshi--Hino~{\cite{AriHin05}}]\label{t:AriHino}
Let~$(\mbbX,\mcE)$ be a quasi-regular strongly local Dirichlet space. Then,
\begin{align}\label{eq:Varadhan}
\lim_{t\rar 0} \ttonde{-2t \log P_t(A_1,A_2)}=\bar{\mssd}_\mssm(A_1,A_2)^2\comm \qquad A_1,A_2\in\A\comm 0<\mssm A_1,\mssm A_2<\infty \fstop
\end{align}
\end{theorem}

In particular, since the left-hand side of~\eqref{eq:Varadhan} is symmetric, the function~$\hr{\mssm}(\emparg,\emparg)$ is symmetric as well.

Among several examples of diffusion processes satisfying~\eqref{eq:Varadhan} are those for which~$\hr{\mssm,A}$ can be precisely identified.
It is therefore natural to ask for sufficient conditions allowing to identify the maximal function~$\hr{\mssm, A}$ with a given point-to-set distance function~$\mssd(\emparg, A)$, as above.
In particular, combining~\eqref{eq:LipDzLoc} with Theorems~\ref{t:Full} and~\ref{t:AriHino}, we have the following.
\begin{corollary}\label{c:AriyoshiHino}
Let all the assumptions of either Theorem~\ref{t:Full}, or Theorem~\ref{t:UniversallyM} be satisfied with~$\mu=\mssm$. Then, for~$i=1,2$ and~$A_i\in\A$ with~$\mssm A_i>0$, there exists~$\tilde A_i\in\A$ with~$\tilde A_i\subset A_i$ and so that~$A_i\triangle \tilde A_i$ is $\mssm$-negligible, and
\begin{align*}
\lim_{t\rar 0} \ttonde{-2t \log P_t(A_1,A_2)}=\lim_{t\rar 0} \ttonde{-2t \log P_t(\tilde A_1,\tilde A_2)}=\tonde{\mssm\textrm{-}\essinf_{y\in A_j}\mssd_\mssm(y,\tilde A_i)}^2 \fstop
\end{align*}
with~$i\neq j$.

Assume further that $(\mbbX,\mcE)$ is additionally strictly local. 
Then, if either~$A_i$ satisfies $\mssm(A_i\setminus \inter_\T A_i)=0$,
\begin{align*}
\lim_{t\rar 0} \ttonde{-2t \log P_t(A_1,A_2)}=\tonde{\mssm\textrm{-}\essinf_{y\in A_j}\mssd_\mssm(y,\tilde A_i)}^2 \comm
\end{align*}
where~$i\neq j$ and~$\tilde A_i$ is any set so that~$\inter_\T A_i\subset \tilde A_i\subset \cl_\T\inter_\T A_i$.
\end{corollary}

Corollary~\ref{c:AriyoshiHino} applies to various classes of spaces discussed in the previous sections (see e.g.\ Examples~\ref{ese:MMS},~\ref{ese:Wiener1},~\ref{ese:RCD}).
Here, we only single out the case of $\RCD$ spaces.

\begin{example}[$\RCD(K,\infty)$ spaces]
Let~$(X,\mssd,\mssm)$ be an $\RCD(K,\infty)$ space, e.g.~\cite{AmbGigSav14b,AmbGigMonRaj12}, endowed with its Cheeger energy~$\Ch[\mssd,\mssm]$.
Here~$\mssd$ is a distance on~$X$, and the topology~$\T$ is induced by~$\mssd$.
As discussed in Examples~\ref{ese:MMS} and~\ref{ese:RCD}, the Dirichlet space $(\mbbX,\Ch[\mssd,\mssm])$ possesses both~$(\Rad{\mssd}{\mssm})$ and~$(\ScL{\mssm}{\T_\mssd}{\mssd})$, it is strictly local by definition, and~$\mssd=\mssd_\mssm$.
Furthermore, since~$\mssd$ is continuous, $\mssd_A$ is as well, and therefore, since~$\mssm$ has full support, $\mssm$-$\essinf_U \mssd_A=\inf_U \mssd_A$ for every open set~$U$.
Thus finally, Corollary~\ref{c:AriyoshiHino} applies, and we obtain
\begin{align*}
\lim_{t\to 0} \ttonde{-2t \log P_t(A_1,A_2)} = \mssd(\inter_\mssd A_1,\inter_\mssd A_2) \comm \qquad A_i\in\Bo{\mssd}\comm  \mssm A_i>0\comm \mssm (A_i\setminus \inter_\mssd A_i)=0 \fstop
\end{align*}
\end{example}

\vspace{1cm}

\hfill

\begin{tabular}{c l}
& \textbf{List of notations}
\\
$\triangle$ & symmetric difference of sets
\\
$\ll$ & absolute continuity of measures
\\
$\mrestr$ & restriction of measures
\\
$*$ \emph{as in} $j^*f$ & pull-back of a function~$f$ via a map~$j$
\\
$\pfwd$ \emph{as in} $j_\pfwd \mu$ & push-forward of a measure~$\mu$ via a map~$j$
\\
$\msA$ & any space of functions
\\
$\msA_b$ & space of bounded functions in~$\msA$
\\
$\dotloc{\msA}$ & broad local space of~$\msA$, \S\ref{ss:BroadLoc}
\\
$\Bo{\T}$ & Borel $\sigma$-algebra induced by a topology~$\T$
\\
$B^\mssd_r(x)$ & $\mssd$-ball of radius~$r\in [0,\infty]$ and center~$x$
\\
$\mcC$ & space of continuous functions
\\
$\mssd$ & extended pseudo-distance
\\
$\mssd_\mu$ & intrinsic (extended pseudo-)distance,~\eqref{eq:IntrinsicD}
\\
$\hr{\mu,A}$ & $\mu$-maximal function of a set~$A$, Prop.~\ref{p:Hino}
\\
$(\mcE,\dom)$ & Dirichlet form with domain~$\dom$
\\
$(\mcE^\mu,\dom^\mu)$ & $\mu$-perturbation of~$(\mcE,\dom)$ by a smooth measure~$\mu$,~\eqref{eq:Perturbed}
\\
$\domdom$ & space of functions with minimal dominant energy measure
\\
$\domext$ & extended domain of a Dirichlet form~$(\mcE,\dom)$
\\
$\domloc$ & local domain of a Dirichlet form~$(\mcE,\dom)$
\\
$f$ & class of a measurable function up to a.e.-equivalence
\\
$\rep f$ & representative of the class~$f$
\\
$\reptwo f$ & quasi-continuous representative of the class~$f$
\\
$\msG$, $\msG_0$, $\msG_c$ & families of quasi-open nests, \S\ref{ss:BroadLoc}
\\
$\DzLoc{\mu}$, $\DzLoc{\mu,\T}$ & broad local spaces of functions with $\mu$-bounded energy, \S\ref{sss:LocDom}
\\
$\Dz{\mu}$, $\Dz{\mu,\T}$ & spaces of functions with $\mu$-bounded energy, \S\ref{sss:LocDom}
\\
$\mcL^p$ & space of $p$-integrable functions,~$p\in [0,\infty]$
\\
$L^p$ & space of classes of $p$-integrable functions up to a.e.-equivalence
\\
$\Li[\mssd]{\rep f}$ & global Lipschitz constant of a function~$\rep f$ with respect to~$\mssd$
\\
$\Lip(\mssd,\A)$ & space of $\mssd$-Lipschitz $\A$-measurable functions
\\
$\Lip(\mssd,\T)$ & space of $\mssd$-Lipschitz $\T$-continuous functions
\\
$\Mbp(\A)$ & non-negative bounded measures on a $\sigma$-algebra~$\A$ 
\\
$\Ms(\A)$ & extended signed $\sigma$-finite measures on a $\sigma$-algebra~$\A$
\\
$\Msp(\A)$ & $\sigma$-finite measures on a $\sigma$-algebra~$\A$
\\
$\mcM_0$, $\mcM$ & absolutely moderate, resp.\ moderate, measures, Dfn.~\ref{d:Moderance}
\\
$\mssm$ & reference measure
\\
$\mcS_0$, $\mcS$ & measures of finite energy integral, resp.\ smooth measures, Dfn.~\ref{d:Smooth}
\\
$\A$ & $\sigma$-algebra
\\
$\A^\mu$ & completion of~$\A$ with respect to\ a measure~$\mu\colon \A\rar [0,\infty]$
\\
$\T$ & topology
\end{tabular}




\newpage

{
\small
\begin{thebibliography}{10}

\bibitem{AidKaw01}
{Aida, S.} and {Kawabi, H.}
\newblock {Short Time Asymptotics of a Certain Infinite Dimensional Diffusion
  Process}.
\newblock In {Decreusefond, L.}, {{\O}ksendal, B.~K.}, and
  {{\"{U}}st{\"{u}}nel, A.~S.}, editors, {\em {Stochastic Analysis and Related
  Topics VII -- Proceedings of the Seventh Silivri Workshop}}, volume~48 of
  {\em {Progress in Probability}}, pages 77--124. {Birkh{\"{a}}user}, 2001.

\bibitem{AidZha02}
{Aida, S.} and {Zhang, T.~S.}
\newblock {On the Small Time Asymptotics of Diffusion Processes on Path
  Groups}.
\newblock {\em {Potential Anal.}}, 16:67--78, 2002.

\bibitem{AlbBraRoe89}
{Albeverio, S.}, {Brasche, J.}, and {R{\"{o}}ckner, M.}
\newblock {Dirichlet forms and generalized Schr{\"{o}}dinger operators}.
\newblock In {Holden, H.} and {Jensen, A.}, editors, {\em {Schr{\"{o}}dinger
  Operators -- Proceedings of the Nordic Summer School in Mathematics --
  Sandbjerg Slot, S{\o}nderborg, Denmark, August 1-12, 1988}}, volume 345 of
  {\em {Lecture Notes in Physics}}, pages 1--42. {Springer-Verlag}, 1989.

\bibitem{AlbKonRoe98}
{Albeverio, S.}, {Kondratiev, Yu. G.}, and {R{\"{o}}ckner, M.}
\newblock {Analysis and Geometry on Configuration Spaces}.
\newblock {\em {J.\ Funct.\ Anal.}}, 154(2):444--500, 1998.

\bibitem{AlbMa91}
{Albeverio, S.} and {Ma, Z.-M.}
\newblock {Diffusion Processes with singular Dirichlet forms}.
\newblock In {Cruzeiro, A.~B.} and {Zambrini, J.~C.}, editors, {\em {Stochastic
  Analysis and Applications -- Proceedings of the 1989 Lisbon Conference}},
  volume~26 of {\em {Progress in Probability}}, pages 11--28.
  {Springer-Verlag}, 1991.

\bibitem{AmbErbSav16}
{Ambrosio, L.}, {Erbar, M.}, and {Savar\'e, G.}
\newblock {Optimal transport, Cheeger energies and contractivity of dynamic
  transport distances in extended spaces}.
\newblock {\em {Nonlinear Anal.}}, 137:77--134, 2016.

\bibitem{AmbGigSav14}
{Ambrosio, L.}, {Gigli, N.}, and {Savar\'e, G.}
\newblock {Calculus and heat flow in metric measure spaces and applications to
  spaces with Ricci bounds from below}.
\newblock {\em {Invent.\ Math.}}, 395:289--391, 2014.

\bibitem{AmbGigSav14b}
{Ambrosio, L.}, {Gigli, N.}, and {Savar{\'{e}}, G.}
\newblock {Metric measure spaces with Riemannian Ricci curvature bounded from
  below}.
\newblock {\em {Duke Math.~J.}}, 163(7):1405--1490, 2014.

\bibitem{AmbGigSav15}
{Ambrosio, L.}, {Gigli, N.}, and {Savar{\'{e}}, G.}
\newblock {Bakry--{\'{E}}mery Curvature-Dimension Condition and Riemannian
  Ricci Curvature Bounds}.
\newblock {\em {Ann.~Probab.}}, 43(1):339--404, 2015.

\bibitem{AmbSavZam09}
{Ambrosio, L.}, {Savar{\'{e}}, G.}, and {Zambotti, L.}
\newblock {Existence and stability for Fokker--Planck equations with
  log-concave reference measure}.
\newblock {\em {Probab.\ Theory Relat.\ Fields}}, 145:517--564, 2009.

\bibitem{AmbGigMonRaj12}
{Ambrosio, L.}, {Gigli, N.}, {Mondino, A.}, and {Rajala, T.}.
\newblock {Riemannian Ricci curvature lower bounds in metric measure spaces
  with $\sigma$-finite measure}.
\newblock {\em {Trans.\ Amer.\ Math.\ Soc.}}, 367:4661--4701, 2015.

\bibitem{AriHin05}
{Ariyoshi, T.} and {Hino, M.}
\newblock {Small-time Asymptotic Estimates in Local Dirichlet Spaces}.
\newblock {\em {Electron.\ J.\ Probab.}}, 10(37):1236--1259, 2005.

\bibitem{BaiNor18}
{Bailleul, I.} and {Norris, J.~R.}
\newblock {Diffusion in small time in incomplete sub-Riemannian manifolds}.
\newblock {\em arXiv:1810.06328}, 2018.

\bibitem{BirMos95}
{Biroli, M.} and {Mosco, U.}
\newblock {A Saint-Venant type principle for Dirichlet forms on discontinuous
  media}.
\newblock {\em {Ann.\ Mat.\ Pura Appl.}}, 169:125--181, 1995.

\bibitem{Bog07}
{Bogachev, V.~I.}
\newblock {\em {Measure Theory}}.
\newblock {Springer-Verlag}, {Berlin}, {2007}.

\bibitem{BogMay96}
{Bogachev, V.~I.} and {Mayer-Wolf, E.}
\newblock {Some Remarks on Rademacher's Theorem in Infinite Dimensions}.
\newblock {\em {Potential Anal.}}, 5:23--30, 1996.

\bibitem{BouHir91}
{Bouleau, N.} and {Hirsch, F.}
\newblock {\em {Dirichlet forms and analysis on Wiener space}}.
\newblock {De Gruyter}, 1991.

\bibitem{Bou69}
{Bourbaki, N.}
\newblock {\em {Int{\'{e}}gration}}, volume {XXXV -- Livre VI} of {\em
  {{\'{E}}l{\'{e}}ments de math{\'{e}}matiques}}.
\newblock {C.C.L.S.}, 1969.

\bibitem{Bou07}
{Bourbaki, N.}
\newblock {\em {Int{\'{e}}gration: Chapitres 1 {\`{a}} 4}}.
\newblock {Springer-Verlag}, {Reprint of the 1965 Edition} edition, 2007.

\bibitem{CarKusStr87}
{Carlen, E.~A.}, {Kusuoka, S.}, and {Stroock, D. W.}
\newblock {Upper Bounds for symmetric Markov transition functions}.
\newblock {\em {Ann.~I.~H.~Poincar{\'{e}}~B}}, 23(S2):245--287, 1987.

\bibitem{Che99}
{Cheeger, J.}
\newblock {Differentiability of Lipschitz Functions on Metric Measure Spaces}.
\newblock {\em {Geom.\ Funct.\ Anal.}}, 9:428--517, 1999.

\bibitem{CheMaRoe94}
{Chen, Z.-Q.}, {Ma, Z.-M.}, and {R\"ockner, M.}
\newblock {Quasi-homeomorphisms of Dirichlet forms}.
\newblock {\em {Nagoya Math. J.}}, 136:1--15, 1994.

\bibitem{CicMorRalRyl07}
{Cicho{\'{n}}, J.}, {Morayne, M.}, {Ryll-Nardzewski, C.}, and {{\.{Z}}eberski,
  S.}
\newblock {On nonmeasurable unions}.
\newblock {\em {Topol.\ Appl.}}, 154:884--893, 2007.

\bibitem{CreSou19}
{Creuz, P.} and {Soultanis, E.}
\newblock {Maximal Metric Surfaces and the Sobolev-to-Lipschitz Property}.
\newblock {\em {arXiv:1909.10385v2}}, 2019.

\bibitem{DeCPal91}
{De Cecco, G.} and {Palmieri, G.}
\newblock {Integral distance on a Lipschitz Riemannian manifold}.
\newblock {\em {Math.~Z.}}, 207:223--243, 1991.

\bibitem{LzDS17+}
{Dello Schiavo, L.}
\newblock {The Dirichlet--Ferguson Diffusion on the Space of Probability
  Measures over a Closed Riemannian Manifold}.
\newblock {\em {arXiv:1811.11598}}, 2018.

\bibitem{LzDS19b}
{Dello Schiavo, L.}
\newblock {A Rademacher-type theorem on $L^2$-Wasserstein spaces over closed
  Riemannian manifolds}.
\newblock {\em {J.\ Funct.\ Anal.}}, 278:108397, 2019.

\bibitem{DonMa93}
{Dong, Z.} and {Ma, Z.-M.}
\newblock {An integral representation theorem for quasi-regular Dirichlet
  spaces}.
\newblock {\em {Chinese Sci.\ Bull.\ (Chinese edition)}}, 38(15):1355--1358,
  1993.
\newblock (in Chinese).

\bibitem{Ebe96}
{Eberle, A.}
\newblock {Girsanov-type transformations of local Dirichlet forms: An analytic
  approach}.
\newblock {\em {Osaka J.~Math.}}, 33(2):497--531, 1996.

\bibitem{EncStr93}
{Enchev, O.} and {Stroock, D. W.}
\newblock {Rademacher's theorem for Wiener functionals}.
\newblock {\em {Ann. Probab.}}, 21(1):25--33, 1993.

\bibitem{Eng89}
{Engelking, R.}
\newblock {\em {General Topology}}, volume~6 of {\em {Sigma series in pure
  mathematics}}.
\newblock {Heldermann}, {Berlin}, {1989}.

\bibitem{Fan94}
{Fang, S.}
\newblock {On the Ornstein--Uhlenbeck Process}.
\newblock {\em {Stoch.\ Stoch.\ Rep.}}, 46:141--159, 1994.

\bibitem{FanZha99}
{Fang, S.} and {Zhang, T.~S.}
\newblock {On the small time behavior of Ornstein--Uhlenbeck processes with
  unbounded linear drifts}.
\newblock {\em {Probab.\ Theory Relat.\ Fields}}, 114:487--504, 1999.

\bibitem{FraLenWin14}
{Frank, R. L.}, {Lenz, D.}, and {Wingert, D.}
\newblock {Intrinsic metrics for non-local symmetric Dirichlet forms and
  applications to spectral theory}.
\newblock {\em {J. Funct. Anal.}}, 266(8):4765--4808, 2014.

\bibitem{Fug71}
{Fuglede, B.}
\newblock {The quasi topology associated with a countably subadditive set
  function}.
\newblock {\em {Ann.\ I.\ Fourier}}, 21(1):123--169, 1971.

\bibitem{FukOshTak11}
{Fukushima, M.}, {Oshima, Y.}, and {Takeda, M.}
\newblock {\em {Dirichlet forms and symmetric Markov processes}}, volume~19 of
  {\em {De Gruyter Studies in Mathematics}}.
\newblock {de Gruyter}, extended edition, 2011.

\bibitem{GarRhoVar14}
{Garban, C.}, {Rhodes, R.}, and {Vargas, V.}
\newblock {On the heat kernel and the Dirichlet form of Liouville Brownian
  motion}.
\newblock {\em {Electron.\ J.\ Probab.}}, 19(96):1--25, 2014.

\bibitem{Gig12a}
{Gigli, N.}
\newblock {On the differential structure of metric measure spaces and
  applications}.
\newblock {\em {Memoirs AMS}}, {236}({1113}), 2015.

\bibitem{Gol87}
{Goldstein, S.}
\newblock {Random Walks and Diffusions on Fractals}.
\newblock In {Kesten, H.}, editor, {\em {Percolation Theory and Ergodic Theory
  of Infinite Particle Systems}}, volume~8 of {\em {The IMA Volumes in
  Mathematics and Its Applications}}, pages 121--129. {Springer}, 1987.

\bibitem{HeiKosShaTys15}
{Heinonen, J.}, {Koskela, P.}, {Shanmugalingam, N.}, and {Tyson, J.~T.}
\newblock {\em Sobolev Spaces on Metric Measure Spaces -- An Approach Based on
  Upper Gradients}, volume~27 of {\em {New Mathematical Monographs}}.
\newblock {Cambridge University Press}, 2015.

\bibitem{Hin09}
{Hino, M.}
\newblock {Energy measures and indices of Dirichlet forms, with applications to
  derivatives on some fractals}.
\newblock {\em {Proc. London Math. Soc.}}, 100:269--302, 2009.

\bibitem{HinRam03}
{Hino, M.} and {Ram{\'{i}}rez, J. A.}
\newblock {Small-Time Gaussian Behavior of Symmetric Diffusion Semigroups}.
\newblock {\em {Ann. Probab.}}, 31(3):1254--1295, 2003.

\bibitem{HinRoeTep13}
{Hinz, M.}, {R{\"{o}}ckner, M.}, and {Teplyaev, A.}
\newblock {Vector analysis for Dirichlet forms and quasilinear PDE and SPDE on
  metric measure spaces}.
\newblock {\em {Stoch.\ Proc.\ Appl.}}, 123:4373--4406, 2013.

\bibitem{Kaj11}
{Kajino, N.}
\newblock {Heat Kernel Asymptotics for the Measurable Riemannian Structure on
  the Sierpinski Gasket}.
\newblock {\em {Potential Anal.}}, 36:67--115, 2012.

\bibitem{KonvRe18}
{Konarovskyi, V. V.} and {Renesse, M.-K. von}.
\newblock {Modified Massive Arratia flow and Wasserstein diffusion}.
\newblock {\em {Comm. Pur. Appl. Math.}}, 2018.
\newblock (to appear).

\bibitem{KosZho12}
{Koskela, P.} and {Zhou, Y.}
\newblock {Geometry and analysis of Dirichlet forms}.
\newblock {\em {Adv. Math.}}, 231:2755--2801, 2012.

\bibitem{Kum97}
{Kumagai, T.}
\newblock {Short time asymptotic behaviour and large deviation for Brownian
  motion on some affine nested fractals}.
\newblock {\em {Publ.~RIMS, Kyoto Univ.}}, 33:223--240, 1997.

\bibitem{Kus89}
{Kusuoka, S.}
\newblock {Dirichlet Forms on Fractals and Products of Random Matrices}.
\newblock {\em {Publ.\ Res.\ I.\ Math.\ Sci., Kyoto Univ.}}, 25:659--680, 1989.

\bibitem{Kuw96}
{Kuwae, K.}
\newblock {On pseudo metrics of Dirichlet forms on separable metric spaces}.
\newblock In {Watanabe, S.}, {Fukushima, M.}, {Prohorov, Yu.~V.}, and
  {Shiryaev, A.~N.}, editors, {\em {Probability Theory and Mathematical
  Statistics -- Proceedings of the Seventh Japan--Russia Symposium: Tokyo 26-30
  July 1995}}, pages 256--265. {World Scientific}, 1996.

\bibitem{Kuw98}
{Kuwae, K.}
\newblock {Functional Calculus for Dirichlet Forms}.
\newblock {\em {Osaka J.~Math.}}, 35:683--715, 1998.

\bibitem{KuwMacShi01}
{Kuwae, K.}, {Machigashira, Y.}, and {Shioya, T.}
\newblock {Sobolev spaces, Laplacian, and heat kernel on Alexandrov spaces}.
\newblock {\em {Math.~Z.}}, 238:269--316, 2001.

\bibitem{Lea87a}
{L{\'{e}}andre, R.}
\newblock {Majoration en temps petit de la densit{\'{e}} d'une diffusion
  d{\'{e}}g{\'{e}}n{\'{e}}r{\'{e}}e}.
\newblock {\em {Probab.\ Theory Relat.\ Fields}}, 74(2):289--294, 1987.

\bibitem{Lea87b}
{L{\'{e}}andre, R.}
\newblock {Minoration en temps petit de la densit{\'{e}} d'une diffusion
  d{\'{e}}g{\'{e}}n{\'{e}}r{\'{e}}e}.
\newblock {\em {J.\ Funct.\ Anal.}}, 74(2):399--414, oct 1987.

\bibitem{MaRoe92}
{Ma, Z.-M.} and {R\"ockner, M.}
\newblock {\em Introduction to the Theory of (Non-Symmetric) Dirichlet Forms}.
\newblock {Graduate Studies in Mathematics}. Springer, 1992.

\bibitem{McS34}
{McShane, E. J.}
\newblock {Extension of Range of Functions}.
\newblock {\em {Bull. Amer. Math. Soc.}}, 40(12):837--842, 1934.

\bibitem{Nor97}
{Norris, J.~R.}
\newblock {Heat kernel asymptotics and the distance function in Lipschitz
  Riemannian manifolds}.
\newblock {\em {Acta Math.}}, 179:79--103, 1997.

\bibitem{Pac13}
{Pachl, J.}
\newblock {\em {Uniform Spaces and Measures}}, volume~30 of {\em {Fields
  Institute Monographs}}.
\newblock {Springer}, 2013.

\bibitem{Pet06}
{Petersen, Peter}.
\newblock {\em {Riemannian Geometry}}, volume {171} of {\em {Graduate Texts in
  Mathematics}}.
\newblock {Springer-Verlag New York}, {2006}.

\bibitem{RenRoe05}
{Ren, J.} and {R{\"{o}}ckner, M.}
\newblock {A Remark on Sets in Infinite Dimensional Spaces with Full or Zero
  Capacity}.
\newblock In {Hida, T.}, editor, {\em {Stochastic Analysis: Classical and
  Quantum: Perspectives of White Noise Theory --- Meijo University, Nagoya,
  Japan, 1--5 November 2004}}, pages 177--186. {Wiley}, 2005.

\bibitem{vReStu09}
{Renesse, M.-K. von} and {Sturm, K.-T.}
\newblock {Entropic measure and Wasserstein diffusion}.
\newblock {\em {Ann.\ Probab.}}, 37(3):1114--1191, 2009.

\bibitem{RoeSch99}
{R{\"o}ckner, M.} and {Schied, A.}
\newblock {Rademacher's Theorem on Configuration Spaces and Applications}.
\newblock {\em {J.\ Funct.\ Anal.}}, 169(2):325--356, 1999.

\bibitem{RoeSch95}
{R{\"{o}}ckner, M.} and {Schmuland, B.}
\newblock {Quasi-regular Dirichlet forms: Examples and Counterexamples}.
\newblock {\em {Can.\ J.\ Math.}}, 47(1):165--200, 1995.

\bibitem{Sav14}
{Savar{\'{e}}, G.}
\newblock {Self-Improvement of the Bakry--{\'{E}}mery Condition and Wasserstein
  Contraction of the Heat Flow in $\mathrm{RCD}(K,\infty)$ Metric Measure
  Spaces}.
\newblock {\em {Discr.\ Cont.\ Dyn.\ Syst.}}, 34(4):1641--1661, 2014.

\bibitem{Sch73}
{Schwartz, L.}
\newblock {\em {Radon measures on arbitrary topological spaces and cylindrical
  measures}}.
\newblock Oxford University Press, 1973.

\bibitem{Sto10}
{Stollmann, P.}
\newblock {A dual characterization of length spaces with application to
  Dirichlet metric spaces}.
\newblock {\em {Studia Math.}}, 198(3):221--233, 2010.

\bibitem{Stu94}
{Sturm, K.-T.}
\newblock {Analysis on local Dirichlet spaces I. Recurrence, conservativeness
  and $L^p$-Liouville properties}.
\newblock {\em {J. reine angew. Math.}}, 456:173--196, 1994.

\bibitem{Stu95}
{Sturm, K.-T.}
\newblock {Sharp estimates for capacities and applications to symmetric
  diffusions}.
\newblock {\em {Probab. Theory Relat. Fields}}, 103:73--89, 1995.

\bibitem{Stu97}
{Sturm, K.-T.}
\newblock {Is a Diffusion Process Determined by Its Intrinsic Metric?}
\newblock {\em {Chaos, Solitons \& Fractals}}, 8(11):1855--1866, 1997.

\bibitem{tElRobSikZhu06}
{ter Elst, A. F. M.}, {Robinson, D. W.}, {Sikora, A.}, and {Zhu, Y.}
\newblock {Dirichlet Forms and Degenerate Elliptic Operators}.
\newblock In {Koelink, E.}, {van Neerven, J.}, {de Pagter, B.}, {Sweers, G.},
  {Luger, A.}, and {Woracek, H.}, editors, {\em {Partial Differential Equations
  and Functional Analysis: The Philippe Cl{\'e}ment Festschrift}}, pages
  73--95. {Birkh{\"a}user}, 2006.

\bibitem{Var67}
{Varadhan, S.~R.~S.}
\newblock {On the Behavior of the Fundamental Solution of the Heat Equation
  with Variable Coefficients}.
\newblock {\em {Comm.\ Pure Appl.\ Math.}}, 20:431--455, 1967.

\bibitem{Zha00}
{Zhang, T.~S.}
\newblock {On the small time asymptotics of diffusion processes on Hilbert
  spaces}.
\newblock {\em {Ann.\ Probab.}}, 28(2):537--557, 2000.

\end{thebibliography}
}
\end{document}